\documentclass[a4paper]{article}%
\usepackage{amsmath}
\usepackage{amsfonts}
\usepackage{amssymb}
\usepackage{graphicx}
\usepackage{hyperref}%
\providecommand{\U}[1]{\protect\rule{.1in}{.1in}}
\providecommand{\U}[1]{\protect\rule{.1in}{.1in}}
\numberwithin{equation}{section}
\newtheorem{theorem}{Theorem}[section]

\newtheorem{corollary}[theorem]{Corollary}

\newtheorem{definition}[theorem]{Definition}
\newtheorem{example}[theorem]{Example}

\newtheorem{lemma}[theorem]{Lemma}

\newtheorem{proposition}[theorem]{Proposition}
\newtheorem{remark}[theorem]{Remark}

\newenvironment{proof}[1][Proof]{\noindent\textbf{#1.} }{\ \rule{0.5em}{0.5em}}
\begin{document}

\title{Fundamental solutions and local solvability \\for nonsmooth H\"{o}rmander's operators\thanks{\textbf{2000
AMS\ Classification}: 35A08, 35A17, 35H20. \textbf{Keywords}: nonsmooth
H\"{o}rmander's vector fields, fundamental solution, solvability, H\"{o}lder
estimates.}}
\author{Marco Bramanti, Luca Brandolini,
\and Maria Manfredini, Marco Pedroni}
\maketitle

\begin{abstract}
We consider operators of the form $L=\sum_{i=1}^{n}X_{i}^{2}+X_{0}$ in a
bounded domain of $\mathbb{R}^{p}$ where $X_{0},X_{1},\ldots,X_{n}$ are
\textit{nonsmooth} H\"{o}rmander's vector fields of step $r$ such that the
highest order commutators are only H\"{o}lder continuous. Applying Levi's
parametrix method we construct a local fundamental solution $\gamma$ for $L$
and provide growth estimates for $\gamma$ and its first derivatives with
respect to the vector fields. Requiring the existence of one more derivative
of the coefficients we prove that $\gamma$ also possesses second derivatives,
and we deduce the local solvability of $L$, constructing, by means of $\gamma
$, a solution to $Lu=f$ with H\"{o}lder continuous $f$. We also prove
$C_{X,loc}^{2,\alpha}$ estimates on this solution.

\end{abstract}
\tableofcontents

\pagebreak\bigskip

\section{Introduction}

\textbf{Object and main results of the paper}

\noindent In the study of elliptic-parabolic degenerate partial differential
operators, an important class is represented by H\"{o}rmander's operators%
\begin{equation}
L=\sum_{i=1}^{n}X_{i}^{2}+X_{0} \label{H}%
\end{equation}
built on real smooth vector fields
\begin{equation}
X_{i}=\sum_{j=1}^{p}b_{ij}\left(  x\right)  \partial_{x_{j}} \label{X_i}%
\end{equation}
which are defined in some domain $\Omega\subset\mathbb{R}^{p}$. A famous
theorem by H\"{o}rmander \cite{H} states that if the Lie algebra generated by
the $X_{i}$'s ($i=0,1,2,...,n$) coincides with the whole $\mathbb{R}^{p}$ at
any point of $\Omega$, then $L$ is hypoelliptic in $\Omega$, that is any
distributional solution to the equation $Lu=f\in C^{\infty}\left(
\Omega\right)  $ belongs to $C^{\infty}\left(  \Omega\right)  $. Over the
years, a number of deep properties of H\"{o}rmander's operators and systems of
H\"{o}rmander's vector fields have been established. Some of them are related
to the metric induced by H\"{o}rmander's vector fields (connectivity property,
doubling property for metric balls, see \cite{NSW}), or to the
\textquotedblleft gradient\textquotedblright\ associated to H\"{o}rmander's
vector fields (Poincar\'{e}'s inequality, see \cite{J}); other properties are
related to second order H\"{o}rmander's operators (properties of fundamental
solutions, see \cite{Fo}, \cite{NSW}, \cite{SC}, or a priori estimates on the
second order derivatives with respect to the vector fields, see \cite{Fo},
\cite{RS}).

One can note that, apart from H\"{o}rmander's hypoellipticity theorem, which
intrinsically requires $C^{\infty}$ regularity of the vector fields, most of
the important existing results in this area are expressed by statements which
are meaningful, and hopefully still hold, under much less regularity of the
vector fields. So a natural question consists in asking how much of the
classical theory of H\"{o}rmander's vector fields and H\"{o}rmander's
operators still holds if we consider a family of vector fields whose
coefficients possess just the right number of derivatives which are enough to
check that H\"{o}rmander's condition at some step $r$ holds (see section 2 for
the definition). However, this generalization is far from being obvious, since
if one tries to repeat the classical proofs just paying attention to the
minimal regularity required, one finds that some arguments need the existence
of a very high number of derivatives (for instance, the double of the step
$r$), while others simply cannot be repeated. Experience shows that proving
relevant results about nonsmooth vector fields under reasonably weak
assumptions is almost always a hard task. Nevertheless, this is a natural
problem if one hopes to settle the basis for applications to nonlinear
equations which involve vector fields depending on the solution itself (such
as Levi-type equations that we will discuss later in this introduction).

This paper is the third step in a larger project started by three of us in
\cite{BBP} and \cite{BBP2}, and devoted to this issue. Our framework is the
following. Let $X_{0},X_{1},...,X_{n}$ be a system of real vector fields,
defined in a bounded domain $\Omega\subset\mathbb{R}^{p}.$ We assume that for
some integer $r\geqslant2\ $and some $\alpha\in(0,1]$ the coefficients of the
vector fields $X_{1},X_{2},...,X_{n}$ belong to $C^{r-1,\alpha}\left(
\Omega\right)  ,$ while the coefficients of $X_{0}$ belong to $C^{r-2,\alpha
\,}\left(  \Omega\right)  $. If $r=2,$ we assume $\alpha=1$. Here and in the
following, $C^{k,\alpha}$ stands for the classical space of functions which
are differentiable up to order $k$, with $\alpha$-H\"{o}lder continuous
derivatives of order $k$. Moreover, we assume that $X_{0},X_{1},...,X_{n}$
satisfy H\"{o}rmander's condition of \emph{weighted step} $r$ in $\Omega$: if
we assign weight $1$ to $X_{1},X_{2},...,X_{n}$ and weight $2$ to $X_{0},$
then the commutators of the vector fields $X_{i},$ up to weight $r$, span
$\mathbb{R}^{p}$ at any point of $\Omega$ (more precise definitions will be
given later)$.$

An extension to this nonsmooth context of some basic properties of the
distance induced by the vector fields, Chow's connectivity theorem, the
estimate on the volume of metric balls, the doubling condition, and
Poincar\'{e}'s inequality has been given in \cite{BBP}. These results also
imply a Sobolev embedding and the validity of Moser's iteration technique to
handle operators of the kind%
\[
\sum_{i,j=1}^{n}X_{i}^{\ast}\left(  a_{ij}\left(  x\right)  X_{j}u\right)  .
\]

In \cite{BBP2} the same authors have extended to the nonsmooth context the
lifting and approximation theory developed in the smooth case by
Rothschild-Stein \cite{RS} and some related results, such as the comparison
between volumes of balls in the lifted and original space. Starting with the
paper \cite{RS}, this technique has been used, in the smooth case, to reduce
the study of general H\"{o}rmander's operators (\ref{H}) to that of left
invariant homogeneous operators on homogeneous groups, for which Folland's
theory developed in \cite{Fo} applies, granting the existence of a homogeneous
left invariant fundamental solution, which is a good starting point to prove
a-priori estimates of several types.

Following this idea, in the present paper we use tools and results from
\cite{BBP} and \cite{BBP2} to study H\"{o}rmander's operators (\ref{H}) built
with nonsmooth vector fields or, briefly, \emph{nonsmooth H\"{o}rmander's
operators}. Namely, we are able to adapt to this situation the classical
Levi's parametrix method, in order to build a fundamental solution
$\gamma\left(  x,y\right)  $ for $L$ (in the small), possessing some good
properties. More precisely, under the above assumptions we prove (see Thm.
\ref{Thm gamma}) that for any $x_{0}\in\Omega$ there exists a neighborhood
$U\left(  x_{0}\right)  $ and a function $\gamma\left(  x,y\right)  $, defined
and continuous in the joint variables for $x,y\in U\left(  x_{0}\right)  $,
$x\neq y$, satisfying%
\begin{equation}
\int\gamma\left(  x,y\right)  L^{\ast}\omega\left(  x\right)  dx=-\omega
\left(  y\right)  \label{gamma 1 intro}%
\end{equation}
for any $\omega\in C_{0}^{\infty}\left(  U\left(  x_{0}\right)  \right)  $;
moreover, $\gamma$ satisfies the bounds%
\begin{align}
\left\vert \gamma\left(  x,y\right)  \right\vert  &  \leqslant c\frac{d\left(
x,y\right)  ^{2}}{\left\vert B\left(  x,d\left(  x,y\right)  \right)
\right\vert };\label{gamma 2 intro}\\
\left\vert X_{i}\gamma\left(  x,y\right)  \right\vert  &  \leqslant
c\frac{d\left(  x,y\right)  }{\left\vert B\left(  x,d\left(  x,y\right)
\right)  \right\vert },\text{\ \ \ \ }i=1,2,...,n, \label{ganmma 3 intro}%
\end{align}
where, here and in the following, $X_{i}\gamma\left(  x,y\right)  $ denotes
the $X_{i}$ derivative with respect to the \emph{first }variable, $x$, the
distance $d$ is the one induced by the vector fields $X_{i}$, and $B\left(
x,r\right)  $ are the corresponding balls.

Under the stronger assumption that the coefficients of the $X_{i}$'s
($i=1,2,...,n$) belong to $C^{r,\alpha}\left(  \Omega\right)  $ and the
coefficients of $X_{0}$ belong to $C^{r-1,\alpha\,}\left(  \Omega\right)  $,
we are able to prove that $\gamma$ also possesses second derivatives with
respect to the vector fields, satisfying the bounds%
\begin{align}
\left\vert X_{j}X_{i}\gamma\left(  x,y\right)  \right\vert  &  \leqslant
\frac{c}{\left\vert B\left(  x,d\left(  x,y\right)  \right)  \right\vert
},\text{\ \ \ \ }i,j=1,2,...,n,\label{gamma 4 intro}\\
\left\vert X_{0}\gamma\left(  x,y\right)  \right\vert  &  \leqslant\frac
{c}{\left\vert B\left(  x,d\left(  x,y\right)  \right)  \right\vert
},\nonumber
\end{align}
and that $\gamma\left(  \cdot,y\right)  $ is a classical solution to the
equation $L\gamma\left(  x,y\right)  =0$ for $x\neq y$ (see Thm.
\ref{Thm XXgamma}). Exploiting these results we prove (see Thm.
\ref{ThLocalSolv}) the following local solvability result for $L$: for every
$x_{0}\in\Omega$ there exists a neighborhood $U$ such that for any $\beta>0$
and $f\in C_{X}^{\beta}\left(  U\right)  $ (i.e., $\beta$-H\"{o}lder
continuous with respect to the distance $d$) there exists a classical solution
$u$ to the equation $Lu=f$ in $U$. Pushing even forward our analysis, we show
that the functions $X_{i}X_{j}\gamma$ satisfy the following local H\"{o}lder
estimate: for every $x_{1},x_{2},y\in U$ such that $d\left(  x_{1},y\right)
\geqslant2d\left(  x_{1},x_{2}\right)  $,%
\begin{equation}
\left\vert X_{i}X_{j}\gamma\left(  x_{1},y\right)  -X_{i}X_{j}\gamma\left(
x_{2},y\right)  \right\vert \leqslant c_{\varepsilon}\left(  \frac{d\left(
x_{1},x_{2}\right)  }{d\left(  x_{1},y\right)  }\right)  ^{\alpha-\varepsilon
}\frac{1}{\left\vert B\left(  x_{1},d\left(  x_{1},y\right)  \right)
\right\vert } \label{gamma 5 intro}%
\end{equation}
for any $\varepsilon\in\left(  0,\alpha\right)  $ and $i,j=1,2,...,n$, with
$c_{\varepsilon}$ depending on $\varepsilon$ (Thm. \ref{Thm schauder fundam}).
As a consequence, we eventually show that the local solution $w$ to $Lw=f$
that we have built for $f\in C_{X}^{\beta}\left(  U\right)  $, with
$\beta<\alpha$, actually belongs to $C_{X,loc}^{2,\beta}\left(  U\right)  $
(see Thm. \ref{Thm holder w}).

\bigskip

\noindent\textbf{Comparison with the existent literature}

\noindent The study of nonsmooth H\"{o}rmander's vector fields has been
carried out by several authors; we refer to the introduction of \cite{BBP} for
a detailed discussion of the related bibliography. Here we just point out that
the peculiarity of the research project consisting in the present paper and
\cite{BBP}, \cite{BBP2} is that of considering nonsmooth H\"{o}rmander's
vector fields of completely general form. Indeed, with the notable exception
of the papers \cite{MM3}, \cite{MM4}, \cite{MM5} by Montanari-Morbidelli and
some papers by Karmanova-Vodopyanov (see \cite{KV}, \cite{VK} and the
references therein), all the other previous results about nonsmooth vector
fields either hold only for vector fields with a particular structure, or
assume axiomatically some important properties of the metric induced by the
vector fields themselves. Another characteristic feature of the present
research is to take explicitly into account the possibility that one of the
vector fields $X_{0}$ (\textquotedblleft the drift\textquotedblright) could
have weight two, as in the case of H\"{o}rmander's operators (\ref{H}). This
is relevant for instance in view of the possible application of the present
theory to operators of Kolmogorov-Fokker-Planck type with nonsmooth drift.
While the literature devoted to the geometry of nonsmooth vector fields is
quite large, the one about \emph{H\"{o}rmander's operators }built on nonsmooth
vector fields is much narrower. Particular classes of operators of this kind
have been studied in the framework of regularity results for nonlinear
equations of Levi type by Citti, Lanconelli, Montanari, starting with the
paper \cite{C1} and continuing with \cite{CM1}, \cite{CLM}, \cite{Mon1} (see
also references therein). A somewhat related field of research is that about
the Levi-Monge-Amp\`{e}re equation, see \cite{ML}, \cite{ML2}, which also
motivates the study of nonvariational operators modeled on (possibly
nonsmooth) H\"{o}rmander's vector fields. Another application of this circle
of ideas to a nonlinear regularization problem has been given by Citti,
Pascucci, Polidoro in \cite{CPP}. However, the present paper seems to be the
first one where H\"{o}rmander's operators built with nonsmooth vector fields
of general structure are studied.

Let us come to some remarks about the techniques used. The parametrix method
was originally developed more than a century ago by E. E. Levi to study
uniformly elliptic equations of order $2n$ (see \cite{levi}), and later
extended to uniformly parabolic operators (see e.g. \cite{Fr}). For more
details about this method in the elliptic case we refer to \cite[\S \ 19]{Mi},
\cite[Part IV, Chap.3]{He} and \cite{K}. In particular, the last reference
contains a rich account of the previous literature on this subject and a
careful discussion of the assumptions made by different authors to implement
the method. The parametrix method was first adapted to hypoelliptic
ultraparabolic operators of Kolmogorov-Fokker-Planck type by Polidoro in
\cite{pol}, exploiting the knowledge of an explicit expression for the
fundamental solution of the \textquotedblleft frozen\textquotedblright%
\ operator, which had been constructed in \cite{LPol}. It was later adapted by
Bonfiglioli, Lanconelli, Uguzzoni in \cite{BLU} to a general class of
operators structured on homogeneous left invariant (smooth) vector fields on
Carnot groups, for which no explicit fundamental solution is known in general,
and by Bramanti, Brandolini, Lanconelli, Uguzzoni in \cite{BBLU} to the more
general context of arbitrary (smooth) H\"{o}rmander's vector fields. Finally,
in the nonsmooth context, the parametrix method has been exploited by
Manfredini in \cite{Ma} to deal with sum of squares of $C^{1,\alpha}%
$-intrinsic vector fields of step 2, with a particular structure.

In order to evaluate our assumptions about the regularity of vector fields,
one can draw a comparison with the assumptions made in the elliptic case, as
reported in \cite{K}. Rewriting our operator in the form%
\[
L=\sum_{j,k=1}^{p}a_{jk}\left(  x\right)  \partial_{x_{j}x_{k}}^{2}+\sum
_{k=1}^{p}b_{k}\left(  x\right)  \partial_{x_{k}}+c\left(  x\right)
\]
one can see that our stronger assumptions (see Assumptions B in
\S \ \ref{sec:further regularity}) imply in the simplest degenerate case $r=2$%
\[
a_{jk}\in C^{2,1}\left(  \Omega\right)  ,b_{k}\in C^{1,1}\left(
\Omega\right)  ,c\in C^{1,1}\left(  \Omega\right)
\]
while in the elliptic case \cite[Thm.3]{K} it is essentially required that%
\[
a_{jk}\in C^{2}\left(  \Omega\right)  \cap C^{0,\alpha}\left(  \Omega\right)
,b_{j}\in C^{1}\left(  \Omega\right)  \cap C^{0,\alpha}\left(  \Omega\right)
,c\in C^{0,\alpha}\left(  \Omega\right)  .
\]

\bigskip

\noindent\textbf{Strategy and plan of the paper}

\noindent The technique of \textquotedblleft lifting and
approximation\textquotedblright\ developed by Rothschild-Stein\ in \cite{RS}
and extended to nonsmooth vector fields in \cite{BBP2}, coupled with the
results by Folland \cite{Fo} suggests that, in order to study the (nonsmooth)
operator (\ref{H}), natural steps consist in lifting $L$, in a neighborhood of
a point $x_{0}\in\mathbb{R}^{p}$, to a new (nonsmooth) operator%
\[
\widetilde{L}=\sum_{i=1}^{n}\widetilde{X}_{i}^{2}+\widetilde{X}_{0}%
\]
defined in a neighborhood $\mathcal{U}$ of $\left(  x_{0},0\right)
\in\mathbb{R}^{p+m}$, and then approximate $\widetilde{L}$ with a (smooth)
left invariant homogeneous operator%
\[
\mathcal{L}=\sum_{i=1}^{n}Y_{i}^{2}+Y_{0}%
\]
which possesses a homogeneous left invariant fundamental solution
$\Gamma\left(  v^{-1}\circ u\right)  $, with respect to a structure of
homogeneous group in $\mathbb{R}^{p+m}$. Then, a natural parametrix of
$\widetilde{L}$ can be defined by
\[
P_{0}\left(  \xi,\eta\right)  =\Gamma\left(  \Theta_{\eta}\left(  \xi\right)
\right)  ,
\]
where the map $\Theta_{\eta}\left(  \xi\right)  $ (a nonsmooth version,
introduced in \cite{BBP2}, of the function defined by Rothschild-Stein in
\cite{RS}) is, for any fixed $\eta\in\mathcal{U}$, a smooth diffeomorphism
which allows to approximate $\widetilde{L}$ with $\mathcal{L}$ near $\eta$,
and $\Theta_{\eta}\left(  \xi\right)  $ depends on $\eta$ in a H\"{o}lder
continuous way. Hence $P_{0}\left(  \xi,\eta\right)  $ is smooth in $\xi$ but
just H\"{o}lder continuous in $\eta$ (or $C^{1,\alpha}$ in $\eta$, if the
coefficients of the $X_{i}$'s are $C^{r,\alpha}$ and the coefficients of
$X_{0}$ are $C^{r-1,\alpha}$, see Proposition \ref{Prop bad theta}). This
rough asymmetry in the properties of $P_{0}$ with respect to the two variables
prevents us from repeating Rothschild-Stein's technique to prove $L^{p}$ or
$C^{\alpha}$ estimates for second order derivatives with respect to the vector
fields, for a solution to $Lu=f$. Instead, one can think to adapt to this case
the classical Levi's parametrix method, which is compatible with a different
degree of regularity of $P_{0}$ in the two variables. Now, if we applied the
parametrix method directly to the kernel $P_{0}$ we would build a local
fundamental solution for $\widetilde{L}$. Starting from this object, however,
there is no obvious way to produce a local fundamental solution for $L$.
Instead, we have to define directly a parametrix for $L,$ shaped on $P_{0}$
saturating the lifted variables by integration, in the following way:
\begin{equation}
P\left(  x,y\right)  =\int_{\mathbb{R}^{m}}\left(  \int_{\mathbb{R}^{m}}%
\Gamma\left(  \Theta_{\left(  y,k\right)  }\left(  x,h\right)  \right)
\varphi\left(  h\right)  dh\right)  \varphi\left(  k\right)  dk,\text{ \ for
}x,y\in U, \label{P intro}%
\end{equation}
where $\varphi\in C_{0}^{\infty}\left(  \mathbb{R}^{m}\right)  $ is a cutoff
function fixed once and for all, equal to one in a neighborhood of the origin.
This $P$ turns out to be a good parametrix for $L$, and starting with it we
can actually construct a local fundamental solution for $L,$ satisfying
natural growth estimates and regularity properties. However, performing this
construction (see \S 4) is a hard task, since we are forced to work in a
metric measure space where the measure of balls does not behave like a fixed
power of the radius, in particular there is not a homogeneous dimension.
Therefore a good deal of preliminary work (see \S 3) has to be done to craft
the geometric and real analysis tools necessary to make the Levi method work.
In particular, it turns out that the right function to measure the size of a
kernel $k\left(  x,y\right)  $ is%
\[
\phi_{\beta}\left(  x,y\right)  =\int_{d\left(  x,y\right)  }^{R}%
\frac{r^{\beta-1}}{\left\vert B\left(  x,r\right)  \right\vert }dr
\]
which for $\beta\in\left(  0,p\right)  $, is bounded by (but not equivalent
to)%
\begin{equation}
c\frac{d\left(  x,y\right)  ^{\beta}}{\left\vert B\left(  x,d\left(
x,y\right)  \right)  \right\vert } \label{d beta}%
\end{equation}
and satisfies a key property which is very useful in iterative computations
(see Theorem \ref{Thm phi alfa}), and could not be proved for (\ref{d beta}).

The Levi method is then implemented as follows. We look for a fundamental
solution for $L$ of the form%
\[
\gamma\left(  x,y\right)  =P\left(  x,y\right)  +J\left(  x,y\right)
\]
where $P$ is as in (\ref{P intro}) and%
\[
J\left(  x,y\right)  =\int_{U}P\left(  x,z\right)  \Phi\left(  z,y\right)
dz.
\]
In turn, we will find $\Phi$ as the series
\[
\Phi\left(  z,y\right)  =\sum_{j=1}^{\infty}Z_{j}\left(  z,y\right)
\text{\ for }z\neq y
\]
where the $Z_{j}$'s are defined inductively by%
\begin{align*}
Z_{1}\left(  x,y\right)   &  =L\left(  P\left(  \cdot,y\right)  \right)
\left(  x\right) \\
Z_{j+1}\left(  x,y\right)   &  =\int_{U}Z_{1}\left(  x,z\right)  Z_{j}\left(
z,y\right)  dz\text{ }\ \ \ \text{for }x\neq y.
\end{align*}

In \S 4, exploiting the results of \S 3 and some results proved in \cite{BBP},
\cite{BBP2} and recalled in \S 2, we prove the basic properties and upper
bounds satisfied by the functions $Z_{1},Z_{j},\Phi,J,$ and we deduce the
existence of a local fundamental solution $\gamma$ satisfying
(\ref{gamma 1 intro}), (\ref{gamma 2 intro}), (\ref{ganmma 3 intro}).

The next step, in \S 5, is then to compute the second derivatives of $\gamma$,
that is%
\[
X_{i}X_{j}\gamma\left(  x,y\right)  =X_{i}X_{j}P\left(  x,y\right)  +X_{i}%
\int_{U}X_{j}P\left(  x,z\right)  \Phi\left(  z,y\right)  dz
\]
(all the $X_{i}$ derivatives being taken with respect to the $x$ variable). In
order to do that one has to exploit, in particular, H\"{o}lder continuity
(with respect to $d$) of $z\mapsto\Phi\left(  z,y\right)  $, to allow
differentiation under the integral sign. Proving H\"{o}lder continuity of
$\Phi$ and the existence of $X_{i}X_{j}\gamma$ forces us to deepen the
analysis of the properties of the map $\Theta_{\eta}\left(  \xi\right)  $ and
to strengthen our assumptions on the vector fields, requiring from now on
$X_{i}\in C^{r,\alpha}$ and $X_{0}\in C^{r-1,\alpha}$. Once the existence of
$X_{i}X_{j}\gamma$ and the upper bound (\ref{gamma 4 intro}) are proved, we
can show that for any $\beta>0\ $and $f\in C_{X}^{\beta}\left(  U\right)  $,
the function%
\begin{equation}
w\left(  x\right)  =-\int_{U}\gamma\left(  x,y\right)  f\left(  y\right)  dy
\label{w intro}%
\end{equation}
is a classical solution to the equation $Lw=f$ in $U$. In particular, we
establish an \textquotedblleft explicit\textquotedblright\ representation
formula for $X_{i}X_{j}w$ (see Corollary \ref{Corollary rep formula}),
containing singular integrals, fractional integrals, and multiplicative terms.
This formula, although rather involved, is designed in view of the subsequent
proof of H\"{o}lder continuity. The point is that, for technical reasons
related to the starting definition of the parametrix $P\left(  x,y\right)  $,
which is assigned by an integral with respect to the \textquotedblleft lifted
variables\textquotedblright, the singular part of
\[
X_{i}\int_{U}X_{j}\gamma\left(  x,y\right)  f\left(  y\right)  dy
\]
cannot be easily written in a form like%
\[
\lim_{\varepsilon\rightarrow0}\int_{d\left(  x,y\right)  >\varepsilon}%
X_{i}X_{j}\gamma\left(  x,y\right)  f\left(  y\right)  dy,
\]
which should allow to apply directly some abstract theory of singular
integrals. Instead, we have to rewrite properly the integral, to transform the
singular part into something like%
\begin{equation}
\int k\left(  x,y\right)  \left[  f\left(  y\right)  -f\left(  x\right)
\right]  dy \label{sing intro}%
\end{equation}
with $k$ singular near the diagonal.

In \S 5 we also prove H\"{o}lder estimates on $X_{i}X_{j}\gamma$, the
difficult part of the estimate being that on $X_{i}X_{j}J$. We then pass to
prove that the solution (\ref{w intro}) to $Lu=f$ possesses locally H\"{o}lder
continuous derivatives $X_{i}X_{j}w$. This amounts to proving H\"{o}lder
continuity of each term of the representation formula for $X_{i}X_{j}w$
previously established. While for the fractional integrals it is fairly enough
to exploit H\"{o}lder continuity of $X_{i}X_{j}J$, the singular integral term
also requires the proof of a cancellation property of the kind%
\[
\left\vert \int_{r_{1}<d\left(  x,y\right)  <r_{2}}k\left(  x,y\right)
dy\right\vert \leqslant c\text{ \ \ for any }r_{1}<r_{2}.
\]
In order to prove $C_{X}^{\alpha}$ continuity of singular and fractional
integrals we both apply some abstract results proved in \cite{BZ2}\ for
locally homogeneous spaces and revise some techniques used in \cite{BB}.

Finally, in Appendix we give some examples of nonsmooth H\"{o}rmander's
operators satisfying assumption A in \S 2 or assumption B in \S 5.

\section{Some known results about \newline nonsmooth H\"{o}rmander's vector
fields\label{sec known results}}

In this section we fix precisely our notation and assumptions, and recall a
number of known facts which will be used throughout the paper. In some cases,
we do not recall the complete definitions given in \cite{BBP,BBP2}, but only
the properties that are needed for our current purposes.

Let $X_{0},X_{1},...,X_{n}$ be a system of real vector fields%
\[
X_{i}=\sum_{j=1}^{p}b_{ij}\left(  x\right)  \partial_{x_{j}},
\]
defined in a bounded, arcwise connected open set $\Omega\subset\mathbb{R}%
^{p}.$ Let us assign to each $X_{i}$ a \textit{weight} $p_{i}$, saying that%
\[
p_{0}=2\text{ and }p_{i}=1\text{ for }i=1,2,\ldots,n.
\]
For any multiindex%
\[
I=\left(  i_{1},i_{2},\ldots,i_{k}\right)
\]
we define the \textit{weight} of $I$ as%
\[
\left\vert I\right\vert =\sum_{j=1}^{k}p_{i_{j}}
\]
and we set
\[
X_{I}=X_{i_{1}}X_{i_{2}}...X_{i_{k}}
\]
and
\[
X_{\left[  I\right]  }=\left[  X_{i_{1}},\left[  X_{i_{2}},...\left[
X_{i_{k-1}},X_{i_{k}}\right]  ...\right]  \right]  ,
\]
where $\left[  X,Y\right]  $ is the usual Lie bracket of vector fields. If
$I=\left(  i_{1}\right)  ,$ then%
\[
X_{\left[  I\right]  }=X_{i_{1}}=X_{I}.
\]
As usual, $X_{\left[  I\right]  }$ can be seen either as a differential
operator or as a vector field. We will write $X_{\left[  I\right]  }f$ to
denote the differential operator $X_{\left[  I\right]  }$ acting on a function
$f$, and $\left(  X_{\left[  I\right]  }\right)  _{x}$ to denote the vector
field $X_{\left[  I\right]  }$ evaluated at the point $x$.

For a positive integer $k$ and $\alpha\in(0,1]$ we define the (classical)
H\"{o}lder space $C^{k,\alpha}\left(  \Omega\right)  $ of functions $k$ times
differentiable (in classical sense), with derivatives of order $k$ belonging
to the H\"{o}lder (or Lipschitz) space $C^{\alpha}\left(  \Omega\right)  ,$
defined by the finiteness of the norm%
\[
\left\Vert f\right\Vert _{C^{\alpha}\left(  \Omega\right)  }=\sup_{x\in\Omega
}\left\vert f\left(  x\right)  \right\vert +\left\vert f\right\vert
_{C^{\alpha}\left(  \Omega\right)  },
\]
with%
\[
\left\vert f\right\vert _{C^{\alpha}\left(  \Omega\right)  }=\sup
_{x,y\in\Omega,x\neq y}\frac{\left\vert f\left(  x\right)  -f\left(  y\right)
\right\vert }{\left\vert x-y\right\vert ^{\alpha}}.
\]

\bigskip\noindent\textbf{Assumptions A. }We assume that for some integer
$r\geqslant2\ $and some $\alpha\in(0,1],$ the coefficients of the vector
fields $X_{1},X_{2},...,X_{n}$ belong to $C^{r-1,\alpha}\left(  \Omega\right)
,$ while the coefficients of $X_{0}$ belong to $C^{r-2,\alpha\,}\left(
\Omega\right)  $. If $r=2,$ we assume $\alpha=1$. Moreover, we assume that
$X_{0},X_{1},...,X_{n}$ satisfy H\"{o}rmander's condition of step $r$ in
$\Omega$, i.e. the vectors%
\[
\left\{  \left(  X_{\left[  I\right]  }\right)  _{x}\right\}  _{\left\vert
I\right\vert \leqslant r}%
\]
span $\mathbb{R}^{p}$ for any $x\in\Omega.$ (For examples of systems of vector
fields satisfying the assumptions, see the Appendix).

\bigskip

We note that under our assumptions, for any $1\leqslant k\leqslant r,$ the
differential operators $\left\{  X_{I}\right\}  _{\left\vert I\right\vert
\leqslant k}$ and the vector fields $\left\{  X_{\left[  I\right]  }\right\}
_{\left\vert I\right\vert \leqslant k}$ are well defined, and have
$C^{r-k,\alpha}$ coefficients.

We will sometimes need the \emph{transpose operator}
\begin{equation}
L^{\ast}=\sum_{i=1}^{n}\left(  X_{i}^{\ast}\right)  ^{2}+X_{0}^{\ast}
\label{transposed}%
\end{equation}
defined by the transpose operators $X_{i}^{\ast}$ of the vector fields, which
act on smooth functions as
\[
X_{i}^{\ast}u\left(  x\right)  =-\sum_{j=1}^{p}\partial_{x_{j}}\left(
b_{ij}\left(  x\right)  u\left(  x\right)  \right)  .
\]
Note that, in order for $L^{\ast}u$ to be well defined, at least as an
$L^{\infty}$ function, we need the $b_{ij}$'s to be at least $C^{1,1}$ for
$i=1,2,...,p,$ and $C^{0,1}$ for $i=0$. This is one of the reasons why we need
$\alpha=1$ if $r=2$. We will also use this in the proof of Theorem
\ref{Thm Calphatheta}.

\bigskip

The subelliptic metric, analogous to that introduced by Nagel-Stein-Wainger in
\cite{NSW}, is defined as follows:

\begin{definition}
\label{Definition CC distance}For any $\delta>0,$ let $C\left(  \delta\right)
$ be the class of absolutely continuous mappings $\varphi:\left[  0,1\right]
\longrightarrow\Omega$ which satisfy%
\[
\varphi^{\prime}\left(  t\right)  =\sum_{\left\vert I\right\vert \leqslant
r}a_{I}\left(  t\right)  \left(  X_{\left[  I\right]  }\right)  _{\varphi
\left(  t\right)  }\text{ a.e.}%
\]
with $a_{I}:\left[  0,1\right]  \rightarrow\mathbb{R}$ measurable functions,
\[
\left\vert a_{I}\left(  t\right)  \right\vert \leqslant\delta^{\left\vert
I\right\vert }.
\]
Then define%
\[
d\left(  x,y\right)  =\inf\left\{  \delta>0:\exists\varphi\in C\left(
\delta\right)  \text{ with }\varphi\left(  0\right)  =x,\varphi\left(
1\right)  =y\right\}
\]
and denote $B(x,\rho)$ the associated ball of center $x$ and radius $\rho$.
\end{definition}

The finiteness of $d$ for any couple of points of $\Omega,$ as well as the
basic properties of this distance in the nonsmooth context have been
established in \cite{BBP}. In particular, we will use the following facts:

\begin{proposition}
[Relation with the Euclidean distance]\label{Prop fefferman phong}There exist
a positive constant $c_{1}$ depending on $\Omega$ and the $X_{i}$'s and, for
every $\Omega^{\prime}\Subset\Omega,$ a positive constant $c_{2}$ depending on
$\Omega^{\prime}$ and the $X_{i}$'s, such that%
\begin{equation}
c_{1}\left\vert x-y\right\vert \leqslant d\left(  x,y\right)  \leqslant
c_{2}\left\vert x-y\right\vert ^{1/r}\text{ for any }x,y\in\Omega^{\prime}.
\label{fefferman-phong}%
\end{equation}
In particular, the distance $d$ induces Euclidean topology.
\end{proposition}

\begin{theorem}
[Doubling condition]\label{Thm doubling nonsmooth}Under the previous
assumptions, for any domain $\Omega^{\prime}\Subset\Omega,$ there exist
positive constants $c,\rho_{0},$ depending on $\Omega,\Omega^{\prime}$ and the
$X_{i}$'s$,$ such that%
\[
\left\vert B\left(  x,2\rho\right)  \right\vert \leqslant c\left\vert B\left(
x,\rho\right)  \right\vert
\]
for any $x\in\Omega^{\prime},$ $\rho<\rho_{0}.$
\end{theorem}

\begin{theorem}
[Volume of metric balls]For any family $\mathcal{I}$ of $p$ multiindices
$I_{1},I_{2},...,I_{p}$ with $\left\vert I_{j}\right\vert \leqslant r$, let
$\left\vert \mathcal{I}\right\vert =\sum_{j=1}^{p}\left\vert I_{j}\right\vert
$ and $\lambda_{\mathcal{I}}\left(  x\right)  $ be the determinant of the
$p\times p$ matrix with rows $\left\{  \left(  X_{\left[  I_{j}\right]
}\right)  _{x}\right\}  _{I_{j}\in\mathcal{I}}$. For any $\Omega^{\prime
}\Subset\Omega$ there exist positive constants $c_{1},c_{2},\rho_{0}$
depending on $\Omega,\Omega^{\prime}$ and the $X_{i}$'s$,$ such that%
\begin{equation}
c_{1}\sum_{\mathcal{I}}\left\vert \lambda_{\mathcal{I}}\left(  x\right)
\right\vert \rho^{\left\vert \mathcal{I}\right\vert }\leqslant\left\vert
B\left(  x,\rho\right)  \right\vert \leqslant c_{2}\sum_{\mathcal{I}%
}\left\vert \lambda_{\mathcal{I}}\left(  x\right)  \right\vert \rho
^{\left\vert \mathcal{I}\right\vert } \label{volume}%
\end{equation}
for any $\rho<\rho_{0},$ $x\in\Omega^{\prime}$, where the sum is taken over
any family $\mathcal{I}$ with the above properties.
\end{theorem}

\begin{definition}
[H\"{o}lder spaces]\label{Def Holder spaces}For any $U\Subset\Omega$ we can
introduce H\"{o}lder spaces $C_{X}^{\alpha}\left(  U\right)  $ with respect to
the distance $d$, letting for $\alpha>0$,%
\[
\left\Vert f\right\Vert _{C_{X}^{\alpha}\left(  U\right)  }=\sup_{x\in
U}\left\vert f\left(  x\right)  \right\vert +\left\vert f\right\vert
_{C_{X}^{\alpha}\left(  U\right)  },
\]
with%
\[
\left\vert f\right\vert _{C_{X}^{\alpha}\left(  U\right)  }=\sup_{x,y\in
U,x\neq y}\frac{\left\vert f\left(  x\right)  -f\left(  y\right)  \right\vert
}{d\left(  x,y\right)  ^{\alpha}}.
\]

Also, we let
\[
C_{X}^{2\,,\alpha}(U)=\{f:U\rightarrow\mathbb{R}|\,\left\Vert f\right\Vert
_{C_{X}^{2,\alpha}\left(  U\right)  }<\infty\}
\]
where
\[
\left\Vert f\right\Vert _{C_{X}^{2,\alpha}\left(  U\right)  }=\left\Vert
f\right\Vert _{C_{X}^{\alpha}\left(  U\right)  }+\sum_{|I|\leqslant
2}\left\Vert X_{I}f\right\Vert _{C_{X}^{\beta}\left(  U\right)  }.
\]

\end{definition}

By (\ref{fefferman-phong}) the following hold:%
\begin{align*}
f  &  \in C^{\alpha}\left(  \Omega^{\prime}\right)  \Rightarrow f\in
C_{X}^{\alpha}\left(  \Omega^{\prime}\right) \\
f  &  \in C_{X}^{\alpha}\left(  \Omega^{\prime}\right)  \Rightarrow f\in
C^{\alpha/r}\left(  \Omega^{\prime}\right)  .
\end{align*}
Note, in particular, that saying \textquotedblleft$f\in C^{\beta}\left(
\Omega^{\prime}\right)  $ for some $\beta>0$\textquotedblright\ is the same as
\textquotedblleft$f\in C_{X}^{\beta}\left(  \Omega^{\prime}\right)  $ for some
$\beta>0$\textquotedblright.

We will also need the following property, which is similar to that proved in
\cite[Thm. 5.11]{BBP}. For convenience of the reader, we recall here its short proof.

\begin{proposition}
\label{Prop Lagrange}Let $\overline{x}\in\Omega$ and let $B\left(
\overline{x},R\right)  \subset\Omega$. For any $f\in C^{1}(B\left(
\overline{x},R\right)  )$, one has
\[
\left\vert f(x)-f(\overline{x})\right\vert \leqslant d\left(  x,\overline
{x}\right)  \left(  \overset{n}{\underset{i=1}{\sum}}\underset{B\left(
\overline{x},R\right)  }{\sup}\left\vert X_{i}f\right\vert +d\left(
x,\overline{x}\right)  \underset{B\left(  \overline{x},R\right)  }{\sup
}\left\vert X_{0}f\right\vert \right)
\]
for any $x\in B\left(  \overline{x},R\right)  $.
\end{proposition}

\begin{proof}
Let $x\in B\left(  \overline{x},R\right)  $, hence by Definition
\ref{Definition CC distance} there exists a curve $\varphi(t)$, such that
$\varphi(0)=\overline{x},\varphi(1)=x$, and
\[
\varphi^{\prime}(t)=\overset{n}{\underset{i=0}{\sum}}\lambda_{i}(t)\left(
X_{i}\right)  _{\varphi(t)}%
\]
with $\left\vert \lambda_{0}(t)\right\vert \leqslant d\left(  x,\overline
{x}\right)  ^{2}$ and $\left\vert \lambda_{i}(t)\right\vert \leqslant d\left(
x,\overline{x}\right)  $ for $i=1,\ldots,n$. Moreover, every point
$\gamma\left(  t\right)  $ for $t\in\left(  0,1\right)  $ belongs to $B\left(
\overline{x},R\right)  $. Then we can write:
\begin{align*}
\left\vert f(x)-f(\overline{x})\right\vert  &  =\left\vert \int_{0}^{1}%
\frac{d}{dt}f(\varphi(t))\,dt\right\vert =\left\vert \int_{0}^{1}\overset
{n}{\underset{i=0}{\sum}}\lambda_{i}(t)\left(  X_{i}f\right)  _{\varphi
(t)}\,dt\right\vert \\
&  \leqslant d\left(  x,\overline{x}\right)  \overset{n}{\underset{i=1}{\sum}%
}\underset{B\left(  \overline{x},R\right)  }{\sup}\left\vert X_{i}f\right\vert
+d\left(  x,\overline{x}\right)  ^{2}\underset{B\left(  \overline{x},R\right)
}{\sup}\left\vert X_{0}f\right\vert ,
\end{align*}
as desired.
\end{proof}

In \cite{BBP2} an extension to nonsmooth vector fields of some known results
by Rothschild-Stein \cite{RS} are proved. The first one is:

\begin{theorem}
[Lifting theorem]\label{Thm lifting}For every $x_{0}\in\Omega$, there exist a
neighborhood $U\left(  x_{0}\right)  ,$ an integer $m$ and vector fields of
the form%
\begin{equation}
\widetilde{X}_{k}=X_{k}+\sum_{j=1}^{m}u_{kj}\left(  x,h_{1},h_{2}%
,...,h_{j-1}\right)  \frac{\partial}{\partial h_{j}} \label{Liftati}%
\end{equation}
$\left(  k=0,1,...,n\right)  $, where the $u_{kj}$'s are polynomials of degree
at most $r-1$, such that the $\widetilde{X}_{k}$'s are free up to step $r$ and
such that $\left\{  \left(  \widetilde{X}_{\left[  I\right]  }\right)
_{\left(  x,h\right)  }\right\}  _{\left\vert I\right\vert \leqslant r}$ span
$\mathbb{R}^{p+m}\equiv\mathbb{R}^{N}$ for every $\left(  x,h\right)  \in
U\left(  x_{0}\right)  \times I$, where $I$ is a neighborhood of $0\in$
$\mathbb{R}^{m}$.
\end{theorem}

We do not repeat here the exact definition of \emph{free }vector fields, in
our weighted situations, because we will never use it explicitly.

An easy consequence of the structure (\ref{Liftati}) of the lifted vector
fields is that for any differentiable function $f\left(  x,h\right)  $ and any
smooth cutoff function $\varphi\left(  h\right)  $ we have
\begin{equation}
\int_{\mathbb{R}^{m}}\widetilde{X}_{k}\left[  f\left(  x,h\right)
\varphi\left(  h\right)  \right]  dh=\int_{\mathbb{R}^{m}}X_{k}f\left(
x,h\right)  \varphi\left(  h\right)  dh \label{X Xtilda}%
\end{equation}
since the integrals $\int_{\mathbb{R}^{m}}\frac{\partial}{\partial h_{j}%
}\left(  ...\right)  dh$ vanish.

We will denote by $\widetilde{d}$ the distance induced in $U\left(
x_{0}\right)  \times I$ by the lifted vector fields $\widetilde{X}_{i}$
$\left(  i=0,1,2,...,n\right)  ,$ as in Definition
\ref{Definition CC distance}, and by $\widetilde{B}(\eta,r)$ the corresponding
metric ball of center $\eta$ and radius $r$. We will also set
\[
\widetilde{L}=\sum_{i=1}^{n}\widetilde{X}_{i}^{2}+\widetilde{X}_{0}.
\]

Let us recall that a structure of \emph{homogeneous group }$\mathbb{G}$ on
$\mathbb{R}^{N}$ consists in a Lie group operation $\circ$ (which we think of
as \emph{translation}) such that the origin is the unit in the group and the
Euclidean opposite is the inverse in the group, and a one-parameter family
$\left\{  D\left(  \lambda\right)  \right\}  _{\lambda>0}$ of group
automorphisms (which we think of as \emph{dilations}), acting as follows:%
\begin{equation}
D\left(  \lambda\right)  \left(  u_{1},u_{2},...,u_{N}\right)  =\left(
\lambda^{\alpha_{1}}u_{1},\lambda^{\alpha_{2}}u_{2},...,\lambda^{\alpha_{N}%
}u_{N}\right)  , \label{dilation}%
\end{equation}
for some positive integers $\alpha_{1},\alpha_{2},...,\alpha_{N}.$ The sum of
these integers is called the \emph{homogeneous dimension }$Q$ of $\mathbb{G}.$

A homogeneous norm on $\mathbb{G}$ is any function $\left\Vert \cdot
\right\Vert :\mathbb{G\rightarrow}[0,\infty)$ such that

$\left\Vert u\right\Vert =0\Longleftrightarrow u=0$, $\left\Vert D\left(
\lambda\right)  u\right\Vert =\lambda\left\Vert u\right\Vert $ for any
$\lambda>0$,

$\left\Vert u_{1}\circ u_{2}\right\Vert \leqslant c\left(  \left\Vert
u_{1}\right\Vert +\left\Vert u_{2}\right\Vert \right)  $, $\left\Vert
u^{-1}\right\Vert \leqslant c\left\Vert u\right\Vert $ for any $u,u_{1}%
,u_{2}\in\mathbb{G}$.

Such a homogeneous norm naturally induces a distance $\left\Vert u_{1}%
^{-1}\circ u_{2}\right\Vert $ in $\mathbb{G}$; the (Lebesgue) measure of the
corresponding ball in $\mathbb{G}$ is translation invariant, and multiple of
$r^{Q}$. In the following we will use a fixed homogeneous norm on $\mathbb{G}$.

\begin{definition}
\label{defweight}(See \cite{BBP2}) We say that a vector field%
\[
R=\sum_{j=1}^{N}c_{j}\left(  u\right)  \partial_{u_{j}}%
\]
on the group $\mathbb{G}$ has \emph{weight} $\geqslant\beta$, for some
$\beta\in\mathbb{R}$, if%
\[
\left\vert c_{j}\left(  u\right)  \right\vert \leqslant c\left\Vert
u\right\Vert ^{\alpha_{j}+\beta}%
\]
for $u$ in a neighborhood of $0$.
\end{definition}

The second basic result proved in \cite{BBP2} is:

\begin{theorem}
[\textbf{Approximation by left invariant H\"{o}rmander's operator}%
]\label{Thm liftapprox} Let $x_{0}$, $U\left(  x_{0}\right)  $, and $I$ be as
in the lifting theorem. There exist a structure of homogeneous group
$\mathbb{G}$ on $\mathbb{R}^{N}$, $N=p+m$, a family of homogeneous left
invariant H\"{o}rmander's vector fields $Y_{0},Y_{1},Y_{2},...,Y_{n}$ on
$\mathbb{G}$ and an open set $V\subset U\left(  x_{0}\right)  \times I,$ such
that for any $\eta\in V$ there exists a smooth diffeomorphism $\Theta_{\eta}$
from a neighborhood of $\eta$ containing $V$ onto a neighborhood of the origin
in $\mathbb{G}$ such that $\Theta_{\eta}\left(  \xi\right)  $ and its first
order derivatives with respect to $\xi$ depend on $\eta$ in a $C^{\alpha}$
continuous way, locally uniformly in $\xi$, and for any smooth function
$f:\mathbb{G}\rightarrow\mathbb{R},$%
\begin{equation}
\widetilde{X}_{i}\left(  f\circ\Theta_{\eta}\right)  \left(  \xi\right)
=\left(  Y_{i}f+R_{i}^{\eta}f\right)  \left(  \Theta_{\eta}\left(  \xi\right)
\right)  \text{ }\forall\xi,\eta\in V \label{approx theta}%
\end{equation}
($i=0,1,...,n$) where $R_{i}^{\eta}$ are $C^{r-p_{i},\alpha}$ vector fields of
weight $\geqslant\alpha-p_{i}$. Moreover:

\begin{enumerate}
\item The following equivalences hold:%
\begin{equation}
c_{1}\left\vert \Theta_{\eta}\left(  \xi\right)  \right\vert \leqslant
c_{2}\widetilde{d}\left(  \eta,\xi\right)  \leqslant\left\Vert \Theta_{\eta
}\left(  \xi\right)  \right\Vert \leqslant c_{3}\widetilde{d}\left(  \eta
,\xi\right)  \text{ }\leqslant c_{4}\left\vert \Theta_{\eta}\left(
\xi\right)  \right\vert ^{1/r} \label{equiv Theta d}%
\end{equation}
for any $\xi,\eta\in V$. Also,%
\begin{equation}
c_{1}\rho^{Q}\leqslant\left\vert \widetilde{B}\left(  \xi,\rho\right)
\right\vert \leqslant c_{2}\rho^{Q}\text{ for any }\xi\in V,\rho\leqslant
\rho_{0} \label{meas B^tilde}%
\end{equation}
where $Q$ is the homogeneous dimension of the group $\mathbb{G}$ and
$c_{i},\ \rho_{0}$ are suitable positive constants.

\item The modulus of the Jacobian determinant of $\xi\mapsto\Theta_{\eta
}\left(  \xi\right)  $ has the form%
\begin{equation}
d\xi=c\left(  \eta\right)  \left(  1+O\left(  \left\Vert u\right\Vert \right)
\right)  du, \label{d xi}%
\end{equation}
where
\[
c\left(  \eta\right)  =\left\vert \det\left(  \left(  \widetilde{X}_{\left[
I\right]  }\right)  _{I\in B}\right)  _{\eta}\right\vert
\]
is a $C^{\alpha}$ function, bounded and bounded away from zero. (Here $B$ is
the set of multiindices giving the basis $\{\widetilde{X}_{\left[  I\right]
}\}_{I\in B}$ involved in the definition of the map $\Theta_{\eta}$.) More
explicitly, (\ref{d xi}) means that
\[
d\xi=c\left(  \eta\right)  \left[  1+\omega\left(  \eta,u\right)  \right]  du
\]
with $\left\vert \omega\left(  \eta,u\right)  \right\vert \leqslant
c\left\Vert u\right\Vert $, $\omega$ smooth in $u$ and $C^{\alpha}$ with
respect to $\eta,$ uniformly in $u$.
\end{enumerate}
\end{theorem}

The diffeomorphism $\Theta_{\eta}\left(  \cdot\right)  $ is defined as the
inverse of the exponential function
\[
u\mapsto E\left(  u,\eta\right)  =\exp\left(  \sum_{I\in B}u_{I}S_{\left[
I\right]  ,\eta}\right)  \left(  \eta\right)
\]
where the vector fields $S_{\left[  I\right]  ,\eta}$ are smooth vector fields
depending on $\eta$ in a $C^{\alpha}$ way (see \cite[\S 3]{BBP2} for the details).

In the next theorem we will show that both $E\left(  \cdot,\eta\right)  $ and
$\Theta_{\eta}\left(  \cdot\right)  $ have derivatives that depend on $\eta$
in a $C^{\alpha}$ way. As a consequence we will prove some properties of the
coefficients of the vector fields $R_{i}^{\eta}$.

\begin{theorem}
\label{Thm Calphatheta} ~

\begin{enumerate}
\item For every multi-index $\beta$ the derivatives $\frac{\partial
^{\left\vert \beta\right\vert }E}{\partial u^{\beta}}\left(  u,\eta\right)  $
and $\frac{\partial^{\left\vert \beta\right\vert }\Theta_{\eta}}{\partial
\xi^{\beta}}\left(  \xi\right)  $ depend on $\eta$ in a $C^{\alpha}$ way.

\item If $R_{i}^{\eta}=\sum_{k=1}^{N}c_{ik}^{\eta}\left(  u\right)
\partial_{u_{k}}$ then:

\begin{enumerate}
\item[i.] the functions $c_{ik}^{\eta}\left(  u\right)  $ (for $0\leqslant
i\leqslant n$) and $\frac{\partial c_{ik}^{\eta}}{\partial u_{j}}\left(
u\right)  $ (for $1\leqslant i\leqslant n$) depend on $\eta$ in a $C^{\alpha}$
way, locally uniformly with respect to $u$;

\item[ii.] the vector fields $\sum_{k=1}^{N}\frac{\partial c_{ik}^{\eta}%
}{\partial u_{j}}\left(  u\right)  \partial_{u_{k}}$ (for $1\leqslant
i\leqslant n,1\leqslant j\leqslant N$) have weight $\geqslant\alpha-2$.
\end{enumerate}
\end{enumerate}
\end{theorem}

\begin{proof}
We start with $\frac{\partial^{\left\vert \beta\right\vert }E}{\partial
u^{\beta}}$. We know that%
\[
E\left(  u,\eta\right)  =\gamma\left(  1,u,\eta\right)
\]
where $\gamma$ solves the Cauchy problem%
\[
\left\{
\begin{array}
[c]{l}%
\displaystyle\frac{d}{dt}\gamma\left(  t,u,\eta\right)  =\sum_{I\in B}%
u_{I}\left(  S_{\left[  I\right]  ,\eta}\right)  _{\gamma\left(
t,u,\eta\right)  }\\[2ex]%
\gamma\left(  0,u,\eta\right)  =\eta.
\end{array}
\right.
\]
For a fixed $\eta$ the solution $\gamma\left(  \cdot,\cdot,\eta\right)  $ is
smooth; moreover $\gamma$ depends on $\eta$ in a $C^{\alpha}$ way. Therefore%
\[
\left\{
\begin{array}
[c]{l}%
\displaystyle\frac{d}{dt}\displaystyle\frac{\partial\gamma}{\partial u_{J}%
}\left(  t,u,\eta\right)  =\sum_{I\in B}u_{I}\frac{\partial S_{\left[
I\right]  ,\eta}}{\partial\xi}\left(  \gamma\left(  t,u,\eta\right)  \right)
\frac{\partial\gamma}{\partial u_{J}}\left(  t,u,\eta\right)  +\left(
S_{\left[  J\right]  ,\eta}\right)  _{\gamma\left(  t,u,\eta\right)  }\\[2ex]%
\displaystyle\frac{\partial\gamma}{\partial u_{J}}\left(  0,u,\eta\right)  =0.
\end{array}
\right.
\]
Let now%
\begin{align*}
\omega\left(  t,u,\eta\right)   &  =\frac{\partial\gamma}{\partial u_{J}%
}\left(  t,u,\eta\right)  ,\\
A\left(  t,u,\eta\right)   &  =\sum_{I\in B}u_{I}~\frac{\partial S_{\left[
I\right]  ,\eta}}{\partial\xi}\left(  \gamma\left(  t,u,\eta\right)  \right)
,\\
B_{J}\left(  t,u,\eta\right)   &  =\left(  S_{\left[  J\right]  ,\eta}\right)
_{\gamma\left(  t,u,\eta\right)  }.
\end{align*}
Since $\left(  S_{\left[  J\right]  ,\eta}\right)  _{\xi}$ and $\frac{\partial
S_{\left[  I\right]  ,\eta}}{\partial\xi}\left(  \xi\right)  $ are smooth in
the $\xi$ variable and $C^{\alpha}$ in the $\eta$ variable, the functions
$A\left(  t,u,\eta\right)  $ and $B_{J}\left(  t,u,\eta\right)  $ are smooth
in $\left(  t,u\right)  $ and $C^{\alpha}$ in $\eta$. With the above notation,%
\[
\left\{
\begin{array}
[c]{l}%
\frac{d}{dt}\omega\left(  t,u,\eta\right)  =A\left(  t,u,\eta\right)
~\omega\left(  t,u,\eta\right)  +B_{J}\left(  t,u,\eta\right) \\[2ex]%
\omega\left(  0,u,\eta\right)  =0
\end{array}
\right.
\]
whence we readily see that $\omega$ is $C^{\alpha}$ in $\eta$. This shows that
$\frac{\partial E}{\partial u_{J}}\left(  u,\eta\right)  =\omega\left(
1,u,\eta\right)  $ has the same property. An iteration of this argument shows
that also $\frac{\partial^{\left\vert \beta\right\vert }E}{\partial u^{\beta}%
}$ is $C^{\alpha}$ with respect to $\eta$.

To prove the analogous result for $\frac{\partial^{\left\vert \beta\right\vert
}\Theta_{\eta}}{\partial\xi^{\beta}}\left(  \xi\right)  $ we differentiate
with respect to $\xi$ the identity%
\[
\xi=E\left(  \Theta_{\eta}\left(  \xi\right)  ,\eta\right)
\]
finding the matrix identity%
\[
I=\frac{\partial E}{\partial u}\left(  \Theta_{\eta}\left(  \xi\right)
,\eta\right)  \frac{\partial\Theta_{\eta}}{\partial\xi}\left(  \xi\right)
\]
and then%
\[
\frac{\partial\Theta_{\eta}}{\partial\xi}\left(  \xi\right)  =\left[
\frac{\partial E}{\partial u}\left(  \Theta_{\eta}\left(  \xi\right)
,\eta\right)  \right]  ^{-1}.
\]
Since $\frac{\partial E}{\partial u}\left(  \xi,\eta\right)  $ is smooth in
$\xi$ and $C^{\alpha}$ in $\eta$ and $\Theta_{\eta}\left(  \xi\right)  $ is
$C^{\alpha}$ in $\eta$, we get the desired result. An iteration of this
argument shows that also $\frac{\partial^{\left\vert \beta\right\vert }%
\Theta_{\eta}}{\partial\xi^{\beta}}\left(  \xi\right)  $ is $C^{\alpha}$ in
$\eta$.

To prove 2.i, let $f\left(  u\right)  =u_{k}$ and $g_{\eta}\left(  \xi\right)
=f\left(  \Theta_{\eta}\left(  \xi\right)  \right)  =\left(  \Theta_{\eta
}\left(  \xi\right)  \right)  _{k}$. Then, by (\ref{approx theta}), we have%
\[
\widetilde{X}_{i}g_{\eta}\left(  \xi\right)  =\left(  Y_{i}u_{k}\right)
\left(  \Theta_{\eta}\left(  \xi\right)  \right)  +c_{ik}^{\eta}\left(
\Theta_{\eta}\left(  \xi\right)  \right)  ,
\]
so that%
\begin{equation}
c_{ik}^{\eta}\left(  u\right)  =\widetilde{X}_{i}g_{\eta}\left(  \Theta_{\eta
}^{-1}\left(  u\right)  \right)  -Y_{i}u_{k}. \label{c_iketa}%
\end{equation}
Since $Y_{i}u_{k}$ is independent of $\eta$ it is enough to consider the term
$\widetilde{X}_{i}g_{\eta}\left(  \Theta_{\eta}^{-1}\left(  u\right)  \right)
$. Let us write%
\begin{align*}
&  \left\vert \widetilde{X}_{i}g_{\eta_{1}}\left(  \Theta_{\eta_{1}}%
^{-1}\left(  u\right)  \right)  -\widetilde{X}_{i}g_{\eta_{2}}\left(
\Theta_{\eta_{2}}^{-1}\left(  u\right)  \right)  \right\vert \\
&  \leqslant\left\vert \widetilde{X}_{i}g_{\eta_{1}}\left(  \Theta_{\eta_{1}%
}^{-1}\left(  u\right)  \right)  -\widetilde{X}_{i}g_{\eta_{1}}\left(
\Theta_{\eta_{2}}^{-1}\left(  u\right)  \right)  \right\vert +\left\vert
\widetilde{X}_{i}g_{\eta_{1}}\left(  \Theta_{\eta_{2}}^{-1}\left(  u\right)
\right)  -\widetilde{X}_{i}g_{\eta_{2}}\left(  \Theta_{\eta_{2}}^{-1}\left(
u\right)  \right)  \right\vert
\end{align*}
and
\[
\widetilde{X}_{i}g_{\eta}\left(  \xi\right)  =\sum\widetilde{b}_{ij}\left(
\xi\right)  \frac{\partial g_{\eta}}{\partial\xi_{j}}\left(  \xi\right)  .
\]
By Assumption A the coefficients $\widetilde{b}_{ij}$ are at least Lipschitz.
Since $\frac{\partial g_{\eta}}{\partial\xi_{j}}\left(  \xi\right)  $ are
smooth in $\xi$ we have%
\begin{align*}
\left\vert \widetilde{X}_{i}g_{\eta_{1}}\left(  \Theta_{\eta_{1}}^{-1}\left(
u\right)  \right)  -\widetilde{X}_{i}g_{\eta_{1}}\left(  \Theta_{\eta_{2}%
}^{-1}\left(  u\right)  \right)  \right\vert  &  \leqslant c\left\vert
\Theta_{\eta_{1}}^{-1}\left(  u\right)  -\Theta_{\eta_{2}}^{-1}\left(
u\right)  \right\vert \\
&  \leqslant c\left\vert \eta_{1}-\eta_{2}\right\vert ^{\alpha}.
\end{align*}
Also, since $\frac{\partial g_{\eta}}{\partial\xi_{j}}\left(  \xi\right)  $
depends on $\eta$ in a $C^{\alpha}$ way, we have%
\[
\left\vert \widetilde{X}_{i}g_{\eta_{1}}\left(  \Theta_{\eta_{2}}^{-1}\left(
u\right)  \right)  -\widetilde{X}_{i}g_{\eta_{2}}\left(  \Theta_{\eta_{2}%
}^{-1}\left(  u\right)  \right)  \right\vert \leqslant c\left\vert \eta
_{1}-\eta_{2}\right\vert ^{\alpha}.
\]
By (\ref{c_iketa}) this shows that $\eta\mapsto c_{ik}^{\eta}\left(  u\right)
$ is $C^{\alpha}$. Let us consider now
\[
\frac{\partial c_{ik}^{\eta}}{\partial u_{j}}\left(  u\right)  =\nabla_{\xi
}\left(  \widetilde{X}_{i}g_{\eta}\right)  \left(  \Theta_{\eta}^{-1}\left(
u\right)  \right)  \cdot\frac{\partial\Theta_{\eta}^{-1}\left(  u\right)
}{\partial u_{j}}-\frac{\partial}{\partial u_{j}}\left(  Y_{i}u_{k}\right)  .
\]
Since $\frac{\partial\Theta_{\eta}^{-1}\left(  u\right)  }{\partial u_{j}}$
depends on $\eta$ in a $C^{\alpha}$ way it is enough to study $\nabla_{\xi
}\left(  \widetilde{X}_{i}g_{\eta}\right)  \left(  \Theta_{\eta}^{-1}\left(
u\right)  \right)  $. We have%
\[
\frac{\partial}{\partial\xi_{\ell}}\widetilde{X}_{i}g_{\eta}\left(
\xi\right)  =\sum\frac{\partial\widetilde{b}_{ij}}{\partial\xi_{\ell}}\left(
\xi\right)  \frac{\partial g_{\eta}}{\partial\xi_{j}}\left(  \xi\right)
+\sum\widetilde{b}_{ij}\left(  \xi\right)  \frac{\partial^{2}g_{\eta}%
}{\partial\xi_{\ell}\partial\xi_{j}}\left(  \xi\right)  .
\]
By Assumption A, for $i\neq0$, $\widetilde{b}_{ij}\in C^{r-1,\alpha}$, so that
$\frac{\partial\widetilde{b}_{ij}}{\partial\xi_{\ell}}\in C^{r-2,\alpha}$.
Since for $r=2$ we have $\alpha=1$, $\frac{\partial\widetilde{b}_{ij}%
}{\partial\xi_{\ell}}$ is at least Lipschitz therefore $\frac{\partial
}{\partial\xi_{\ell}}\widetilde{X}_{i}g_{\eta}\left(  \xi\right)  $ is
Lipschitz with respect to $\xi$ and $C^{\alpha}$ with respect to $\eta$. The
proof now follows as in the previous case.

To show 2.ii, we first note that, from the proof of \cite[Prop. 3.5]{BBP2},
one reads that%
\begin{equation}
\left\vert c_{ik}^{\eta}\left(  u\right)  \right\vert \leqslant c\left\vert
u\right\vert ^{r-1+\alpha}. \label{C^1+alfa bound}%
\end{equation}
On the other hand, we know that $c_{ik}^{\eta}\left(  \cdot\right)  \in
C^{r-1,\alpha}$, hence the Taylor expansion of $c_{ik}^{\eta}\left(
\cdot\right)  $ and the bound (\ref{C^1+alfa bound}) imply%
\[
\left\vert \frac{\partial c_{ik}^{\eta}}{\partial u_{j}}\left(  u\right)
\right\vert \leqslant c\left\vert u\right\vert ^{r-2+\alpha}\leqslant
c\left\Vert u\right\Vert ^{\alpha_{k}-2+\alpha}.
\]
This implies 2.ii.
\end{proof}

The assertions on the \textquotedblleft weight\textquotedblright\ of the
remainders $R_{i}^{\eta}$ in point 2.ii of the previous theorem in particular
mean that, whenever $f:\mathbb{G}\rightarrow\mathbb{R}$ is homogeneous of
degree $-k$ (with respect to the dilations $D\left(  \lambda\right)  $), then
near the origin%
\begin{equation}
\left\vert R_{i}^{\eta}f\left(  u\right)  \right\vert \leqslant\frac
{c}{\left\Vert u\right\Vert ^{k+p_{i}-\alpha}}\text{ for }i=0,1,...,n.
\label{remainder}%
\end{equation}

Moreover, the statements 2.i and 2.ii in the above theorem immediately imply:

\begin{corollary}
\label{Coroll second order remainder}All the differential operators
$D_{ij}^{\eta}$ defined by the compositions%
\[
Y_{j}R_{i}^{\eta}\text{, }R_{i}^{\eta}Y_{j}\text{, }R_{i}^{\eta}R_{j}^{\eta
}\text{ }\left(  i,j=1,2,...,n\right)
\]
satisfy the bound%
\begin{equation}
\left\vert D_{ij}^{\eta}f\left(  u\right)  \right\vert \leqslant\frac
{c}{\left\Vert u\right\Vert ^{k+2-\alpha}} \label{second order remainder}%
\end{equation}
for $u$ in a neighborhood of the origin, whenever $f:\mathbb{G}\rightarrow
\mathbb{R}$ is $D\left(  \lambda\right)  $-homogeneous of degree $-k$. Also,
the coefficients of $D_{ij}^{\eta}$ depend on $\eta$ in a $C^{\alpha}$ way.
\end{corollary}

Next, we have to point out some properties related to the volume of metric balls.

\begin{remark}
In contrast with (\ref{meas B^tilde}), if we apply the estimates
(\ref{volume}) for $x$ in the neighborhood $U\left(  x_{0}\right)  $ where the
lifting theorem applies, we find the following useful inequalities%
\begin{equation}
c_{1}\left(  \frac{r_{1}}{r_{2}}\right)  ^{p}\leqslant\frac{\left\vert
B\left(  x,r_{1}\right)  \right\vert }{\left\vert B\left(  x,r_{2}\right)
\right\vert }\leqslant c_{2}\left(  \frac{r_{1}}{r_{2}}\right)  ^{Q}
\label{volume inequalities}%
\end{equation}
for any $r_{1},r_{2}$ with $\rho_{0}>r_{1}>r_{2}>0.$ This follows from the
inequalities $p\leqslant\left\vert \mathcal{I}\right\vert \leqslant Q,$
holding for each $\mathcal{I}$ in the sums appearing in (\ref{volume}).
\end{remark}

The following nonsmooth version of a well-known result by S\'anchez-Calle
\cite{SC} and Nagel, Stein,\ Wainger \cite{NSW}, has been proved in
\cite{BBP}, and allows one to compare the volume of balls in the lifted and in
the original variables.

\begin{theorem}
\label{Thm Sanchez-Calle} Let $x_{0}$, $U\left(  x_{0}\right)  $, and $I$ be
as in the lifting theorem. Then, up to possibly shrinking the set $U\left(
x_{0}\right)  $, there exist positive constants $c_{1},c_{2},\rho_{0},$ and
$\delta\in\left(  0,1\right)  $ such that for any $\left(  x,h\right)  \in
U\left(  x_{0}\right)  \times I$, any $y\in B\left(  x,\delta\rho\right)  ,$
$0<\rho<\rho_{0},$ we have
\begin{equation}
c_{1}\frac{\left\vert \widetilde{B}\left(  \left(  x,h\right)  ,\rho\right)
\right\vert }{\left\vert B\left(  x,\rho\right)  \right\vert }\leqslant
\int_{\mathbb{R}^{m}}\chi_{\widetilde{B}\left(  \left(  x,h\right)
,\rho\right)  }\left(  y,s\right)  ds\leqslant c_{2}\frac{\left\vert
\widetilde{B}\left(  \left(  x,h\right)  ,\rho\right)  \right\vert
}{\left\vert B\left(  x,\rho\right)  \right\vert }. \label{Sanchez-Calle}%
\end{equation}
Actually the second inequality holds for every $y\in U\left(  x_{0}\right)  $.
Also, the projection of $\widetilde{B}\left(  \left(  x,h\right)
,\rho\right)  $ on $\mathbb{R}^{p}$ is exactly $B\left(  x,\rho\right)  .$
\end{theorem}

\begin{remark}
Actually (\ref{Sanchez-Calle}) is stated in \cite{BBP} when $X_{0}$ is
lacking; however, the proof given in \cite{BBP} relies on the analog result
which holds for smooth H\"{o}rmander's vector fields. In turn, the result for
smooth H\"{o}rmander's vector fields has been proved in \cite{SC} when $X_{0}$
is lacking, while just one of the two inequalities in (\ref{Sanchez-Calle})
has been proved in \cite{NSW} also in presence of $X_{0};$ however, as shown
in \cite{J}, the same argument used in \cite{NSW} allows one to prove also the
other inequality. Hence (\ref{Sanchez-Calle}) holds in the smooth case also in
presence of $X_{0},$ and the same is true for nonsmooth H\"{o}rmander's vector fields.
\end{remark}

\bigskip

\textbf{Notation.\label{Notation} }Throughout the paper we will handle four
types of vector fields, which will be regarded as differential operators
acting on different variables. The vector fields%
\[
Y_{i}\text{ and }R_{i}^{\eta}%
\]
act on the variable $u$ in the group $\mathbb{G}$ (that is, they are written
in the coordinates $u$), and we will often have $u=\Theta_{\eta}\left(
\xi\right)  $; moreover, the coefficients of the $R_{i}^{\eta}$'s depend on
the variable $\eta$ as a parameter. The vector fields
\[
X_{i}\text{ and }\widetilde{X}_{i}%
\]
act on $\mathbb{R}^{p}$, $\mathbb{R}^{N}$, respectively; they are often
applied on a function of two variables, and in this case, they will
\emph{always} be seen as acting on the \emph{first }variable, which in
$\mathbb{R}^{p}$ is called $x$ and in $\mathbb{R}^{N}$ is called $\xi=\left(
x,h\right)  $. For instance,%
\begin{align*}
X_{i}f\left(  x,y\right)   &  =X_{i}\left[  f\left(  \cdot,y\right)  \right]
\left(  x\right)  ;\\
\widetilde{X}_{i}f\left(  \xi,\eta\right)   &  =\widetilde{X}_{i}\left[
f\left(  \cdot,\eta\right)  \right]  \left(  \xi\right)  .
\end{align*}
These conventions will be applied consistently throughout the paper.

\section{Geometric estimates\label{sec geometric}}

In this section we establish some estimates which relate the growth of some
kernels defined in the lifted space with that of kernels defined in the
original space $\mathbb{R}^{p}$. The fact that the volume of metric balls in
$\mathbb{R}^{p}$ does not behave like a fixed power of the radius makes these
estimates delicate to be proved. These results will be fundamental throughout
the following.

Let $\Omega\subset\mathbb{R}^{p}$ be a domain where our assumptions are
satisfied, $\Omega^{\prime}\Subset\Omega,x_{0}\in\Omega^{\prime},$ $U\left(
x_{0}\right)  =B\left(  x_{0},r_{0}\right)  \Subset\Omega$ a neighborhood of
$x_{0}$ where the lifting and approximation theorem is applicable, $R$ a
number small enough so that $B\left(  x,2R\right)  \Subset\Omega$ for any
$x\in U\left(  x_{0}\right)  $. Let $\varphi,\psi\in C_{0}^{\infty}\left(
\mathbb{R}^{m}\right)  $ be supported in the neighborhood $I$ of the origin
which appears in Theorem \ref{Thm lifting}. Shrinking if necessary \ $U\left(
x_{0}\right)  $ and the supports of $\varphi,\psi,$ we can assume that
$4r_{0}\leqslant R$ and%
\[
\widetilde{d}\left(  \left(  x,h\right)  ,\left(  y,k\right)  \right)  <R
\]
for $x,y\in U\left(  x_{0}\right)  $ and $h,k$ in the supports of
$\varphi,\psi,$ respectively. With this notation, we have the following:

\begin{lemma}
\label{Lemma NSW}For every $\beta\in\mathbb{R}$ there exists $c>0$ such that%
\[
\int_{\mathbb{R}^{m}}\int_{\mathbb{R}^{m}}\frac{\psi\left(  h\right)
}{\left\Vert \Theta_{\left(  y,k\right)  }\left(  x,h\right)  \right\Vert
^{Q-\beta}}dh\,\varphi\left(  k\right)  dk\leqslant c\int_{d\left(
x,y\right)  }^{R}\frac{r^{\beta-1}}{\left\vert B\left(  x,r\right)
\right\vert }dr
\]
for any $x,y\in U\left(  x_{0}\right)  $.
\end{lemma}

This is just \cite[Thm.\ 5]{NSW}, in our nonsmooth context. It can be proved
at the same way using Theorem \ref{Thm Sanchez-Calle}.

It is convenient to give a name to the function which appears in the previous
Lemma, since it will be a central object throughout the following.

\begin{definition}
For $x,y\in U\left(  x_{0}\right)  ,x\neq y\ $and $\beta\in\mathbb{R},$ let%
\begin{equation}
\phi_{\beta}\left(  x,y\right)  =\int_{d\left(  x,y\right)  }^{R}%
\frac{r^{\beta-1}}{\left\vert B\left(  x,r\right)  \right\vert }dr.
\label{Phi beta}%
\end{equation}

\end{definition}

The estimate in the previous lemma is made more readable by the next:

\begin{lemma}
\label{Lemma nonintegrale}For $x,y\in U\left(  x_{0}\right)  ,x\neq y$, the
following inequalities hold:%
\[
\phi_{\beta}\left(  x,y\right)  \leqslant\left\{
\begin{array}
[c]{ll}%
\displaystyle c\frac{d\left(  x,y\right)  ^{\beta}}{\left\vert B\left(
x,d\left(  x,y\right)  \right)  \right\vert } & \text{for }\beta<p\\
& \\
\displaystyle c\frac{d\left(  x,y\right)  ^{p}}{\left\vert B\left(  x,d\left(
x,y\right)  \right)  \right\vert }\log\frac{R}{d\left(  x,y\right)  } &
\text{for }\beta=p\\
& \\
\displaystyle c\frac{d\left(  x,y\right)  ^{p}}{\left\vert B\left(  x,d\left(
x,y\right)  \right)  \right\vert }R^{\beta-p} & \text{for }\beta>p
\end{array}
\right.
\]
(recall that $p$ is the Euclidean dimension of the space of variables $x,y$).
\end{lemma}

\begin{proof}
By (\ref{volume inequalities}) we have:%
\[
\left\vert B\left(  x,r\right)  \right\vert \geqslant c\left\vert B\left(
x,d\left(  x,y\right)  \right)  \right\vert \left(  \frac{r}{d\left(
x,y\right)  }\right)  ^{p}\text{ for }d\left(  x,y\right)  <r<R.
\]
Hence, for $\beta<p,$%
\begin{align*}
\int_{d\left(  x,y\right)  }^{R}\frac{r^{\beta-1}}{\left\vert B\left(
x,r\right)  \right\vert }dr  &  \leqslant c\frac{d\left(  x,y\right)  ^{p}%
}{\left\vert B\left(  x,d\left(  x,y\right)  \right)  \right\vert }%
\int_{d\left(  x,y\right)  }^{R}\frac{r^{\beta-1}}{r^{p}}dr\\
&  =c\frac{d\left(  x,y\right)  ^{p}}{\left\vert B\left(  x,d\left(
x,y\right)  \right)  \right\vert }\left[  \frac{d\left(  x,y\right)
^{\beta-p}-R^{\beta-p}}{p-\beta}\right] \\
&  \leqslant c\frac{d\left(  x,y\right)  ^{p}}{\left\vert B\left(  x,d\left(
x,y\right)  \right)  \right\vert }d\left(  x,y\right)  ^{\beta-p}%
=c\frac{d\left(  x,y\right)  ^{\beta}}{\left\vert B\left(  x,d\left(
x,y\right)  \right)  \right\vert }.
\end{align*}
The proof in other cases is analogous.
\end{proof}

By a standard computation the previous lemma immediately implies

\begin{corollary}
\label{coroll Psi beta}For any $\beta>0$ the following bounds hold:%
\[
\Psi_{\beta}\left(  x,r\right)  \equiv\int_{d\left(  x,y\right)  <r}%
\phi_{\beta}\left(  x,y\right)  dy\leqslant\left\{
\begin{array}
[c]{ll}%
cr^{\beta} & \text{if }\beta<p\\
c_{\varepsilon}r^{\beta-\varepsilon} & \text{if }\beta=p\text{ (any
}\varepsilon>0\text{)}\\
cr^{p} & \text{if }\beta>p
\end{array}
\right.
\]
where in case $\beta<p$ the constant $c$ is independent of $R$. In any case,
$\Psi_{\beta}\left(  x,r\right)  \rightarrow0$ as $r\rightarrow0,$ uniformly
in $x$.
\end{corollary}

\begin{theorem}
\label{Thm phi alfa}We have the following:

1) there exists $c>0$ such that for every $\beta,\gamma>0$ :%
\[
\int_{U\left(  x_{0}\right)  }\phi_{\beta}\left(  x,y\right)  \phi_{\gamma
}\left(  y,z\right)  dy\leqslant c\left(  \frac{1}{\beta}+\frac{1}{\gamma
}\right)  \phi_{\beta+\gamma}\left(  x,z\right)
\]
for every $x,z\in U\left(  x_{0}\right)  $.

2) there exists $c>0$ such that for every $\gamma>Q$%
\[
\phi_{\gamma}\left(  x,y\right)  \leqslant cR^{\gamma-Q}%
\]
for every $x,y\in U\left(  x_{0}\right)  $. (Recall that $Q$ is the
homogeneous dimension of the group in the lifted space).
\end{theorem}

\begin{remark}
Comparing point 2) in the statement of the above theorem with the case
$\beta>p$ in the statement of Lemma \ref{Lemma nonintegrale}, one can see why
in our context it is necessary to work with the functions $\phi_{\beta}$
instead of the simpler functions
\[
\psi_{\beta}\left(  x,y\right)  =\frac{d\left(  x,y\right)  ^{\beta}%
}{\left\vert B\left(  x,d\left(  x,y\right)  \right)  \right\vert }.
\]
The point is that the functions $\phi_{\beta}$ are bounded for $\beta$ large
enough, so that an iterative construction involving integrals of the kind
\[
\int_{U\left(  x_{0}\right)  }\phi_{\beta}\left(  x,y\right)  \phi_{\gamma
}\left(  y,z\right)  dy
\]
ends with a bounded function. On the other hand, if one tries to prove an
analog of the previous theorem for the $\psi_{\beta}$'s$,$ the best upper
bound one can find is%
\[
\frac{d\left(  x,y\right)  ^{p}}{\left\vert B\left(  x,d\left(  x,y\right)
\right)  \right\vert }%
\]
which is generally unbounded, because $\left\vert B\left(  x,d\left(
x,y\right)  \right)  \right\vert \geqslant cd\left(  x,y\right)  ^{Q}$ with
$Q>p.$ This \textquotedblleft dimensional gap\textquotedblright\ occurs in our
general context since the measure of a ball does not behave like a fixed power
of the radius.
\end{remark}

\begin{proof}
We start by noting that%
\[
\phi_{\beta}\left(  x,y\right)  \leqslant c\phi_{\beta}\left(  y,x\right)  .
\]
Indeed,
\[
\phi_{\beta}\left(  x,y\right)  =\int_{d\left(  x,y\right)  }^{R}%
\frac{r^{\beta-1}}{\left\vert B\left(  x,r\right)  \right\vert }dr\leqslant
c\int_{d\left(  x,y\right)  }^{R}\frac{r^{\beta-1}}{\left\vert B\left(
y,r\right)  \right\vert }dr=c\phi_{\beta}\left(  y,x\right)
\]
since for $d\left(  x,y\right)  <r$ we have $B\left(  y,r\right)  \subset
B\left(  x,2r\right)  ;$ for $x\in U\left(  x_{0}\right)  $ and $r\leqslant R$
the doubling condition is applicable and gives%
\[
\left\vert B\left(  y,r\right)  \right\vert \leqslant\left\vert B\left(
x,2r\right)  \right\vert \leqslant c\left\vert B\left(  x,r\right)
\right\vert .
\]
Also, since $R\geqslant4r_{0}\geqslant2d\left(  x,y\right)  $ for any $x,y\in
U\left(  x_{0}\right)  $, we have%
\begin{align}
\int_{\frac{1}{2}d\left(  x,y\right)  }^{R}\frac{r^{\beta-1}}{\left\vert
B\left(  x,r\right)  \right\vert }dr  &  =\int_{\frac{1}{2}d\left(
x,y\right)  }^{R/2}\frac{r^{\beta-1}}{\left\vert B\left(  x,r\right)
\right\vert }dr+\int_{R/2}^{R}\frac{r^{\beta-1}}{\left\vert B\left(
x,r\right)  \right\vert }dr\nonumber\\
&  \leqslant\frac{c}{2^{\beta}}\int_{d\left(  x,y\right)  }^{R}\frac
{r^{\beta-1}}{\left\vert B\left(  x,r\right)  \right\vert }dr+\int_{d\left(
x,y\right)  }^{R}\frac{r^{\beta-1}}{\left\vert B\left(  x,r\right)
\right\vert }dr\nonumber\\
&  \leqslant c\int_{d\left(  x,y\right)  }^{R}\frac{r^{\beta-1}}{\left\vert
B\left(  x,r\right)  \right\vert }dr \label{comparable phi}%
\end{align}
where $c$ is independent of $\beta$. Now,%
\begin{align*}
&  \int_{U\left(  x_{0}\right)  }\phi_{\beta}\left(  x,y\right)  \phi_{\gamma
}\left(  y,z\right)  dy\\
&  =\int_{d\left(  x,y\right)  <\frac{1}{2}d\left(  x,z\right)  }\left(
...\right)  dy+\int_{d\left(  z,y\right)  <\frac{1}{2}d\left(  x,z\right)
}\left(  ...\right)  dy+\int_{\substack{d\left(  x,y\right)  \geqslant\frac
{1}{2}d\left(  x,z\right)  \\d\left(  z,y\right)  \geqslant\frac{1}{2}d\left(
x,z\right)  }}\left(  ...\right)  dy\\
&  \equiv I+II+III.
\end{align*}

To bound $I$ we note that $\frac{1}{2}d\left(  y,z\right)  \leqslant d\left(
x,z\right)  \leqslant2d\left(  y,z\right)  $, hence%
\begin{align*}
I  &  =\int_{d\left(  x,y\right)  <\frac{1}{2}d\left(  x,z\right)  }\left(
\int_{d\left(  x,y\right)  }^{R}\frac{r^{\beta-1}}{\left\vert B\left(
x,r\right)  \right\vert }dr\int_{d\left(  y,z\right)  }^{R}\frac{s^{\gamma-1}%
}{\left\vert B\left(  z,s\right)  \right\vert }ds\right)  dy\\
&  \leqslant c\int_{\frac{1}{2}d\left(  x,z\right)  }^{R}\frac{s^{\gamma-1}%
}{\left\vert B\left(  z,s\right)  \right\vert }ds\int_{d\left(  x,y\right)
<\frac{1}{2}d\left(  x,z\right)  }\left(  \int_{d\left(  x,y\right)  }%
^{R}\frac{r^{\beta-1}}{\left\vert B\left(  x,r\right)  \right\vert }dr\right)
dy
\end{align*}
and, applying Fubini's theorem in the integral in $drdy$,
\[
d\left(  x,y\right)  <\frac{1}{2}d\left(  x,z\right)  ,d\left(  x,y\right)
<r<R\Longrightarrow0<r<R,d\left(  x,y\right)  <\min\left(  \frac{1}{2}d\left(
x,z\right)  ,r\right)  ,
\]
we have that
\begin{align*}
I  &  \leqslant c\int_{\frac{1}{2}d\left(  x,z\right)  }^{R}\frac{s^{\gamma
-1}}{\left\vert B\left(  z,s\right)  \right\vert }ds\int_{0}^{R}\frac
{r^{\beta-1}}{\left\vert B\left(  x,r\right)  \right\vert }\left(
\int_{d\left(  x,y\right)  <\frac{1}{2}d\left(  x,z\right)  \wedge
r}dy\right)  dr\\
&  =c\int_{\frac{1}{2}d\left(  x,z\right)  }^{R}\frac{s^{\gamma-1}}{\left\vert
B\left(  z,s\right)  \right\vert }ds\left\{  \int_{0}^{\frac{1}{2}d\left(
x,z\right)  }\frac{r^{\beta-1}}{\left\vert B\left(  x,r\right)  \right\vert
}\left(  \int_{d\left(  x,y\right)  <r}dy\right)  dr\right. \\
&  ~~~~~~~+\left.  \int_{\frac{1}{2}d\left(  x,z\right)  }^{R}\frac
{r^{\beta-1}}{\left\vert B\left(  x,r\right)  \right\vert }\left(
\int_{d\left(  x,y\right)  <\frac{1}{2}d\left(  x,z\right)  }dy\right)
dr\right\} \\
&  \leqslant c\int_{\frac{1}{2}d\left(  x,z\right)  }^{R}\frac{s^{\gamma-1}%
}{\left\vert B\left(  z,s\right)  \right\vert }ds\left\{  \int_{0}^{\frac
{1}{2}d\left(  x,z\right)  }r^{\beta-1}dr+\int_{\frac{1}{2}d\left(
x,z\right)  }^{R}\frac{r^{\beta-1}}{\left\vert B\left(  x,r\right)
\right\vert }\left\vert B\left(  x,d\left(  x,z\right)  \right)  \right\vert
dr\right\} \\
&  \equiv I_{A}+I_{B}.
\end{align*}
In turn,%
\[
I_{A}=\frac{c}{\beta}\left(  \frac{1}{2}d\left(  x,z\right)  \right)  ^{\beta
}\int_{\frac{1}{2}d\left(  x,z\right)  }^{R}\frac{s^{\gamma-1}}{\left\vert
B\left(  z,s\right)  \right\vert }ds\leqslant\frac{c}{\beta}\int_{d\left(
x,z\right)  }^{R}\frac{s^{\beta+\gamma-1}}{\left\vert B\left(  z,s\right)
\right\vert }ds
\]
and, using the notation $B(x;z)=B\left(  x,d\left(  x,z\right)  \right)  $ and
applying (\ref{comparable phi}),%
\begin{align*}
I_{B}  &  \leqslant c\left\vert B\left(  x;z\right)  \right\vert
\int_{d\left(  x,z\right)  }^{R}\frac{s^{\gamma-1}}{\left\vert B\left(
z,s\right)  \right\vert }ds\int_{d\left(  x,z\right)  }^{R}\frac{r^{\beta-1}%
}{\left\vert B\left(  x,r\right)  \right\vert }dr\\
&  =c\left\vert B\left(  x;z\right)  \right\vert \int_{d\left(  x,z\right)
}^{R}\frac{s^{\gamma-1}}{\left\vert B\left(  z,s\right)  \right\vert }\left(
\int_{d\left(  x,z\right)  }^{s}\frac{r^{\beta-1}}{\left\vert B\left(
x,r\right)  \right\vert }dr+\int_{s}^{R}\frac{r^{\beta-1}}{\left\vert B\left(
x,r\right)  \right\vert }dr\right)  ds\\
&  \equiv I_{B_{1}}+I_{B_{2}},
\end{align*}
where, since in $I_{B_{2}}$ we have $d\left(  x,z\right)  <s<r$, then
$\left\vert B\left(  x;z\right)  \right\vert \leqslant\left\vert B\left(
z,s\right)  \right\vert $ and therefore
\begin{equation}
I_{B_{2}}\leqslant c\int_{d\left(  x,z\right)  }^{R}s^{\gamma-1}\left(
\int_{s}^{R}\frac{r^{\beta-1}}{\left\vert B\left(  x,r\right)  \right\vert
}dr\right)  ds \label{I B 2}%
\end{equation}
applying Fubini's theorem:

$d\left(  x,z\right)  <s<R,s<r<R\Longrightarrow d\left(  x,z\right)
<r<R,d\left(  x,z\right)  <s<r$%
\begin{align*}
&  =c\int_{d\left(  x,z\right)  }^{R}\frac{r^{\beta-1}}{\left\vert B\left(
x,r\right)  \right\vert }\left(  \int_{d\left(  x,z\right)  }^{r}s^{\gamma
-1}ds\right)  dr\\
&  \leqslant c\int_{d\left(  x,z\right)  }^{R}\frac{r^{\beta-1}}{\left\vert
B\left(  x,r\right)  \right\vert }\left(  \int_{0}^{r}s^{\gamma-1}ds\right)
dr\\
&  =\frac{c}{\gamma}\int_{d\left(  x,z\right)  }^{R}\frac{r^{\beta+\gamma-1}%
}{\left\vert B\left(  x,r\right)  \right\vert }dr.
\end{align*}
As to $I_{B_{1}},$ applying once more Fubini's theorem,
\[
d\left(  x,z\right)  <s<R,d\left(  x,z\right)  <r<s\Longrightarrow d\left(
x,z\right)  <r<R,r<s<R,
\]
we have
\[
I_{B_{1}}=c\left\vert B\left(  x;z\right)  \right\vert \int_{d\left(
x,z\right)  }^{R}\frac{r^{\beta-1}}{\left\vert B\left(  x,r\right)
\right\vert }\left(  \int_{r}^{R}\frac{s^{\gamma-1}}{\left\vert B\left(
z,s\right)  \right\vert }ds\right)  dr
\]
since $d\left(  x,z\right)  <r$ implies $\left\vert B\left(  x;z\right)
\right\vert \leqslant\left\vert B\left(  x,r\right)  \right\vert $,%
\[
\leqslant c\int_{d\left(  x,z\right)  }^{R}r^{\beta-1}\left(  \int_{r}%
^{R}\frac{s^{\gamma-1}}{\left\vert B\left(  z,s\right)  \right\vert
}ds\right)  dr
\]
and this can be handled as $I_{B_{2}}$ (see (\ref{I B 2})).

We have therefore proved that $I$ satisfies the desired bound. The term $II$
can be handled analogously (by symmetry).

Let us come to the bound on $III$. Since%
\[
d\left(  x,y\right)  \geqslant\frac{1}{2}d\left(  x,z\right)  ,\text{
}d\left(  z,y\right)  \geqslant\frac{1}{2}d\left(  x,z\right)  \text{ and
}d\left(  x,y\right)  <r<R,d\left(  y,z\right)  <s<R
\]
imply%
\[
\frac{1}{2}d\left(  x,z\right)  <r<R,\frac{1}{2}d\left(  x,z\right)  <s<R
\]
and
\[
\text{ }\frac{1}{2}d\left(  x,z\right)  <d\left(  x,y\right)  <r,\frac{1}%
{2}d\left(  x,z\right)  <d\left(  y,z\right)  <s,
\]
applying Fubini's theorem in the triple integral gives%
\begin{align*}
III  &  =\int_{\frac{1}{2}d\left(  x,z\right)  }^{R}\frac{r^{\beta-1}%
}{\left\vert B\left(  x,r\right)  \right\vert }\int_{\frac{1}{2}d\left(
x,z\right)  }^{R}\frac{s^{\gamma-1}}{\left\vert B\left(  z,s\right)
\right\vert }\left(  \int_{\substack{\frac{1}{2}d\left(  x,z\right)  <d\left(
x,y\right)  <r\\\frac{1}{2}d\left(  x,z\right)  <d\left(  y,z\right)
<s}}dy\right)  dsdr\\
&  \leqslant\int_{\frac{1}{2}d\left(  x,z\right)  }^{R}\frac{r^{\beta-1}%
}{\left\vert B\left(  x,r\right)  \right\vert }\left(  \int_{\frac{1}%
{2}d\left(  x,z\right)  }^{R}\frac{s^{\gamma-1}}{\left\vert B\left(
z,s\right)  \right\vert }\left\vert B\left(  x,r\right)  \cap B\left(
z,s\right)  \right\vert ds\right)  dr\\
&  =\int_{\frac{1}{2}d\left(  x,z\right)  }^{R}\frac{r^{\beta-1}}{\left\vert
B\left(  x,r\right)  \right\vert }\left(  \int_{\frac{1}{2}d\left(
x,z\right)  }^{r}+\int_{r}^{R}\right)  \left(  \frac{s^{\gamma-1}}{\left\vert
B\left(  z,s\right)  \right\vert }\left\vert B\left(  x,r\right)  \cap
B\left(  z,s\right)  \right\vert ds\right)  dr\\
&  \equiv III_{A}+III_{B}.
\end{align*}
Now,%
\begin{align*}
III_{A}  &  \leqslant\int_{\frac{1}{2}d\left(  x,z\right)  }^{R}\frac
{r^{\beta-1}}{\left\vert B\left(  x,r\right)  \right\vert }\left(  \int
_{\frac{1}{2}d\left(  x,z\right)  }^{r}s^{\gamma-1}ds\right)  dr\\
&  \leqslant\int_{\frac{1}{2}d\left(  x,z\right)  }^{R}\frac{r^{\beta-1}%
}{\left\vert B\left(  x,r\right)  \right\vert }\left(  \int_{0}^{r}%
s^{\gamma-1}ds\right)  dr\\
&  =\frac{1}{\gamma}\int_{\frac{1}{2}d\left(  x,z\right)  }^{R}\frac
{r^{\beta+\gamma-1}}{\left\vert B\left(  x,r\right)  \right\vert }dr\\
&  \leqslant\frac{c}{\gamma}\int_{d\left(  x,z\right)  }^{R}\frac
{r^{\beta+\gamma-1}}{\left\vert B\left(  x,r\right)  \right\vert }dr
\end{align*}
by (\ref{comparable phi}). As to $III_{B}$, since%
\[
\frac{1}{2}d\left(  x,z\right)  <r<R,r<s<R\Longrightarrow\frac{1}{2}d\left(
x,z\right)  <s<R,\frac{1}{2}d\left(  x,z\right)  <r<s,
\]
by Fubini's theorem,%
\begin{align*}
III_{B}  &  \leqslant\int_{\frac{1}{2}d\left(  x,z\right)  }^{R}r^{\beta
-1}\left(  \int_{r}^{R}\frac{s^{\gamma-1}}{\left\vert B\left(  z,s\right)
\right\vert }ds\right)  dr\\
&  =\int_{\frac{1}{2}d\left(  x,z\right)  }^{R}\frac{s^{\gamma-1}}{\left\vert
B\left(  z,s\right)  \right\vert }\left(  \int_{\frac{1}{2}d\left(
x,z\right)  }^{s}r^{\beta-1}dr\right)  ds\\
&  \leqslant\int_{\frac{1}{2}d\left(  x,z\right)  }^{R}\frac{s^{\gamma-1}%
}{\left\vert B\left(  z,s\right)  \right\vert }\left(  \int_{0}^{s}r^{\beta
-1}dr\right)  ds\\
&  =\frac{1}{\beta}\int_{\frac{1}{2}d\left(  x,z\right)  }^{R}\frac
{s^{\beta+\gamma-1}}{\left\vert B\left(  z,s\right)  \right\vert }ds\\
&  \leqslant\frac{c}{\beta}\int_{d\left(  x,z\right)  }^{R}\frac
{s^{\beta+\gamma-1}}{\left\vert B\left(  z,s\right)  \right\vert }ds.
\end{align*}
This shows that also $III$ satisfies the desired bound, and point 1 of the
theorem is proved.

As to point 2, the volume estimate (\ref{volume inequalities}) gives, for any
$r<R,$%
\[
\left\vert B\left(  x,r\right)  \right\vert \geqslant c\left(  \frac{r}%
{R}\right)  ^{Q}\left\vert B\left(  x,R\right)  \right\vert \geqslant cr^{Q}%
\]
since
\[
\inf_{x\in\Omega^{\prime}}\left\vert B\left(  x,R\right)  \right\vert
\geqslant c>0
\]
as easily follows by the doubling condition. Then, for any $\gamma>Q$,
\[
\int_{d\left(  x,y\right)  }^{R}\frac{r^{\gamma-1}}{\left\vert B\left(
x,r\right)  \right\vert }dr\leqslant\int_{d\left(  x,y\right)  }^{R}%
\frac{r^{\gamma-1}}{cr^{Q}}dr\leqslant c\int_{0}^{R}r^{\gamma-1-Q}%
dr=cR^{\gamma-Q}.
\]

\end{proof}

In order to deal with continuity matters of the next sections, we will need
the following

\begin{proposition}
\label{Lemma joint continuity}Let $T\subset U\left(  x_{0}\right)  $ be an
open set.

(i) Let $f\left(  x,y\right)  ,g\left(  x,y\right)  $ be two functions defined
in $T\times T$ satisfying%
\begin{align*}
\left\vert f\left(  x,y\right)  \right\vert  &  \leqslant c\phi_{\beta}\left(
x,y\right)  ;\\
\left\vert g\left(  x,y\right)  \right\vert  &  \leqslant c\phi_{\gamma
}\left(  x,y\right)  ,
\end{align*}
for some $\beta,\gamma>0$ and any $x,y\in T,x\neq y$. Assume that both $f$ and
$g$ are continuous in the joint variables $\left(  x,y\right)  $ for $x\neq
y$. Then the function%
\[
h\left(  x,y\right)  =\int_{T}f\left(  x,z\right)  g\left(  z,y\right)  dz
\]
is jointly continuous in $T\times T$ for $x\neq y$.

(ii) Let $f\left(  x,y\right)  $ be a function defined in $T\times T$
satisfying%
\[
\left\vert f\left(  x,y\right)  \right\vert \leqslant c\frac{d\left(
x,y\right)  ^{\beta}}{\left\vert B\left(  x,d\left(  x,y\right)  \right)
\right\vert }%
\]
for some $\beta>0,$ $f\left(  x,y\right)  $ measurable with respect to $y$ for
every $x$, and continuous with respect to $x$ at any $x\neq y,$ for a.e. $y$.
Then the function%
\[
m\left(  x\right)  =\int_{T}f\left(  x,y\right)  dy
\]
is continuous in $T$.
\end{proposition}

\begin{proof}
(i) Let $\varphi_{\varepsilon}:[0,\infty)\rightarrow\left[  0,1\right]  $ be a
continuous function such that $\varphi_{\varepsilon}\left(  t\right)  =0$ for
$t\leqslant\varepsilon/2,$ $\varphi_{\varepsilon}\left(  t\right)  =1$ for
$t\geqslant\varepsilon,$ and define%
\begin{align*}
f_{\varepsilon}\left(  x,y\right)   &  =f\left(  x,y\right)  \varphi
_{\varepsilon}\left(  d\left(  x,y\right)  \right)  ;\\
g_{\varepsilon}\left(  x,y\right)   &  =g\left(  x,y\right)  \varphi
_{\varepsilon}\left(  d\left(  x,y\right)  \right)  ;\\
h_{\varepsilon}\left(  x,y\right)   &  =\int_{T}f_{\varepsilon}\left(
x,z\right)  g_{\varepsilon}\left(  z,y\right)  dz.
\end{align*}
For any fixed $\varepsilon>0$ the function%
\[
f_{\varepsilon}\left(  x,z\right)  g_{\varepsilon}\left(  z,y\right)
\]
is measurable with respect to $z$ for every $\left(  x,y\right)  $ and, for
any $z\in T,$ continuous in the joint variables $\left(  x,y\right)  .$
Moreover by our assumption on $f,g$ and Lemma \ref{Lemma nonintegrale},%
\[
\left\vert f_{\varepsilon}\left(  x,z\right)  g_{\varepsilon}\left(
z,y\right)  \right\vert \leqslant c\frac{1}{\left\vert B\left(  x,\varepsilon
\right)  \right\vert }\frac{1}{\left\vert B\left(  y,\varepsilon\right)
\right\vert }\leqslant c\left(  \varepsilon\right)  .
\]
Then, by Lebesgue theorem, $h_{\varepsilon}$ is continuous in $T\times T$,
since $T$ has finite measure. Let us show that $h_{\varepsilon}\left(
x,y\right)  \rightarrow h\left(  x,y\right)  $ locally uniformly for $x\neq
y,$ which will imply the continuity of $h$. To see this, let us write%
\begin{align*}
h_{\varepsilon}\left(  x,y\right)  -h\left(  x,y\right)   &  =\int_{T}\left[
f_{\varepsilon}\left(  x,z\right)  -f\left(  x,z\right)  \right]
g_{\varepsilon}\left(  z,y\right)  dz\\
&  +\int_{T}f\left(  x,z\right)  \left[  g_{\varepsilon}\left(  z,y\right)
-g\left(  z,y\right)  \right]  dz
\end{align*}
and%
\begin{align*}
\left\vert h_{\varepsilon}\left(  x,y\right)  -h\left(  x,y\right)
\right\vert  &  \leqslant c\int_{d\left(  x,z\right)  <\varepsilon}\phi
_{\beta}\left(  x,z\right)  \phi_{\gamma}\left(  z,y\right)  dz\\
&  +c\int_{d\left(  z,y\right)  <\varepsilon}\phi_{\beta}\left(  x,z\right)
\phi_{\gamma}\left(  z,y\right)  dz=I+II.
\end{align*}
Now, for $d\left(  x,y\right)  \geqslant\delta>0$ and $\varepsilon<\delta/2,$
$d\left(  x,z\right)  <\varepsilon$ implies $d\left(  z,y\right)  >\delta/2,$
hence by Lemma \ref{Lemma nonintegrale} $\phi_{\gamma}\left(  z,y\right)
\leqslant c\left(  \delta\right)  $ and%
\[
I\leqslant c\left(  \delta\right)  \int_{d\left(  x,z\right)  <\varepsilon
}\phi_{\beta}\left(  x,z\right)  dz=c\left(  \delta\right)  \Psi_{\beta
}\left(  x,\varepsilon\right)  \rightarrow0
\]
as $\varepsilon\rightarrow0,$ uniformly for $d\left(  x,y\right)
\geqslant\delta>0$ (see Corollary \ref{coroll Psi beta}). Analogously%
\[
II\leqslant c\left(  \delta\right)  \int_{d\left(  z,y\right)  <\varepsilon
}\phi_{\gamma}\left(  z,y\right)  dz=c\left(  \delta\right)  \Psi_{\gamma
}\left(  y,\varepsilon\right)  \rightarrow0
\]
as $\varepsilon\rightarrow0,$ uniformly for $d\left(  x,y\right)
\geqslant\delta>0.$ Hence (i) is proved. The proof of (ii) is similar but easier.
\end{proof}

\section{The parametrix method\label{sec parametrix}}

Let $x_{0}$, $U\left(  x_{0}\right)  $, and $I$ be as in the previous
sections. To shorten notation, in the following we will write $U$ instead of
$U\left(  x_{0}\right)  $. We will denote by $\xi,\eta$ lifted variables
ranging in the small domain%
\[
V\subset U\times I\subset\mathbb{R}^{p+m},
\]
as in the approximation theorem. By known results of Folland \cite{Fo}, the
operator
\[
\mathcal{L}=\sum_{i=1}^{p}Y_{i}^{2}+Y_{0}%
\]
possesses a fundamental solution $\Gamma$ on $\mathbb{G},$ left invariant and
homogeneous of degree $2-Q.$ (Recall that, in order for Folland's theory to be
applicable, the homogeneous dimension $Q$ of $\mathbb{G}$ must be $\geqslant
3$. However, this restriction only rules out uniformly elliptic operators in
two variables).

In particular, this means that for some positive constant $c$ we have%

\begin{align}
\left\vert \Gamma\left(  \Theta_{\eta}\left(  \xi\right)  \right)
\right\vert  &  \leqslant\frac{c}{\left\Vert \Theta_{\eta}\left(  \xi\right)
\right\Vert ^{Q-2}};\nonumber\\
\left\vert \left(  Y_{i}\Gamma\right)  \left(  \Theta_{\eta}\left(
\xi\right)  \right)  \right\vert ,\left\vert \widetilde{X}_{i}\left[
\Gamma\left(  \Theta_{\eta}\left(  \xi\right)  \right)  \right]  \right\vert
&  \leqslant\frac{c}{\left\Vert \Theta_{\eta}\left(  \xi\right)  \right\Vert
^{Q-1}};\nonumber\\
& \label{bound Gamma}\\
\left\vert \left(  Y_{i}Y_{j}\Gamma\right)  \left(  \Theta_{\eta}\left(
\xi\right)  \right)  \right\vert ,\left\vert \widetilde{X}_{i}\widetilde
{X}_{j}\left[  \Gamma\left(  \Theta_{\eta}\left(  \xi\right)  \right)
\right]  \right\vert  &  \leqslant\frac{c}{\left\Vert \Theta_{\eta}\left(
\xi\right)  \right\Vert ^{Q}};\nonumber\\
\left\vert \left(  Y_{0}\Gamma\right)  \left(  \Theta_{\eta}\left(
\xi\right)  \right)  \right\vert ,\left\vert \widetilde{X}_{0}\left[
\Gamma\left(  \Theta_{\eta}\left(  \xi\right)  \right)  \right]  \right\vert
&  \leqslant\frac{c}{\left\Vert \Theta_{\eta}\left(  \xi\right)  \right\Vert
^{Q}},\nonumber
\end{align}
for every $\eta,\xi\in V,$ $\eta\neq\xi$, where the $\widetilde{X}%
$-derivatives act on the $\xi$ variable. Recall that, according to the
Notation stated at the end of \S \ \ref{sec known results}, we will always
assume that differential operators act on the $\xi$ variable of $\Gamma\left(
\Theta_{\eta}\left(  \xi\right)  \right)  $. Also, recall that by
(\ref{equiv Theta d}) $\left\Vert \Theta_{\eta}\left(  \xi\right)  \right\Vert
$ is equivalent to $\widetilde{d}\left(  \eta,\xi\right)  $.

Let us define the following (local) \emph{parametrix for the operator} $L$.
For $x,y\in U$, we set
\begin{equation}
P\left(  x,y\right)  =\int_{\mathbb{R}^{m}}\left(  \int_{\mathbb{R}^{m}}%
\Gamma\left(  \Theta_{\left(  y,k\right)  }\left(  x,h\right)  \right)
\varphi\left(  h\right)  dh\right)  \varphi\left(  k\right)  dk, \label{P}%
\end{equation}
where $\varphi\in C_{0}^{\infty}\left(  \mathbb{R}^{m}\right)  $ is a cutoff
function fixed once and for all, equal to one in a neighborhood of the origin
and supported in $I$. It is worth telling that the alternative definition
\[
\int_{\mathbb{R}^{m}}\Gamma\left(  \Theta_{\left(  y,0\right)  }\left(
x,h\right)  \right)  \varphi\left(  h\right)  dh
\]
of the parametrix (as in \cite[eq.(20)]{NSW}) would be fit for the purposes of
this section, but not for those of section \ref{sec:further regularity}. Let
us also note that, in case our vector fields $X_{i}$ were free up to step $s$,
the lifting procedure would be unnecessary, we would simply have
$\widetilde{X}_{i}=X_{i}$ and:
\[
P\left(  x,y\right)  =\Gamma\left(  \Theta_{y}\left(  x\right)  \right)  .
\]

As already sketched in the introduction, the strategy is then the following.
We look for a fundamental solution for $L$ of the form%
\[
\gamma\left(  x,y\right)  =P\left(  x,y\right)  +J\left(  x,y\right)
\]
where%
\[
J\left(  x,y\right)  =\int_{U}P\left(  x,z\right)  \Phi\left(  z,y\right)
dz.
\]
In turn, we will find $\Phi$ as the series
\begin{equation}
\Phi\left(  z,y\right)  =\sum_{j=1}^{\infty}Z_{j}\left(  z,y\right)
\text{\ for }z\neq y \label{phi}%
\end{equation}
where the $Z_{j}$'s are defined inductively by%
\begin{align}
Z_{1}\left(  x,y\right)   &  =LP\left(  x,y\right) \label{Z_1}\\
Z_{j+1}\left(  x,y\right)   &  =\int_{U}Z_{1}\left(  x,z\right)  Z_{j}\left(
z,y\right)  dz\text{ }\ \ \ \text{for }x\neq y.\nonumber
\end{align}
More precisely, we will eventually find that the above identities need to be
slightly modified multiplying some of the involved functions by a suitable
coefficient $c_{0}\left(  x\right)  $; the necessity of this will be clear in
the following.

Before carrying out this plan step by step, let us clarify the way how our
constants will depend on the vector fields:

\bigskip\noindent\textbf{Dependence of the constants.} All the constants in
the upper bounds proved in this section will depend on the vector fields
$X_{i}$'s only through the following quantities:

(i) the norms $C^{r-1,\alpha}\left(  \Omega\right)  $ of the coefficients of
$X_{i}$ $\left(  i=1,2,...,n\right)  $ and the norms $C^{r-2,\alpha}\left(
\Omega\right)  $ of the coefficients of $X_{0}$;

(ii) a positive constant $c_{0}$ such that the following bound holds:%
\[
\inf_{x\in\Omega}\max_{\left\vert I_{1}\right\vert ,\left\vert I_{2}%
\right\vert ,...,\left\vert I_{p}\right\vert \leqslant r}\left\vert
\det\left(  \left(  X_{\left[  I_{1}\right]  }\right)  _{x},\left(  X_{\left[
I_{2}\right]  }\right)  _{x},...,\left(  X_{\left[  I_{p}\right]  }\right)
_{x}\right)  \right\vert \geqslant c_{0},
\]
where \textquotedblleft$\det$\textquotedblright\ denotes the determinant of
the $p\times p$ matrix having the vectors $\left(  X_{\left[  I_{i}\right]
}\right)  _{x}$ as rows.

\begin{proposition}
[Properties of $P$]\label{Prop P}Under the above assumptions and with the
above notation, we have, for any $x,y\in U$:%
\begin{align}
P\left(  \cdot,y\right)   &  \in C^{\infty}\left(  U\setminus\left\{
y\right\}  \right)  ;\label{P reg 1}\\
P\left(  x,\cdot\right)   &  \in C_{loc}^{\alpha}\left(  U\setminus\left\{
x\right\}  \right)  ;\label{P reg 2}\\
P  &  \in C\left(  U\times U\setminus\Delta\right)  ;\label{P reg 3}\\
X_{i}P,~X_{j}X_{i}P,~X_{0}P  &  \in C\left(  U\times U\setminus\Delta\right)
\label{P reg 4}%
\end{align}
for $i,j=1,2,...,n,$ where $\Delta=\left\{  \left(  x,x\right)  :x\in
U\right\}  $ and the exponent $\alpha\in(0,1]$ is the one appearing in the
assumptions on the coefficients of the vector fields $X_{i}$'s. Moreover:%
\begin{align}
\left\vert P\left(  x,y\right)  \right\vert  &  \leqslant c\phi_{2}\left(
x,y\right)  ;\label{Stima P}\\
\left\vert X_{i}P\left(  x,y\right)  \right\vert  &  \leqslant c\phi
_{1}\left(  x,y\right)  \text{ for }i=1,2,...,n;\label{Stima XP}\\
\left\vert X_{j}X_{i}P\left(  x,y\right)  \right\vert ,\left\vert
X_{0}P\left(  x,y\right)  \right\vert  &  \leqslant c\phi_{0}\left(
x,y\right)  \text{ for }i,j=1,2,...,n. \label{Stima XXP}%
\end{align}

\end{proposition}

(For the meaning of the symbol $X_{i}P\left(  x,y\right)  $, recall the
Notation fixed at the end of \S \ \ref{sec known results}). Note that,
regardless the infinite differentiability of $P\left(  \cdot,y\right)  $, only
$r$ derivatives of $P\left(  \cdot,y\right)  $ \emph{with respect to the
vector fields }$X_{i}$ exist (since the vector fields themselves are
nonsmooth). In particular, recalling that $r\geqslant2$, we have that
$X_{i}X_{j}P\left(  x,y\right)  $ is well defined for any $x\neq y$.

\bigskip

\begin{proof}
From (\ref{P}) we read that for any $x\neq y$ the integral defining $P$ is
absolutely convergent, and $P$ can be differentiated under the integral sign.
Since $\Gamma$ is smooth outside the origin, by the properties of the map
$\Theta$ stated in Theorem \ref{Thm liftapprox}, condition (\ref{P reg 1})
immediately follows. To prove (\ref{P reg 2}) and\ (\ref{P reg 3}) we will
show that for $x\neq y$ we have a locally uniform (in $x$) control on the
$C_{loc}^{\alpha}$ modulus of continuity in $y$ for $P\left(  x,\cdot\right)
$. Namely, since $\Gamma$ is smooth outside the origin, we can write%
\[
\left\vert \Gamma\left(  u_{1}\right)  -\Gamma\left(  u_{2}\right)
\right\vert \leqslant c\left(  \delta\right)  \left\vert u_{1}-u_{2}%
\right\vert
\]
if $\left\vert u_{1}\right\vert \geqslant\delta$ and $\left\vert u_{1}%
-u_{2}\right\vert \leqslant\delta/2$. Also, we know that, by Theorem
\ref{Thm liftapprox}
\[
d\left(  x,y\right)  \leqslant\widetilde{d}\left(  \left(  x,h\right)
,\left(  y,k\right)  \right)  \leqslant c\left\Vert \Theta_{\left(
y,k\right)  }\left(  x,h\right)  \right\Vert \leqslant c\left\vert
\Theta_{\left(  y,k\right)  }\left(  x,h\right)  \right\vert ^{1/r}%
\]
and%
\[
\left\vert \Theta_{\left(  y_{1},k\right)  }\left(  x,h\right)  -\Theta
_{\left(  y_{2},k\right)  }\left(  x,h\right)  \right\vert \leqslant
c\left\vert y_{1}-y_{2}\right\vert ^{\alpha},
\]
hence there exist constants $c_{1},c_{2}$ such that for any fixed $\delta>0$,
if $d\left(  x,y_{1}\right)  \geqslant c_{1}\delta^{1/r}$ and $\left\vert
y_{1}-y_{2}\right\vert \leqslant c_{2}\delta^{1/\alpha}$ then%
\begin{align*}
\left\vert P\left(  x,y_{1}\right)  -P\left(  x,y_{2}\right)  \right\vert  &
\leqslant\int_{\mathbb{R}^{m}}\left(  \int_{\mathbb{R}^{m}}\left\vert
\Gamma\left(  \Theta_{\left(  y_{1},k\right)  }\left(  x,h\right)  \right)
-\Gamma\left(  \Theta_{\left(  y_{2},k\right)  }\left(  x,h\right)  \right)
\right\vert \varphi\left(  h\right)  dh\right)  \varphi\left(  k\right)  dk\\
&  \leqslant c\left(  \delta\right)  \left\vert y_{1}-y_{2}\right\vert
^{\alpha}%
\end{align*}
which means that $P\left(  x,\cdot\right)  $ is $C^{\alpha}$ locally uniformly
for $x\neq y$.

Lemma \ref{Lemma NSW} with $\beta=2$ together with (\ref{equiv Theta d}),
(\ref{bound Gamma}) and (\ref{P}) implies (\ref{Stima P}).

Moreover, by (\ref{X Xtilda}) and (\ref{approx theta}),%
\begin{align*}
&  X_{i}P\left(  x,y\right)  =\int_{\mathbb{R}^{m}}\int_{\mathbb{R}^{m}%
}\widetilde{X}_{i}\left[  \Gamma\left(  \Theta_{\left(  y,k\right)  }\left(
x,h\right)  \right)  \varphi\left(  h\right)  \right]  dh\,\varphi\left(
k\right)  dk\\
=  &  \int_{\mathbb{R}^{m}}\int_{\mathbb{R}^{m}}\left\{  \left[  \left(
Y_{i}\Gamma\right)  \left(  \Theta_{\left(  y,k\right)  }\left(  x,h\right)
\right)  +\left(  R_{i}^{\left(  y,k\right)  }\Gamma\right)  \left(
\Theta_{\left(  y,k\right)  }\left(  x,h\right)  \right)  \right]
\varphi\left(  h\right)  \right. \\
&  +\left.  \Gamma\left(  \Theta_{\left(  y,k\right)  }\left(  x,h\right)
\right)  \widetilde{X}_{i}\varphi\left(  h\right)  \right\}  dh\,\varphi
\left(  k\right)  dk\\
&  \equiv\int_{\mathbb{R}^{m}}\int_{\mathbb{R}^{m}}\left(  Y_{i}\Gamma\right)
\left(  \Theta_{\left(  y,k\right)  }\left(  x,h\right)  \right)
\varphi\left(  h\right)  dh\,\varphi\left(  k\right)  dk\\
&  +\int_{\mathbb{R}^{m}}\int_{\mathbb{R}^{m}}\sum_{l=1}^{2}Q_{l}\left(
y,k;~x,h\right)  \varphi_{l}\left(  h\right)  dh\,\varphi\left(  k\right)  dk
\end{align*}
where $\varphi_{l}\in C_{0}^{\infty}\left(  \mathbb{R}^{m}\right)  $ and, by
(\ref{remainder}) and (\ref{bound Gamma}),
\begin{align*}
\left\vert Q_{l}\left(  y,k;~x,h\right)  \right\vert  &  \leqslant\frac
{c}{\left\Vert \Theta_{\left(  y,k\right)  }\left(  x,h\right)  \right\Vert
^{Q-1+\alpha}},\\
\left\vert \left(  Y_{i}\Gamma\right)  \left(  \Theta_{\left(  y,k\right)
}\left(  x,h\right)  \right)  \right\vert  &  \leqslant\frac{c}{\left\Vert
\Theta_{\left(  y,k\right)  }\left(  x,h\right)  \right\Vert ^{Q-1}},
\end{align*}
so that Lemma \ref{Lemma NSW} implies (\ref{Stima XP}).

The proof of (\ref{Stima XXP}) is an iteration of the previous argument:%
\begin{align*}
X_{j}X_{i}P\left(  x,y\right)  =  &  \int_{\mathbb{R}^{m}}\int_{\mathbb{R}%
^{m}}\widetilde{X}_{j}\widetilde{X}_{i}\left[  \Gamma\left(  \Theta_{\left(
y,k\right)  }\left(  x,h\right)  \right)  \varphi\left(  h\right)  \right]
dh\,\varphi(k)dk\\
=  &  \int_{\mathbb{R}^{m}}\int_{\mathbb{R}^{m}}\left\{  \left(  Y_{j}%
Y_{i}\Gamma\right)  \left(  \Theta_{\left(  y,k\right)  }\left(  x,h\right)
\right)  \varphi\left(  h\right)  \right. \\
&  +\left(  Y_{j}R_{i}^{\left(  y,k\right)  }\Gamma+R_{j}^{\left(  y,k\right)
}Y_{i}\Gamma+R_{j}^{\left(  y,k\right)  }R_{i}^{\left(  y,k\right)  }%
\Gamma\right)  \left(  \Theta_{\left(  y,k\right)  }\left(  x,h\right)
\right)  \varphi\left(  h\right) \\
&  +\left[  \left(  Y_{j}\Gamma\right)  \left(  \Theta_{\left(  y,k\right)
}\left(  x,h\right)  \right)  +\left(  R_{j}^{\left(  y,k\right)  }%
\Gamma\right)  \left(  \Theta_{\left(  y,k\right)  }\left(  x,h\right)
\right)  \right]  \widetilde{X}_{i}\varphi\left(  h\right) \\
&  +\left[  \left(  Y_{i}\Gamma\right)  \left(  \Theta_{\left(  y,k\right)
}\left(  x,h\right)  \right)  +\left(  R_{i}^{\left(  y,k\right)  }%
\Gamma\right)  \left(  \Theta_{\left(  y,k\right)  }\left(  x,h\right)
\right)  \right]  \widetilde{X}_{j}\varphi\left(  h\right) \\
&  \left.  +\Gamma\left(  \Theta_{\left(  y,k\right)  }\left(  x,h\right)
\right)  \widetilde{X}_{j}\widetilde{X}_{i}\varphi\left(  h\right)  \right\}
dh\,\varphi(k)dk\\
=  &  \int_{\mathbb{R}^{m}}\int_{\mathbb{R}^{m}}\left(  Y_{j}Y_{i}%
\Gamma\right)  \left(  \Theta_{\left(  y,k\right)  }\left(  x,h\right)
\right)  \varphi\left(  h\right)  dh\varphi(k)dk\\
&  +\int_{\mathbb{R}^{m}}\int_{\mathbb{R}^{m}}\sum_{s}Q_{s}\left(
y,k;~x,h\right)  \varphi_{s}\left(  h\right)  dh\,\varphi\left(  k\right)  dk
\end{align*}
with $\varphi_{s}\in C_{0}^{\infty}\left(  \mathbb{R}^{m}\right)  $ and, as
above, exploiting now also Corollary \ref{Coroll second order remainder},%
\begin{align*}
\left\vert Q_{s}\left(  y,k;~x,h\right)  \right\vert  &  \leqslant\frac
{c}{\left\Vert \Theta_{\left(  y,k\right)  }\left(  x,h\right)  \right\Vert
^{Q-\alpha}},\\
\left\vert \left(  Y_{j}Y_{i}\Gamma\right)  \left(  \Theta_{\left(
y,k\right)  }\left(  x,h\right)  \right)  \right\vert  &  \leqslant\frac
{c}{\left\Vert \Theta_{\left(  y,k\right)  }\left(  x,h\right)  \right\Vert
^{Q}}%
\end{align*}
which by Lemma \ref{Lemma NSW} implies (\ref{Stima XXP}) for $X_{j}%
X_{i}P\left(  x,y\right)  $. The proof of the bound on $X_{0}P\left(
x,y\right)  $ is similar.

Finally, the explicit expression of the derivatives $X_{i}P,X_{j}X_{i}%
P,X_{0}P$ allows us to repeat the argument used to prove (\ref{P reg 3}),
showing that also (\ref{P reg 4}) holds.
\end{proof}

\begin{proposition}
[Properties of $Z_{1}$]\label{Prop bound Z_1}Let $L$ be as in (\ref{H}) and,
for $x,y\in U$, $x\neq y$, let%
\begin{equation}
Z_{1}\left(  x,y\right)  =LP\left(  x,y\right)  . \label{defZ_1}%
\end{equation}
Under the above assumptions and with the above notation, we have:
\begin{align}
Z_{1}\left(  \cdot,y\right)   &  \in C_{loc}^{r-2,\alpha}\left(
U\setminus\left\{  y\right\}  \right)  ;\label{Z1 cont 1}\\
Z_{1}\left(  x,\cdot\right)   &  \in C_{loc}^{\alpha}\left(  U\setminus
\left\{  x\right\}  \right)  ;\label{Z1 cont 2}\\
Z_{1}  &  \in C\left(  U\times U\setminus\Delta\right)  . \label{Z1 cont 3}%
\end{align}
Moreover:%
\begin{equation}
\left\vert Z_{1}\left(  x,y\right)  \right\vert \leqslant c_{1}\phi_{\alpha
}\left(  x,y\right)  . \label{Z1 bound}%
\end{equation}

\end{proposition}

\begin{proof}
Let us first prove (\ref{Z1 bound}). The computation is similar to that of the
previous proof. However we have to write it explicitly because we will need it
in the following. By (\ref{X Xtilda}) and Theorem \ref{Thm liftapprox} we have%
\begin{align*}
&  Z_{1}\left(  x,y\right)  =\int_{\mathbb{R}^{m}}\int_{\mathbb{R}^{m}%
}\widetilde{L}\left[  \Gamma\left(  \Theta_{\left(  y,k\right)  }\left(
x,h\right)  \right)  \varphi\left(  h\right)  \right]  dh\,\varphi(k)dk\\
&  =\int_{\mathbb{R}^{m}}\int_{\mathbb{R}^{m}}\left\{  \left(  \mathcal{L}%
\Gamma\right)  \left(  \Theta_{\left(  y,k\right)  }\left(  x,h\right)
\right)  \varphi\left(  h\right)  +\right. \\
&  +\left(  \sum_{j}\left(  Y_{j}R_{j}^{\left(  y,k\right)  }\Gamma
+R_{j}^{\left(  y,k\right)  }Y_{j}\Gamma+R_{j}^{\left(  y,k\right)  }%
R_{j}^{\left(  y,k\right)  }\Gamma\right)  +R_{0}^{\left(  y,k\right)  }%
\Gamma\right)  \left(  \Theta_{\left(  y,k\right)  }\left(  x,h\right)
\right)  \varphi\left(  h\right) \\
&  +2\sum_{j}\left[  \left(  Y_{j}\Gamma\right)  \left(  \Theta_{\left(
y,k\right)  }\left(  x,h\right)  \right)  +R_{j}^{\left(  y,k\right)  }%
\Gamma\left(  \Theta_{\left(  y,k\right)  }\left(  x,h\right)  \right)
\right]  \widetilde{X}_{j}\varphi\left(  h\right) \\
&  +\left.  \Gamma\left(  \Theta_{\left(  y,k\right)  }\left(  x,h\right)
\right)  \widetilde{L}\varphi\left(  h\right)  \right\}  dh\,\varphi(k)dk.
\end{align*}
Since $\left(  \mathcal{L}\Gamma\right)  \left(  u\right)  =0$ for $u\neq0,$
then $\left(  \mathcal{L}\Gamma\right)  \left(  \Theta_{\left(  y,k\right)
}\left(  x,h\right)  \right)  =0$ for $\left(  x,h\right)  \neq\left(
y,k\right)  $, so that, for $x\neq y$,
\begin{equation}
Z_{1}\left(  x,y\right)  =\int_{\mathbb{R}^{m}}\int_{\mathbb{R}^{m}}\sum
_{i=1}^{3}Q_{i}\left(  y,k,~x,h\right)  \varphi_{i}\left(  h\right)
dh\,\varphi(k)dk \label{Z1}%
\end{equation}
where $\varphi_{i}\in C_{0}^{\infty}\left(  \mathbb{R}^{m}\right)  $ and, by
Corollary \ref{Coroll second order remainder},
\[
\left\vert Q_{i}\left(  y,k,~x,h\right)  \right\vert \leqslant\frac
{c}{\left\Vert \Theta_{\left(  y,k\right)  }\left(  x,h\right)  \right\Vert
^{Q-\alpha}}.
\]
It follows that
\[
\left\vert Z_{1}\left(  x,y\right)  \right\vert \leqslant c\int_{\mathbb{R}%
^{m}}\left(  \int_{\mathbb{R}^{m}}\frac{\psi\left(  h\right)  }{\left\Vert
\Theta_{\left(  y,k\right)  }\left(  x,h\right)  \right\Vert ^{Q-\alpha}%
}dh\right)  \varphi(k)dk
\]
for some $\psi\in C_{0}^{\infty}\left(  \mathbb{R}^{m}\right)  $. By Lemma
\ref{Lemma NSW}, (\ref{Z1 bound}) follows.

As to the regularity of $Z_{1}$, let us inspect for instance the term%
\[
\sum_{j}\int_{\mathbb{R}^{m}}\int_{\mathbb{R}^{m}}\left(  R_{j}^{\left(
y,k\right)  }R_{j}^{\left(  y,k\right)  }\Gamma\right)  \left(  \Theta
_{\left(  y,k\right)  }\left(  x,h\right)  \right)  \varphi\left(  h\right)
dh\,\varphi\left(  k\right)  dk
\]
(all the others being more regular). By Corollary
\ref{Coroll second order remainder},%
\[
u\mapsto R_{j}^{\left(  y,k\right)  }R_{j}^{\left(  y,k\right)  }\Gamma\left(
u\right)
\]
is a $C_{loc}^{r-2,\alpha}$ function outside the origin. Since $\xi
\mapsto\Theta_{\left(  y,k\right)  }\left(  \xi\right)  $ is smooth,
\[
\left(  x,h\right)  \mapsto R_{j}^{\left(  y,k\right)  }R_{j}^{\left(
y,k\right)  }\Gamma\left(  \Theta_{\left(  y,k\right)  }\left(  x,h\right)
\right)
\]
is at least $C_{loc}^{r-2,\alpha}$ for $\left(  x,h\right)  \neq\left(
y,k\right)  ,$ and $Z_{1}\left(  \cdot,y\right)  \in C_{loc}^{r-2,\alpha
}\left(  U\setminus\left\{  y\right\}  \right)  .$

To deal with the regularity of $Z_{1}\left(  x,\cdot\right)  $ note that by
Corollary \ref{Coroll second order remainder}, $R_{j}^{\eta}R_{j}^{\eta}%
\Gamma\left(  u\right)  $ depends on $\eta$ in a $C^{\alpha}$ continuous way
locally uniformly in $u\neq0$. It follows that%
\[
y\mapsto\left(  R_{j}^{\left(  y,k\right)  }R_{j}^{\left(  y,k\right)  }%
\Gamma\right)  \left(  \Theta_{\left(  y,k\right)  }\left(  x,h\right)
\right)
\]
is $C_{loc}^{\alpha}\left(  U\setminus\left\{  x\right\}  \right)  $ and the
same is true for $Z_{1}\left(  x,\cdot\right)  $, by an argument similar to
that used in the proof of Proposition \ref{Prop P} to deal with $P\left(
x,\cdot\right)  $. Joint continuity of $Z_{1}$ in $\left(  x,y\right)  ,$
outside the diagonal, also follows from these facts by a local uniformity argument.
\end{proof}

Next, we can prove:

\begin{proposition}
[Properties of $\Phi$]\label{Prop Phi}Let, for $j=1,2,3,...,$%
\begin{equation}
Z_{j+1}\left(  x,y\right)  =\int_{U}Z_{1}\left(  x,z\right)  Z_{j}\left(
z,y\right)  dz\text{ }\ \ \ \text{for }x,y\in U,x\neq y. \label{defZj}%
\end{equation}
Then the functions $Z_{j}\left(  x,y\right)  $ are well defined for $x,y\in
U,x\neq y.$ Moreover, shrinking $U$ if necessary, the series%
\begin{equation}
\Phi\left(  x,y\right)  =\sum_{j=1}^{\infty}Z_{j}\left(  x,y\right)
\label{Phi}%
\end{equation}
converges for $x,y\in U,x\neq y$ and the function $\Phi$ satisfies the bound%
\begin{equation}
\left\vert \Phi\left(  x,y\right)  \right\vert \leqslant c\phi_{\alpha}\left(
x,y\right)  \label{phi bound}%
\end{equation}
and the integral equation%
\begin{equation}
\Phi\left(  x,y\right)  =Z_{1}\left(  x,y\right)  +\int_{U}Z_{1}\left(
x,z\right)  \Phi\left(  z,y\right)  dz\text{ \ \ for }x,y\in U,x\neq y.
\label{Int Eq}%
\end{equation}
Finally,
\[
Z_{j},\Phi\in C\left(  U\times U\setminus\Delta\right)  .
\]

\end{proposition}

\begin{proof}
By definition of $Z_{j},$ the bound (\ref{Z1 bound}) and Theorem
\ref{Thm phi alfa} we have, recursively:%
\begin{align*}
\left\vert Z_{2}\left(  x,y\right)  \right\vert  &  \leqslant c_{1}^{2}%
c\frac{2}{\alpha}\phi_{2\alpha}\left(  x,y\right)  ;\\
\left\vert Z_{3}\left(  x,y\right)  \right\vert  &  \leqslant c_{1}^{3}\left(
c\frac{2}{\alpha}\right)  c\left(  \frac{1}{\alpha}+\frac{1}{2\alpha}\right)
\phi_{3\alpha}\left(  x,y\right)  \leqslant c_{1}^{3}\left(  c\frac{2}{\alpha
}\right)  ^{2}\phi_{3\alpha}\left(  x,y\right)  ;\\
&  ...\\
\left\vert Z_{j_{0}}\left(  x,y\right)  \right\vert  &  \leqslant c_{1}%
^{j_{0}}\left(  c\frac{2}{\alpha}\right)  ^{j_{0}-1}\phi_{j_{0}\alpha}\left(
x,y\right)  \leqslant CR^{j_{0}\alpha-Q}\text{ }\leqslant CR^{\alpha},
\end{align*}
where $j_{0}$ is the least integer such that $j_{0}>Q/\alpha$. Then:%
\[
\left\vert Z_{j_{0}+k}\left(  x,y\right)  \right\vert \leqslant CR^{\alpha
}\left(  cc_{1}R^{\alpha}\right)  ^{k}\text{ for any }k\geqslant0.
\]
We now choose $U$ small enough in order to get
\[
\delta\equiv cc_{1}R^{\alpha}<1.
\]
Then%
\[
\left\vert Z_{j_{0}+k}\left(  x,y\right)  \right\vert \leqslant C\delta^{k}%
\]
so that the series%
\[
\sum_{j=j_{0}}^{\infty}Z_{j}\left(  x,y\right)
\]
totally converges and the upper bound (\ref{phi bound}) holds. Moreover, we
can write, for any $x,y\in U,x\neq y$:%
\begin{align*}
&  Z_{1}\left(  x,y\right)  +\int_{U}Z_{1}\left(  x,z\right)  \Phi\left(
z,y\right)  dz\\
&  =Z_{1}\left(  x,y\right)  +\int_{U}Z_{1}\left(  x,z\right)  \sum
_{j=1}^{\infty}Z_{j}\left(  z,y\right)  dz\\
&  =Z_{1}\left(  x,y\right)  +\sum_{j=1}^{\infty}\int_{U}Z_{1}\left(
x,z\right)  Z_{j}\left(  z,y\right)  dz\\
&  =Z_{1}\left(  x,y\right)  +\sum_{j=1}^{\infty}Z_{j+1}\left(  x,y\right) \\
&  =\sum_{j=1}^{\infty}Z_{j}\left(  x,y\right)  =\Phi\left(  x,y\right)
\end{align*}
so that (\ref{Int Eq}) holds. Let us come to the continuity properties of
$Z_{j},\Phi.$ By (\ref{Z1 cont 3}) and Lemma \ref{Lemma joint continuity}, the
definition (\ref{defZj}) recursively implies that
\[
Z_{j}\in C\left(  U\times U\setminus\Delta\right)  \text{ for }j=2,3,...
\]
Since, by the above proof, the series in (\ref{phi}) totally converges, this
also implies that%
\[
\Phi\in C\left(  U\times U\setminus\Delta\right)  .
\]

\end{proof}

\begin{proposition}
[Properties of $J$]\label{Prop stime P J}Let $U$ be as in the previous
proposition. For $x,y\in U,x\neq y$, let%
\begin{equation}
J\left(  x,y\right)  =\int_{U}P\left(  x,z\right)  \Phi\left(  z,y\right)  dz.
\label{J}%
\end{equation}
Then: $J$ and $X_{i}J$ ($i=1,2,...,n$) are well defined for any $x,y\in
U,x\neq y;$%
\begin{equation}
J,\,X_{i}J\in C\left(  U\times U\setminus\Delta\right)  ; \label{J cont}%
\end{equation}
moreover, the following estimates hold ($i=1,2,...,n$):%
\begin{align}
\left\vert J\left(  x,y\right)  \right\vert  &  \leqslant c\phi_{2+\alpha
}\left(  x,y\right)  ;\label{stima J}\\
\left\vert X_{i}J\left(  x,y\right)  \right\vert  &  \leqslant c\phi
_{1+\alpha}\left(  x,y\right)  . \label{stima XJ}%
\end{align}

\end{proposition}

\begin{proof}
By (\ref{Stima P}), (\ref{phi bound}) and Theorem \ref{Thm phi alfa}, we have%
\[
\left\vert J\left(  x,y\right)  \right\vert \leqslant c\int_{U}\phi_{2}\left(
x,z\right)  \phi_{\alpha}\left(  z,y\right)  dz\leqslant c\phi_{2+\alpha
}\left(  x,y\right)  .
\]
Also, $X_{i}J$ is well defined, indeed%
\begin{align*}
X_{i}J\left(  x,y\right)   &  =X_{i}\int_{U}P\left(  x,z\right)  \Phi\left(
z,y\right)  dz\\
&  =\int_{U}X_{i}P\left(  x,z\right)  \Phi\left(  z,y\right)  dz
\end{align*}
and, by Propositions \ref{Prop P}, \ref{Prop Phi} and Theorem
\ref{Thm phi alfa},%
\[
\left\vert X_{i}J\left(  x,z\right)  \right\vert \leqslant c\int_{U}\phi
_{1}\left(  x,z\right)  \phi_{\alpha}\left(  z,y\right)  dz\leqslant
c\phi_{1+\alpha}\left(  x,y\right)  .
\]
As to the continuity properties: by Proposition \ref{Prop P} and Proposition
\ref{Prop Phi} we know that $P\left(  x,y\right)  ,X_{i}P\left(  x,y\right)
,\Phi\left(  x,y\right)  $ are continuous in the joint variables for $x\neq y$
and satisfy the bounds%
\begin{align*}
\left\vert P\left(  x,y\right)  \right\vert  &  \leqslant c\phi_{2}\left(
x,y\right)  ;\\
\left\vert X_{i}P\left(  x,y\right)  \right\vert  &  \leqslant c\phi
_{1}\left(  x,y\right)  \text{ }\left(  i=1,2,...,n\right)  ;\\
\left\vert \Phi\left(  x,y\right)  \right\vert  &  \leqslant c\phi_{\alpha
}\left(  x,y\right)  .
\end{align*}
Hence Proposition \ref{Lemma joint continuity} implies (\ref{J cont}).
\end{proof}

\begin{proposition}
\label{Prop J}The following identity and upper bound hold in weak sense:%
\begin{align}
LJ\left(  x,y\right)   &  =\int_{U}Z_{1}\left(  x,z\right)  \Phi\left(
z,y\right)  dz-c_{0}\left(  x\right)  \Phi\left(  x,y\right)
,\label{LJ identity}\\
\left\vert LJ\left(  x,y\right)  \right\vert  &  \leqslant c\phi_{\alpha
}\left(  x,y\right)  \label{stima LJ}%
\end{align}
where%
\[
c_{0}\left(  x\right)  =\int_{\mathbb{R}^{m}}c\left(  x,k\right)  \varphi
^{2}(k)dk
\]
and $c\left(  x,k\right)  $ is defined in (\ref{d xi}).

Explicitly, denoting by $G\left(  x,y\right)  $ the right hand side of
(\ref{LJ identity}), we have%
\begin{equation}
\int_{U}J\left(  x,y\right)  L^{\ast}\psi\left(  x\right)  dx=\int_{U}G\left(
x,y\right)  \psi\left(  x\right)  dx \label{LJ identity weak}%
\end{equation}
for any $\psi\in C_{0}^{\infty}\left(  U\right)  $ and $y\in U$, where
$L^{\ast}$ is the transposed operator of $L$ (see (\ref{transposed})), and%
\[
\left\vert G\left(  x,y\right)  \right\vert \leqslant c\phi_{\alpha}\left(
x,y\right)  .
\]

\end{proposition}

For the proof of the above proposition we need the following lemma.

\begin{lemma}
\label{Cutoff}Let $\omega$ be a smooth function on $\mathbb{G}$ such that
$\omega\left(  u\right)  =0$ for $\left\Vert u\right\Vert <\frac{1}{2}$ and
$\omega\left(  u\right)  =1$ for $\left\Vert u\right\Vert >1$ and let
$\omega_{\varepsilon}\left(  u\right)  =\omega\left(  D\left(  \varepsilon
^{-1}\right)  u\right)  $. Let $R_{1}$ and $R_{2}$ be vector fields on
$\mathbb{G}$ given by
\[
R_{1}=\sum_{j=1}^{N}a_{j}\left(  u\right)  \partial_{u_{j}},R_{2}=\sum
_{j=1}^{N}b_{j}\left(  u\right)  \partial_{u_{j}}%
\]
and assume that, for a couple of $s_{1},s_{2}\in\mathbb{R}$ and some constant
$c>0,$ every $j,k=1,2,...,N,$%
\begin{align*}
\left\vert a_{j}\left(  u\right)  \right\vert  &  \leqslant c\left\Vert
u\right\Vert ^{s_{1}+\alpha_{j}};\\
\left\vert b_{j}\left(  u\right)  \right\vert  &  \leqslant c\left\Vert
u\right\Vert ^{s_{2}+\alpha_{j}};\\
\left\vert \partial_{u_{k}}b_{j}\left(  u\right)  \right\vert  &  \leqslant
c\left\Vert u\right\Vert ^{s_{2}+\alpha_{j}-\alpha_{k}}%
\end{align*}
where the $\alpha_{j}$'s are as in (\ref{dilation}). Then there exists
$c^{\prime}>0$ such that for every $\varepsilon>0$ and $u\in\mathbb{G}$%
\begin{align*}
\left\vert R_{1}\omega_{\varepsilon}\left(  u\right)  \right\vert  &
\leqslant c^{\prime}\left\Vert u\right\Vert ^{s_{1}}\\
\left\vert R_{1}R_{2}\omega_{\varepsilon}\left(  u\right)  \right\vert  &
\leqslant c^{\prime}\left\Vert u\right\Vert ^{s_{1}+s_{2}}.
\end{align*}

\end{lemma}

\begin{proof}
We have%
\begin{align*}
\left\vert R_{1}\omega_{\varepsilon}\left(  u\right)  \right\vert  &
\leqslant\sum\left\vert a_{j}\left(  u\right)  \right\vert \left\vert
\frac{\partial\omega_{\varepsilon}}{\partial u_{j}}\left(  u\right)
\right\vert \\
&  \leqslant\sum\left\Vert u\right\Vert ^{s_{1}+\alpha_{j}}\frac
{1}{\varepsilon^{\alpha_{j}}}\left\vert \frac{\partial\omega}{\partial u_{j}%
}\left(  D\left(  \varepsilon^{-1}\right)  u\right)  \right\vert \\
&  \leqslant c\left\Vert u\right\Vert ^{s_{1}}%
\end{align*}
since on the support of $\frac{\partial\omega}{\partial u_{j}}\left(  D\left(
\varepsilon^{-1}\right)  u\right)  $ we have $\left\Vert u\right\Vert
\leqslant\varepsilon$. Similarly,%
\begin{align*}
&  \left\vert R_{1}R_{2}\omega_{\varepsilon}\left(  u\right)  \right\vert =\\
&  =\left\vert \left(  \sum_{k=1}^{N}a_{k}\left(  u\right)  \partial_{u_{k}%
}\sum_{j=1}^{N}b_{j}\left(  u\right)  \partial_{u_{j}}\right)  \omega
_{\varepsilon}\left(  u\right)  \right\vert \\
&  \leqslant c\sum_{k=1}^{N}\left\Vert u\right\Vert ^{s_{1}+\alpha_{k}}%
\sum_{j=1}^{N}\left\vert \partial_{u_{k}}b_{j}\left(  u\right)  \partial
_{u_{j}}\omega_{\varepsilon}\left(  u\right)  +b_{j}\left(  u\right)
\partial_{u_{k}u_{j}}^{2}\omega_{\varepsilon}\right\vert \\
&  \leqslant c\sum_{k=1}^{N}\left\Vert u\right\Vert ^{s_{1}+\alpha_{k}}%
\sum_{j=1}^{N}\left(  \frac{\left\Vert u\right\Vert ^{s_{2}+\alpha_{j}%
-\alpha_{k}}}{\varepsilon^{\alpha_{j}}}\left\vert \partial_{u_{j}}%
\omega\left(  D\left(  \frac{1}{\varepsilon}\right)  u\right)  \right\vert
+\frac{\left\Vert u\right\Vert ^{s_{2}+\alpha_{j}}}{\varepsilon^{\alpha
_{k}+\alpha_{j}}}\left\vert \partial_{u_{k}u_{j}}^{2}\omega\left(  D\left(
\frac{1}{\varepsilon}\right)  u\right)  \right\vert \right) \\
&  \leqslant c\sum_{k=1}^{N}\left\Vert u\right\Vert ^{s_{1}+\alpha_{k}%
}\left\Vert u\right\Vert ^{s_{2}-\alpha_{k}}=c\left\Vert u\right\Vert
^{s_{1}+s_{2}}.
\end{align*}

\end{proof}

\begin{proof}
[Proof of Proposition \ref{Prop J}]To prove (\ref{LJ identity}) we use a
distributional argument. Let $\omega_{\varepsilon}$ be as in the previous
Lemma, let $\Gamma_{\varepsilon}=\omega_{\varepsilon}\Gamma$ and define%
\[
P_{\varepsilon}\left(  x,y\right)  =\int_{\mathbb{R}^{m}}\left(
\int_{\mathbb{R}^{m}}\Gamma_{\varepsilon}\left(  \Theta_{\left(  y,k\right)
}\left(  x,h\right)  \right)  \varphi\left(  h\right)  dh\right)
\,\varphi\left(  k\right)  dk
\]
and%
\[
J_{\varepsilon}\left(  x,y\right)  =\int_{U}P_{\varepsilon}\left(  x,z\right)
\Phi\left(  z,y\right)  dz.
\]
We have%
\[
J_{\varepsilon}\left(  x,y\right)  =\int_{\mathbb{R}^{m}}\left(
\int_{\mathbb{R}^{m}}\int_{U}\Gamma_{\varepsilon}\left(  \Theta_{\left(
z,k\right)  }\left(  x,h\right)  \right)  \Phi\left(  z,y\right)
dz\ \varphi\left(  h\right)  dh\right)  \,\varphi\left(  k\right)  dk
\]
and%
\begin{align*}
&  LJ_{\varepsilon}\left(  x,y\right) \\
&  =\int_{\mathbb{R}^{m}}\int_{\mathbb{R}^{m}}\int_{U}\widetilde{L}\left[
\Gamma_{\varepsilon}\left(  \Theta_{\left(  z,k\right)  }\left(  x,h\right)
\right)  \varphi\left(  h\right)  \right]  \Phi\left(  z,y\right)
\varphi\left(  k\right)  dzdhdk
\end{align*}%
\begin{align*}
&  =\int_{\mathbb{R}^{m}}\int_{\mathbb{R}^{m}}\int_{U}\mathcal{L}%
\Gamma_{\varepsilon}\left(  \Theta_{\left(  z,k\right)  }\left(  x,h\right)
\right)  \varphi\left(  h\right)  \Phi\left(  z,y\right)  \varphi\left(
k\right)  dzdhdk\\
&  +\int_{\mathbb{R}^{m}}\int_{\mathbb{R}^{m}}\int_{U}\left(  \sum_{i}\left(
Y_{i}R_{i}^{\left(  z,k\right)  }+R_{i}^{\left(  z,k\right)  }Y_{i}%
+R_{i}^{\left(  z,k\right)  }R_{i}^{\left(  z,k\right)  }+R_{0}^{\left(
z,k\right)  }\right)  \Gamma_{\varepsilon}\left(  \Theta_{\left(  z,k\right)
}\left(  x,h\right)  \right)  \right)  \times\\
&  ~~~~~~~~~~\times\varphi\left(  h\right)  \Phi\left(  z,y\right)
\varphi\left(  k\right)  dzdhdk\\
&  +\int_{\mathbb{R}^{m}}\int_{\mathbb{R}^{m}}\int_{U}2\sum_{i}Y_{i}%
\Gamma_{\varepsilon}\left(  \Theta_{\left(  z,k\right)  }\left(  x,h\right)
\right)  \widetilde{X}_{i}\varphi\left(  h\right)  \Phi\left(  z,y\right)
\varphi\left(  k\right)  dzdhdk\\
&  +\int_{\mathbb{R}^{m}}\int_{\mathbb{R}^{m}}\int_{U}\Gamma_{\varepsilon
}\left(  \Theta_{\left(  z,k\right)  }\left(  x,h\right)  \right)
\widetilde{L}^{\left(  x,h\right)  }\varphi\left(  h\right)  \Phi\left(
z,y\right)  \varphi\left(  k\right)  dzdhdk.
\end{align*}
To bound
\[
\sum_{i}\left(  Y_{i}R_{i}^{\left(  z,k\right)  }+R_{i}^{\left(  z,k\right)
}Y_{i}+R_{i}^{\left(  z,k\right)  }R_{i}^{\left(  z,k\right)  }+R_{0}^{\left(
z,k\right)  }\right)  \Gamma_{\varepsilon}\left(  u\right)
\]
we now recall that, by Theorem \ref{Thm liftapprox}, the vector fields
$R_{i}^{\left(  z,k\right)  },Y_{i},R_{0}^{\left(  z,k\right)  }$ satisfy the
assumptions of Lemma \ref{Cutoff} \ with $s_{1}$ or $s_{2}$ equal to
$\alpha-1,-1,\alpha-2,$ respectively. A simple computation shows that%
\[
\left\vert \sum_{i}\left(  Y_{i}R_{i}^{\left(  z,k\right)  }+R_{i}^{\left(
z,k\right)  }Y_{i}+R_{i}^{\left(  z,k\right)  }R_{i}^{\left(  z,k\right)
}+R_{0}^{\left(  z,k\right)  }\right)  \Gamma_{\varepsilon}\left(  u\right)
\right\vert \leqslant\frac{c}{\left\Vert u\right\Vert ^{Q-\alpha}}.
\]
Hence for suitable $\varphi_{j}\in C_{0}^{\infty}\left(  \mathbb{R}%
^{m}\right)  $ and $Q_{\varepsilon,j}$ satisfying%
\[
\left\vert Q_{\varepsilon,j}\left(  z,k;~x,h\right)  \right\vert
\leqslant\frac{c}{\left\Vert \Theta_{\left(  z,k\right)  }\left(  x,h\right)
\right\Vert ^{Q-\alpha}}%
\]
we have
\begin{align*}
LJ_{\varepsilon}\left(  x,y\right)   &  =\int_{\mathbb{R}^{m}}\int
_{\mathbb{R}^{m}}\int_{U}\mathcal{L}\Gamma_{\varepsilon}\left(  \Theta
_{\left(  z,k\right)  }\left(  x,h\right)  \right)  \varphi\left(  h\right)
\Phi\left(  z,y\right)  \varphi\left(  k\right)  dzdhdk\\
&  +\int_{\mathbb{R}^{m}}\int_{\mathbb{R}^{m}}\int_{U}\sum_{j=1}%
^{3}Q_{\varepsilon,j}\left(  z,k;~x,h\right)  \varphi_{j}\left(  h\right)
\Phi\left(  z,y\right)  \varphi\left(  k\right)  dzdhdk.
\end{align*}
Let now $\psi\in C_{0}^{\infty}\left(  U\right)  $ be any test function. Then%
\begin{align*}
&  \int_{\mathbb{R}^{p}}LJ_{\varepsilon}\left(  x,y\right)  \psi\left(
x\right)  dx=\\
&  =\int_{U}\int_{\mathbb{R}^{m}}\int_{\mathbb{R}^{m}}\int_{\mathbb{R}^{p}%
}\mathcal{L}\Gamma_{\varepsilon}\left(  \Theta_{\left(  z,k\right)  }\left(
x,h\right)  \right)  \psi\left(  x\right)  \varphi\left(  h\right)
\varphi\left(  k\right)  dxdhdk\Phi\left(  z,y\right)  dz\\
&  +\int_{U}\int_{\mathbb{R}^{m}}\int_{\mathbb{R}^{m}}\int_{\mathbb{R}^{p}%
}\sum_{j=1}^{3}Q_{\varepsilon,j}\left(  z,k;~x,h\right)  \psi\left(  x\right)
\varphi_{j}\left(  h\right)  \varphi\left(  k\right)  dxdhdk\Phi\left(
z,y\right)  dz.
\end{align*}
Let now change variable in the first integral setting $u=\Theta_{\left(
z,k\right)  }\left(  x,h\right)  .$ Then, by (\ref{d xi}),%
\[
dxdh=c\left(  z,k\right)  \left(  1+O\left(  \left\Vert u\right\Vert \right)
\right)  du,
\]
and setting
\[
\widehat{\varphi}_{(z,k)}\left(  u\right)  =\left.  \varphi\left(  h\right)
\psi\left(  x\right)  \vphantom{\int}\right\vert _{u=\Theta_{\left(
z,k\right)  }\left(  x,h\right)  }%
\]
we have
\begin{align*}
&  \int_{\mathbb{R}^{p}}LJ_{\varepsilon}\left(  x,y\right)  \psi\left(
x\right)  dx=\int_{U}\int_{\mathbb{R}^{m}}\int_{\mathbb{G}}\mathcal{L}%
\Gamma_{\varepsilon}\left(  u\right)  \widehat{\varphi}_{(z,k)}\left(
u\right)  du\ c(z,k)\varphi(k)dk\Phi\left(  z,y\right)  dz\\
&  +\int_{U}\int_{\mathbb{R}^{m}}\int_{\mathbb{G}}\mathcal{L}\Gamma
_{\varepsilon}\left(  u\right)  O\left(  \left\Vert u\right\Vert \right)
\widehat{\varphi}_{(z,k)}\left(  u\right)  du\ c(z,k)\varphi(k)dk\Phi\left(
z,y\right)  dz\\
&  +\int_{U}\int_{\mathbb{R}^{m}}\int_{\mathbb{R}^{m}}\int_{\mathbb{R}^{p}%
}\sum_{j=1}^{3}Q_{\varepsilon,j}\left(  z,k;~x,h\right)  \varphi_{j}\left(
h\right)  \psi\left(  x\right)  dx\varphi(k)dhdk\Phi\left(  z,y\right)  dz.
\end{align*}
Since $\mathcal{L}\Gamma_{\varepsilon}\left(  u\right)  =0$ for $\left\Vert
u\right\Vert >\varepsilon$, letting $\varepsilon\rightarrow0$ the second
integral in the right hand side vanishes, by Lebesgue's theorem, and
integrating by part in the first integral we get:%
\begin{align*}
\lim_{\varepsilon\rightarrow0}  &  \int_{\mathbb{R}^{p}}LJ_{\varepsilon
}\left(  x,y\right)  \psi\left(  x\right)  dx=\\
&  =\int_{U}\int_{\mathbb{R}^{m}}\int_{\mathbb{G}}\Gamma\left(  u\right)
\mathcal{L}^{\ast}\widehat{\varphi}_{(z,k)}\left(  u\right)  du\ c(z,k)\varphi
(k)dk\Phi\left(  z,y\right)  dz\\
&  +\int_{U}\int_{\mathbb{R}^{m}}\int_{\mathbb{R}^{m}}\int_{\mathbb{R}^{p}%
}\sum_{j=1}^{3}Q_{j}\left(  z,k;~x,h\right)  \varphi_{j}\left(  h\right)
\psi\left(  x\right)  \varphi(k)dxdhdk\Phi\left(  z,y\right)  dz
\end{align*}
where $Q_{j}$ are as in (\ref{Z1}) and
\[
\mathcal{L}^{\ast}=\sum_{i=1}^{n}Y_{i}^{2}-Y_{0}%
\]
is the adjoint operator of $\mathcal{L}$, so that%
\[
\int_{\mathbb{G}}\Gamma\left(  u\right)  \mathcal{L}^{\ast}\widehat{\varphi
}_{(z,k)}\left(  u\right)  du=-\widehat{\varphi}_{(z,k)}\left(  0\right)
=-\varphi\left(  k\right)  \psi\left(  z\right)
\]
and%
\begin{align*}
&  \lim_{\varepsilon\rightarrow0}\int_{\mathbb{R}^{p}}LJ_{\varepsilon}\left(
x,y\right)  \psi\left(  x\right)  dx\\
&  =-\int_{U}\psi\left(  z\right)  c_{0}\left(  z\right)  \Phi\left(
z,y\right)  dz+\int_{U}\int_{\mathbb{R}^{p}}Z_{1}\left(  x,z\right)
\Phi\left(  z,y\right)  \psi\left(  x\right)  dxdz,
\end{align*}
having set%
\begin{equation}
c_{0}\left(  z\right)  =\int_{\mathbb{R}^{m}}c\left(  z,k\right)  \varphi
^{2}(k)dk. \label{c_0}%
\end{equation}
On the other hand,
\[
\lim_{\varepsilon\rightarrow0}\int_{\mathbb{R}^{p}}LJ_{\varepsilon}\left(
x,y\right)  \psi\left(  x\right)  dx=\lim_{\varepsilon\rightarrow0}%
\int_{\mathbb{R}^{p}}J_{\varepsilon}\left(  x,y\right)  L^{\ast}\psi\left(
x\right)  dx=\int_{\mathbb{R}^{p}}J\left(  x,y\right)  L^{\ast}\psi\left(
x\right)  dx,
\]
which easily follows by Lebesgue's dominated convergence theorem and the bound
(\ref{stima J})\ on $J,J_{\varepsilon}$. Therefore%
\[
LJ\left(  x,y\right)  =-c_{0}\left(  x\right)  \Phi\left(  x,y\right)
+\int_{U}Z_{1}\left(  x,y\right)  \Phi\left(  z,y\right)  dz,
\]
which is (\ref{LJ identity}). This also implies, by (\ref{phi bound}),
(\ref{Z1 bound}) and Theorem \ref{Thm phi alfa}:%
\begin{align*}
\left\vert LJ\left(  x,y\right)  \right\vert  &  \leqslant c\phi_{\alpha
}\left(  x,y\right)  +c\int_{U}\phi_{\alpha}\left(  x,z\right)  \phi_{\alpha
}\left(  z,y\right)  dz\\
&  \leqslant c\phi_{\alpha}\left(  x,y\right)  +c\phi_{2\alpha}\left(
x,y\right)  \leqslant c\phi_{\alpha}\left(  x,y\right)  ,
\end{align*}
which is (\ref{stima LJ}).
\end{proof}

In view of the presence of the term $c_{0}\left(  x\right)  $ in the identity
(\ref{LJ identity}) we now modify our previous construction as follows:%
\begin{align*}
Z_{1}^{\prime}\left(  x,y\right)   &  =\frac{1}{c_{0}\left(  x\right)  }%
Z_{1}\left(  x,y\right)  ;\\
Z_{k+1}^{\prime}\left(  x,y\right)   &  =\int_{U}Z_{1}^{\prime}\left(
x,z\right)  Z_{k}^{\prime}\left(  z,y\right)  dz;\\
\Phi^{\prime}\left(  x,y\right)   &  =\sum_{k=1}^{\infty}Z_{k}^{\prime}\left(
x,y\right)  ;\\
J^{\prime}\left(  x,y\right)   &  =\int_{U}P\left(  x,z\right)  \Phi^{\prime
}\left(  z,y\right)  dz.
\end{align*}
With these definitions, the following hold:%
\begin{align}
\Phi^{\prime}\left(  x,y\right)   &  =Z_{1}^{\prime}\left(  x,y\right)
+\int_{U}Z_{1}^{\prime}\left(  x,z\right)  \Phi^{\prime}\left(  z,y\right)
dz;\label{int Eq prime}\\
c_{0}\left(  x\right)  \Phi^{\prime}\left(  x,y\right)   &  =Z_{1}\left(
x,y\right)  +\int_{U}Z_{1}\left(  x,z\right)  \Phi^{\prime}\left(  z,y\right)
dz;\label{c Phi'}\\
LJ^{\prime}\left(  x,y\right)   &  =\int_{U}Z_{1}\left(  x,z\right)
\Phi^{\prime}\left(  z,y\right)  dz-c_{0}\left(  x\right)  \Phi^{\prime
}\left(  x,y\right)  . \label{LJ'}%
\end{align}

\begin{remark}
\label{Remark c_0}Recalling that%
\[
0<c_{1}\leqslant c_{0}\left(  x\right)  \leqslant c_{2}%
\]
for any $x\in U$, and that $c_{0}\in C^{\alpha}\left(  U\right)  $ (since by
Theorem \ref{Thm liftapprox} the function $c$ is H\"{o}lder continuous), it is
immediate to check that the functions $Z_{1}^{\prime},Z_{k}^{\prime}%
,\Phi^{\prime},J^{\prime}$ satisfy the same upper bounds (with different
constants) and continuity properties proved in Propositions
\ref{Prop bound Z_1}, \ref{Prop Phi}, \ref{Prop J} for $Z_{1},Z_{k},\Phi,J$, respectively.
\end{remark}

We have, at last:

\begin{theorem}
[Existence of fundamental solution]\label{Thm gamma}Let%
\[
\gamma\left(  x,y\right)  =\frac{1}{c_{0}\left(  y\right)  }\left[  P\left(
x,y\right)  +J^{\prime}\left(  x,y\right)  \right]  .
\]
Then $\gamma\left(  x,y\right)  $ and $X_{i}\gamma\left(  x,y\right)  $
($i=1,2,...,n$) are well defined and continuous in the joint variables $x,y\in
U,x\neq y$, and satisfy the following bounds:%
\begin{align}
\left\vert \gamma\left(  x,y\right)  \right\vert  &  \leqslant c\phi
_{2}\left(  x,y\right)  ;\label{gam 1}\\
\left\vert X_{i}\gamma\left(  x,y\right)  \right\vert  &  \leqslant c\phi
_{1}\left(  x,y\right)  . \label{gam 2}%
\end{align}
Moreover, $\gamma\left(  \cdot,y\right)  $ is a weak solution to
$L\gamma\left(  \cdot,y\right)  =-\delta_{y}$, that is:%
\begin{equation}
\int_{U}\gamma\left(  x,y\right)  L^{\ast}\psi\left(  x\right)  dx=-\psi
\left(  y\right)  \label{gam 4}%
\end{equation}
for any $\psi\in C_{0}^{\infty}\left(  U\right)  ,y\in U$. Finally, if
$X_{0}\equiv0$, then there exists $\varepsilon>0$ such that%
\begin{equation}
\gamma\left(  x,y\right)  >0\text{ for }d\left(  x,y\right)  <\varepsilon.
\label{gam 3}%
\end{equation}

\end{theorem}

\begin{remark}
When $X_{0}$ does not vanish the fundamental solution $\Gamma$ of the
homogeneous operator can be proved to be only non-negative, as the example of
the heat operator suggests. As a consequence nothing can be said in this case
about the sign of $\gamma$ near the pole.
\end{remark}

\begin{proof}
By (\ref{P reg 3}), (\ref{P reg 4}), (\ref{J cont}) and Remark
\ref{Remark c_0} the functions $\gamma\left(  x,y\right)  \ $and $X_{i}%
\gamma\left(  x,y\right)  $ are continuous in the joint variables $x,y\in
U,x\neq y$.

The bounds (\ref{gam 1}), (\ref{gam 2}) follow from Proposition
\ref{Prop stime P J} and Proposition \ref{Prop P}. As to (\ref{gam 3}),%
\[
\left\vert c_{0}\left(  y\right)  \gamma\left(  x,y\right)  -P\left(
x,y\right)  \right\vert =\left\vert J^{\prime}\left(  x,y\right)  \right\vert
\leqslant c\phi_{2+\alpha}\left(  x,y\right)  \leqslant c\frac{d\left(
x,y\right)  ^{2+\alpha}}{\left\vert B\left(  x,d\left(  x,y\right)  \right)
\right\vert }.
\]
If $X_{0}\equiv0$, then also $Y_{0}\equiv0$ and by \cite[Prop.\ 5.3.13,
p.243]{BLU} the function $\Gamma$ is strictly positive, hence%
\[
\Gamma\left(  u\right)  \geqslant\frac{c}{\left\Vert u\right\Vert ^{Q-2}}%
\]
and, reasoning like in Lemma \ref{Lemma NSW} one can check that%
\begin{align*}
P\left(  x,y\right)   &  \geqslant c\int_{\left\vert k\right\vert
\leqslant\varepsilon}\int_{\left\vert h\right\vert \leqslant\varepsilon}%
\frac{dhdk}{\left\Vert \Theta_{\left(  y,k\right)  }\left(  x,h\right)
\right\Vert ^{Q-2}}\\
&  \geqslant c\int_{d\left(  x,y\right)  }^{R}\frac{r}{\left\vert B\left(
x,r\right)  \right\vert }dr\geqslant c\frac{d\left(  x,y\right)  ^{2}%
}{\left\vert B\left(  x,d\left(  x,y\right)  \right)  \right\vert }%
\end{align*}
and (\ref{gam 3}) follows.

To prove (\ref{gam 4}), we have to show that for any test function $\psi$
\[
-\psi\left(  y\right)  c_{0}\left(  y\right)  =\int P\left(  x,y\right)
L^{\ast}\psi\left(  x\right)  dx+\int J^{\prime}\left(  x,y\right)  L^{\ast
}\psi\left(  x\right)  dx\equiv A+B.
\]
As to $A,$ exploiting the same computation performed in the proof of
Proposition \ref{Prop J},
\begin{align*}
A  &  =\lim_{\varepsilon\rightarrow0}\int P_{\varepsilon}\left(  x,y\right)
L^{\ast}\psi\left(  x\right)  dx=\lim_{\varepsilon\rightarrow0}\int
LP_{\varepsilon}\left(  x,y\right)  \psi\left(  x\right)  dx\\
&  =-\psi\left(  y\right)  c_{0}\left(  y\right)  +\int_{\mathbb{R}^{p}}%
Z_{1}\left(  x,y\right)  \psi\left(  x\right)  dx.
\end{align*}
On the other hand, by (\ref{LJ'}),
\[
B=\int_{\mathbb{R}^{p}}\psi\left(  x\right)  \left\{  \int Z_{1}\left(
x,z\right)  \Phi^{\prime}\left(  z,y\right)  dz-c_{0}\left(  x\right)
\Phi^{\prime}\left(  x,y\right)  \right\}  dx.
\]
By (\ref{c Phi'}),%
\begin{align*}
A+B  &  =-\psi\left(  y\right)  c_{0}\left(  y\right)  +\\
&  +\int_{\mathbb{R}^{p}}\psi\left(  x\right)  \left\{  Z_{1}\left(
x,y\right)  +\int Z_{1}\left(  x,z\right)  \Phi^{\prime}\left(  z,y\right)
dz-c_{0}\left(  x\right)  \Phi^{\prime}\left(  x,y\right)  \right\}  dx\\
&  =-\psi\left(  y\right)  c_{0}\left(  y\right)
\end{align*}
and we are done.
\end{proof}

\section{Further regularity of the fundamental solution \newline and local
solvability of $L$}

\label{sec:further regularity}In this section, under a stronger regularity
assumption on the coefficients of the vector fields, we will show that the
fundamental solution $\gamma\left(  \cdot,y\right)  $ actually possesses
second order derivatives with respect to the vector fields, satisfying natural
growth bounds, and $\gamma\left(  \cdot,y\right)  $ satisfies the equation
$Lu=0$ (outside the pole) in classical sense. As a consequence, we can
establish a local solvability result for the operator $L$.

\bigskip\noindent\textbf{Assumptions B. }In this section we assume that for
some integer $r\geqslant2\ $and some $\alpha\in(0,1],$ the coefficients of the
vector fields $X_{1},X_{2},...,X_{n}$ belong to $C^{r,\alpha}\left(
\Omega\right)  ,$ while the coefficients of $X_{0}$ belong to $C^{r-1,\alpha
\,}\left(  \Omega\right)  $. If $r=2$ we assume $\alpha=1$. Moreover, we still
assume that $X_{0},X_{1},...,X_{n}$ satisfy H\"{o}rmander's condition of step
$r$ in $\Omega$: the vectors
\[
\left\{  \left(  X_{\left[  I\right]  }\right)  _{x}\right\}  _{\left\vert
I\right\vert \leqslant r}%
\]
span $\mathbb{R}^{p}$ for any $x\in\Omega$. (For examples of systems of vector
fields satisfying the assumptions, see the Appendix).

\medskip Throughout this section we keep using the notation introduced in
\S \ref{sec parametrix}; in particular, $U$ stands for a fixed neighborhood of
a point $x_{0}\in\Omega$ where all the previous construction can be performed.
Accordingly to Assumptions B, from now on the constants appearing in our
estimates will have the following dependence on the vector fields:

\bigskip\noindent\textbf{Dependence of the constants.} All the constants
appearing in the upper bounds proved in this section will depend on the vector
fields only through the following quantities:

(i) the norms $C^{r,\alpha}\left(  \Omega\right)  $ of the coefficients of
$X_{i}$ $\left(  i=1,2,...,n\right)  $ and the norms $C^{r-1,\alpha}\left(
\Omega\right)  $ of the coefficients of $X_{0}$;

(ii) a positive constant $c_{0}$ such that the following bound holds:%
\[
\inf_{x\in\Omega}\max_{\left\vert I_{1}\right\vert ,\left\vert I_{2}%
\right\vert ,...,\left\vert I_{p}\right\vert \leqslant r}\left\vert
\det\left(  \left(  X_{\left[  I_{1}\right]  }\right)  _{x},\left(  X_{\left[
I_{2}\right]  }\right)  _{x},...,\left(  X_{\left[  I_{p}\right]  }\right)
_{x}\right)  \right\vert \geqslant c_{0}.
\]

\bigskip

Before proceeding we need to define precisely our functional framework and the
notion of solution.

\begin{definition}
\label{def solution}If $u$ is a function, not necessarily smooth, defined in
an open set $D\subseteq\Omega$, then:

we say that $X_{i}u$ exists in $D$ if the classical $X_{i}$-directional
derivative of $u$ exists in $D$;

we say that $u\in C_{X}^{1}\left(  D\right)  $ if for $i=1,2,...,n,$ the
derivatives $X_{i}u$ exist and are continuous in $D$;

we say that $u\in C_{X}^{2}\left(  D\right)  $ if $u\in C_{X}^{1}\left(
D\right)  $ and for $i,j=1,2,...,n,$ the derivatives $X_{i}X_{j}u$ and
$X_{0}u$ exist and are continuous in $D$.

Let $f$ be a continuous function in $D$. We say that $u$ is a (classical)
solution to%
\[
Lu=f\text{ in }D
\]
if $u\in C_{X}^{2}\left(  D\right)  \ $and $Lu\left(  x\right)  =f\left(
x\right)  $ for every $x\in D$.

We say that the operator $L$ is \emph{locally solvable }in $\Omega$ if for
every $x_{0}\in\Omega$ there exists a neighborhood $U\left(  x_{0}\right)  $
such that for every $\beta>0$ and $f\in C^{\beta}\left(  U\left(
x_{0}\right)  \right)  $ the equation $Lw=f$ has at least a $C_{X}^{2}\left(
U\left(  x_{0}\right)  \right)  $ solution.
\end{definition}

Note that, by Proposition \ref{Prop Lagrange}, any $C_{X}^{2}\left(  D\right)
$ function is necessarily continuous; if $X_{0}\equiv0$ the same conclusion
holds for $C_{X}^{1}\left(  D\right)  $ functions.

\begin{remark}
We recall that, even for the classical Laplacian, under the mere assumption of
continuity of $f$ in $D$, a $C^{2}\left(  D\right)  $ solution to $\Delta w=f$
may not exist. A counterexample is given for instance in \cite[exercise 4.9,
p.71]{GT}. Therefore the condition $f\in C^{\beta}\left(  D\right)  $ in the
definition of solvability is a natural requirement.
\end{remark}

The existence of $X_{i}u$ will be sometimes established by the following:

\begin{lemma}
\label{Lemma derivata classica}Let $D\subset\mathbb{R}^{p}$ be an open set and
let $X$ be a $C^{1}\left(  D\right)  $ vector field. Let $w$ be a $C\left(
D\right)  $ function and let $w_{\varepsilon}\in C^{1}\left(  D\right)  $ be
such that for $x\in D$, $w_{\varepsilon}\left(  x\right)  \rightarrow w\left(
x\right)  $ as $\varepsilon\rightarrow0$ and $Xw_{\varepsilon}\rightarrow g$
uniformly on $D$. Then $w$ is differentiable along $X$ and $Xw=g$.
\end{lemma}

\begin{proof}
Let $x\in D$ and let $\upsilon\left(  t\right)  $ be an integral curve of $X$
such that $\upsilon\left(  0\right)  =x$ and let $h_{\varepsilon}\left(
t\right)  =w_{\varepsilon}\left(  \upsilon\left(  t\right)  \right)  $. Since
$h_{\varepsilon}^{\prime}\left(  t\right)  $ converges uniformly we have%
\begin{align*}
g\left(  \upsilon\left(  t\right)  \right)   &  =\lim_{\varepsilon
\rightarrow0}Xw_{\varepsilon}\left(  \upsilon\left(  t\right)  \right)
=\lim_{\varepsilon\rightarrow0}h_{\varepsilon}^{\prime}\left(  t\right) \\
&  =\frac{d}{dt}\left(  \lim_{\varepsilon\rightarrow0}h_{\varepsilon}\left(
t\right)  \right)  =\frac{d}{dt}\left(  \lim_{\varepsilon\rightarrow
0}w_{\varepsilon}\left(  \upsilon\left(  t\right)  \right)  \right)
=Xw\left(  \upsilon\left(  t\right)  \right)  ,
\end{align*}
so that%
\[
g\left(  x\right)  =Xw\left(  x\right)  .
\]

\end{proof}

\subsection{Preliminary results}

We now need to sharpen the analysis of the map $\Theta_{\eta}\left(
\xi\right)  $ performed in \cite{BBP2} showing that, under the above
(stronger) Assumptions B, this function possesses reasonable properties also
with respect to the \textquotedblleft bad\textquotedblright\ variable $\eta.$
Namely, the following holds:\pagebreak

\begin{proposition}
\label{Prop bad theta}Under Assumptions B:

\begin{enumerate}
\item[i)] the vector fields $R_{i}^{\eta}$ appearing in (\ref{approx theta})
are $C^{r+1-p_{i},\alpha}$ vector fields of weight $\geqslant\alpha-p_{i},$
depending on $\eta$ in a $C^{\alpha}$ way.

\item[ii)] the coefficients of the differential operators $D_{i}^{\eta}$
defined by the compositions%
\[
Y_{i}R_{j}^{\eta}R_{k}^{\eta},\text{ }R_{i}^{\eta}R_{j}^{\eta}R_{k}^{\eta
},\text{ }Y_{i}Y_{j}R_{k}^{\eta},\text{ }R_{i}^{\eta}Y_{j}^{\eta}R_{k}^{\eta
},\text{ }Y_{i}R_{j}^{\eta}Y_{k},\text{ }R_{i}^{\eta}R_{j}^{\eta}Y_{k},\text{
}Y_{i}R_{0}^{\eta},\text{ }R_{i}^{\eta}R_{0}^{\eta}\text{\ }%
\]
$\left(  i=0,1,2,...,n;j,k=1,2,...,n\right)  $ satisfy the bound%
\[
\left\vert D_{i}^{\eta}f\left(  u\right)  \right\vert \leqslant\frac
{c}{\left\Vert u\right\Vert ^{\mu+2+p_{i}-\alpha}}%
\]
for $u$ in a neighborhood of the origin, whenever $f:\mathbb{G}\rightarrow
\mathbb{R}$ is $D\left(  \lambda\right)  $-homogeneous of degree $-\mu$. Also,
the coefficients of $D_{i}^{\eta}$ depend on $\eta$ in a $C^{\alpha}$ way.

\item[iii)] the change of variables $\eta\mapsto u=\Theta_{\eta}\left(
\xi\right)  $ is a $C^{1,\alpha}$ diffeomorphism in a neighborhood of the
origin and its inverse $\eta=\Theta_{(\cdot)}(\xi)^{-1}(u)$ is $C^{1,\alpha}$
in the joint variables $(\xi,u)$. Moreover we have%
\[
d\eta=c\left(  \xi\right)  \left(  1+\chi\left(  \xi,u\right)  \right)  du,
\]
where, analogously to Theorem \ref{Thm liftapprox}, $c\left(  \cdot\right)  $
is a $C^{\alpha}$ function, bounded and bounded away from zero, $\chi\left(
\xi,u\right)  $ is $C^{\alpha}$ in the joint variables $\left(  \xi,u\right)
$ and for every $\gamma_{1},\gamma_{2}\geqslant0$ such that $\gamma_{1}%
+\gamma_{2}\leqslant\alpha$ there exists a constant $c$ such that%
\[
\left\vert \chi\left(  \xi_{1},u\right)  -\chi\left(  \xi_{2},u\right)
\right\vert \leqslant c\left\vert \xi_{1}-\xi_{2}\right\vert ^{\gamma_{1}%
}\left\Vert u\right\Vert ^{\gamma_{2}}.
\]
In particular%
\[
\left\vert \chi\left(  \xi,u\right)  \right\vert \leqslant c\left\Vert
u\right\Vert ^{\alpha}.
\]

\end{enumerate}
\end{proposition}

\begin{proof}
i) This follows with the same proof of \cite[Thm.\ 3.9]{BBP2}, under
assumption B.

ii) This follows as Corollary \ref{Coroll second order remainder}, by point
2.i of Theorem \ref{Thm Calphatheta}. Actually, the same proof of point 2.i of
Theorem \ref{Thm Calphatheta} implies this stronger conclusion, under the
stronger assumption B.

iii) With the notations of \cite[section 3.2]{BBP2} let%
\[
\xi=E\left(  u,\eta\right)  =\exp\left(  \sum_{I\in B}u_{I}S_{\left[
I\right]  ,\eta}\right)  \left(  \eta\right)
\]
and recall that $\Theta_{\eta}\left(  \xi\right)  $ is defined by $E\left(
\Theta_{\eta}\left(  \xi\right)  ,\eta\right)  =\xi$. Observe that, for every
fixed $\xi$, to express $\eta$ as a function of $u$ is equivalent to solve
with respect to $\eta$ the equation%
\begin{equation}
E\left(  u,\eta\right)  -\xi=0\text{.} \label{Dini}%
\end{equation}
Revising the proof of \cite[Thm.\ 3.9]{BBP2} under the assumption $b_{ij}\in
C^{r,\alpha}\left(  \Omega\right)  $, one can see that the smooth vector
fields $S_{\left[  I\right]  ,\eta}$ depend on $\eta$ in a $C^{1,\alpha}$ way.
This implies that $\xi=E\left(  u,\eta\right)  $ depends in a $C^{1,\alpha}$
way on the joint variables $\left(  u,\eta\right)  $ (see \cite[Prop.
30]{BBP2}). Since $E\left(  0,\eta\right)  =\eta$ we have $\frac{\partial
E}{\partial\eta}\left(  0,\eta\right)  =I$. The implicit function theorem
applied to equation (\ref{Dini}) shows that $\eta=\eta\left(  u,\xi\right)  $
is at least $C^{1}$ in the joint variables. The standard argument used to
prove the further regularity of the implicit function allows to prove that
this function is indeed $C^{1,\alpha}$ in the joint variables. Also, since%
\[
E\left(  u,\eta\left(  u,\xi\right)  \right)  -\xi=0
\]
differentiating with respect to $u$ yields
\[
\frac{\partial E}{\partial u}\left(  u,\eta\left(  u,\xi\right)  \right)
+\frac{\partial E}{\partial\eta}\left(  u,\eta\left(  u,\xi\right)  \right)
\frac{\partial\eta}{\partial u}\left(  u,\xi\right)  =0.
\]
Evaluating this identity for $u=0$ (that is $\eta=\xi$) gives%
\[
\frac{\partial E}{\partial u}\left(  0,\xi\right)  +I\frac{\partial\eta
}{\partial u}\left(  0,\xi\right)  =0
\]
so that%
\[
\frac{\partial\eta}{\partial u}\left(  0,\xi\right)  =-\frac{\partial
E}{\partial u}\left(  0,\xi\right)  =-\left(  \left(  S_{\left[  I\right]
,\xi}\right)  _{\xi}\right)  _{I\in B}=-\left(  \left(  \widetilde{X}_{\left[
I\right]  }\right)  _{\xi}\right)  _{I\in B}.
\]
Since%
\[
d\eta=J_{\xi}\left(  u\right)  du
\]
with $J_{\xi}\left(  u\right)  =\left\vert \det\frac{\partial\eta}{\partial
u}\left(  \xi,u\right)  \right\vert $, we have%
\begin{equation}
J_{\xi}\left(  u\right)  =\left\vert \det\left(  \left(  \widetilde
{X}_{\left[  I\right]  }\right)  _{\xi}\right)  _{I\in B}\right\vert +\chi
_{0}\left(  \xi,u\right)  . \label{chi zero}%
\end{equation}
Note that $\chi_{0}\left(  \xi,u\right)  $ is $C^{\alpha}$ in the joint
variables $\left(  u,\xi\right)  $ since $\eta\left(  u,\xi\right)  $ is
$C^{1,\alpha}$.

Assume now $\left\vert \xi_{1}-\xi_{2}\right\vert <\left\vert u\right\vert $,
then for any $\gamma_{1},\gamma_{2}\geqslant0$ with $\gamma_{1}+\gamma
_{2}\leqslant0$,
\[
\left\vert \chi_{0}\left(  u,\xi_{1}\right)  -\chi_{0}\left(  u,\xi
_{2}\right)  \right\vert \leqslant c\left\vert \xi_{1}-\xi_{2}\right\vert
^{\alpha}\leqslant c\left\vert \xi_{1}-\xi_{2}\right\vert ^{\gamma_{1}%
}\left\vert u\right\vert ^{\gamma_{2}}.
\]
If $\left\vert \xi_{1}-\xi_{2}\right\vert \geqslant\left\vert u\right\vert ,$
since $\chi_{0}\left(  0,\xi_{1}\right)  =\chi_{0}\left(  0,\xi_{2}\right)
=0$ we have%
\begin{align*}
\left\vert \chi_{0}\left(  u,\xi_{1}\right)  -\chi_{0}\left(  u,\xi
_{2}\right)  \right\vert  &  \leqslant\left\vert \chi_{0}\left(  u,\xi
_{1}\right)  -\chi_{0}\left(  0,\xi_{1}\right)  \right\vert +\left\vert
\chi_{0}\left(  u,\xi_{2}\right)  -\chi_{0}\left(  0,\xi_{2}\right)
\right\vert \\
&  \leqslant c\left\vert u\right\vert ^{\alpha}\leqslant c\left\vert \xi
_{1}-\xi_{2}\right\vert ^{\gamma_{1}}\left\vert u\right\vert ^{\gamma_{2}}.
\end{align*}
Hence in any case%
\begin{equation}
\left\vert \chi_{0}\left(  \xi_{1},u\right)  -\chi_{0}\left(  \xi
_{2},u\right)  \right\vert \leqslant c\left\vert \xi_{1}-\xi_{2}\right\vert
^{\gamma_{1}}\left\vert \left\vert u\right\vert \right\vert ^{\gamma_{2}}.
\label{biholder}%
\end{equation}
Then (\ref{chi zero}) can be rewritten as%
\[
d\eta=c\left(  \xi\right)  \left(  1+\chi\left(  \xi,u\right)  \right)  du
\]
where%
\[
c\left(  \xi\right)  =\left\vert \det\left(  \left(  \widetilde{X}_{\left[
I\right]  }\right)  _{\xi}\right)  _{I\in B}\right\vert
\]
is $C^{\alpha}$ and locally bounded away from zero, while%
\[
\chi\left(  \xi,u\right)  =\frac{\chi_{0}\left(  \xi,u\right)  }{c\left(
\xi\right)  }%
\]
still satisfies (\ref{biholder}). Hence point (iii) is proved.
\end{proof}

The following H\"{o}lder continuity estimate on the function $\Phi^{\prime}%
\ $will be crucial.

\begin{proposition}
\label{Phi holder}For any $\varepsilon\in\left(  0,\alpha\right)  $ there
exists $c>0$ such that%
\[
\left\vert \Phi^{\prime}\left(  x_{1},y\right)  -\Phi^{\prime}\left(
x_{2},y\right)  \right\vert \leqslant cd\left(  x_{1},x_{2}\right)
^{\alpha-\varepsilon}\phi_{\varepsilon}\left(  x_{1},y\right)
\]
for any $x_{1},x_{2},y\in U$ with $d\left(  x_{1},y\right)  \geqslant3d\left(
x_{1},x_{2}\right)  $.
\end{proposition}

Note that the same result holds if the number $3$ is replaced by another
constant $k>1,$ with $c$ depending on $k$.

The following easy variation of the previous result will be also useful:

\begin{corollary}
\label{coroll Phi holder}For any $\varepsilon\in\left(  0,\alpha\right)  $
there exists $c>0$ such that%
\[
\left\vert \Phi^{\prime}\left(  x_{1},y\right)  -\Phi^{\prime}\left(
x_{2},y\right)  \right\vert \leqslant cd\left(  x_{1},x_{2}\right)
^{\alpha-\varepsilon}\left[  \frac{d\left(  x_{1},y\right)  ^{\varepsilon}%
}{\left\vert B\left(  x_{1},d\left(  x_{1},y\right)  \right)  \right\vert
}+\frac{d\left(  x_{2},y\right)  ^{\varepsilon}}{\left\vert B\left(
x_{2},d\left(  x_{2},y\right)  \right)  \right\vert }\right]
\]
for any $x_{1},x_{2},y\in U$ with $y\neq x_{1},x_{2}.$
\end{corollary}

\begin{proof}
[Proof of Corollary \ref{coroll Phi holder}]If $d\left(  x_{1},y\right)
\geqslant3d\left(  x_{1},x_{2}\right)  $ by Proposition \ref{Phi holder} and
Lemma \ref{Lemma nonintegrale} we can bound%
\[
\left\vert \Phi^{\prime}\left(  x_{1},y\right)  -\Phi^{\prime}\left(
x_{2},y\right)  \right\vert \leqslant cd\left(  x_{1},x_{2}\right)
^{\alpha-\varepsilon}\phi_{\varepsilon}\left(  x_{1},y\right)  \leqslant
cd\left(  x_{1},x_{2}\right)  ^{\alpha-\varepsilon}\frac{d\left(
x_{1},y\right)  ^{\varepsilon}}{\left\vert B\left(  x_{1},d\left(
x_{1},y\right)  \right)  \right\vert }.
\]
Analogously if $d\left(  x_{2},y\right)  \geqslant3d\left(  x_{1}%
,x_{2}\right)  $ we can write%
\[
\left\vert \Phi^{\prime}\left(  x_{1},y\right)  -\Phi^{\prime}\left(
x_{2},y\right)  \right\vert \leqslant cd\left(  x_{1},x_{2}\right)
^{\alpha-\varepsilon}\phi_{\varepsilon}\left(  x_{2},y\right)  \leqslant
cd\left(  x_{1},x_{2}\right)  ^{\alpha-\varepsilon}\frac{d\left(
x_{2},y\right)  ^{\varepsilon}}{\left\vert B\left(  x_{2},d\left(
x_{1},y\right)  \right)  \right\vert }.
\]
Hence, let us assume $3d\left(  x_{1},x_{2}\right)  >\max\left(  d\left(
x_{1},y\right)  ,d\left(  x_{2},y\right)  \right)  $. Then by Proposition
\ref{Prop Phi} and Lemma \ref{Lemma nonintegrale}:
\begin{align*}
\left\vert \Phi^{\prime}\left(  x_{1},y\right)  -\Phi^{\prime}\left(
x_{2},y\right)  \right\vert  &  \leqslant c\left\{  \phi_{\alpha}\left(
x_{1},y\right)  +\phi_{\alpha}\left(  x_{2},y\right)  \right\} \\
&  \leqslant c\left\{  \frac{d\left(  x_{1},y\right)  ^{\alpha}}{\left\vert
B\left(  x_{1},d\left(  x_{1},y\right)  \right)  \right\vert }+\frac{d\left(
x_{2},y\right)  ^{\alpha}}{\left\vert B\left(  x_{2},d\left(  x_{2},y\right)
\right)  \right\vert }\right\} \\
&  \leqslant cd\left(  x_{1},x_{2}\right)  ^{\alpha-\varepsilon}\left\{
\frac{d\left(  x_{1},y\right)  ^{\varepsilon}}{\left\vert B\left(
x_{1},d\left(  x_{1},y\right)  \right)  \right\vert }+\frac{d\left(
x_{2},y\right)  ^{\varepsilon}}{\left\vert B\left(  x_{2},d\left(
x_{2},y\right)  \right)  \right\vert }\right\}  .
\end{align*}

\end{proof}

Proposition \ref{Phi holder} will be proved in several steps, establishing
first an analogous result for the functions $Z_{1}^{\prime}$ and
$Z_{k}^{\prime}$.

\begin{lemma}
\label{Prop Z_1 Holder}For every $x_{1},x_{2},y\in U$ with $d\left(
x_{1},y\right)  \geqslant2d\left(  x_{1},x_{2}\right)  $ we have%
\begin{equation}
\left\vert Z_{1}^{\prime}\left(  x_{1},y\right)  -Z_{1}^{\prime}\left(
x_{2},y\right)  \right\vert \leqslant cd\left(  x_{1},x_{2}\right)  ^{\alpha
}\phi_{0}\left(  x_{1},y\right)  . \label{HolderZprimo}%
\end{equation}

\end{lemma}

\begin{proof}
Since%
\[
Z_{1}^{\prime}\left(  x,y\right)  =\frac{1}{c_{0}\left(  x\right)  }%
Z_{1}\left(  x,y\right)
\]
with $c_{0}$ H\"{o}lder continuous and bounded away from zero, it suffices to
prove (\ref{HolderZprimo}) with $Z_{1}^{^{\prime}}$ replaced by $Z_{1}$. Under
assumptions B, the explicit expression of $Z_{1}$ given in the proof of
Proposition \ref{Prop bound Z_1} shows, by Proposition \ref{Prop bad theta},
that%
\[
Z_{1}\left(  \cdot,y\right)  \in C_{loc}^{1,\alpha}\left(  U\setminus\left\{
y\right\}  \right)  .
\]
In particular, for fixed $y,x_{1}$, we have that $Z_{1}\left(  \cdot,y\right)
\in C^{1,\alpha}\left(  B\left(  x_{1},\frac{1}{2}d\left(  x_{1},y\right)
\right)  \right)  $ and we can apply Proposition \ref{Prop Lagrange} with
$R=\frac{1}{2}d\left(  x_{1},y\right)  $, writing
\begin{align}
\left\vert Z_{1}\left(  x_{1},y\right)  -Z_{1}\left(  x_{2},y\right)
\right\vert  &  \leqslant cd\left(  x_{1},x_{2}\right)  \left(  \overset
{n}{\underset{i=1}{\sum}}\sup_{x\in B\left(  x_{1},\frac{1}{2}d\left(
x_{1},y\right)  \right)  }\left\vert X_{i}Z_{1}\left(  x,y\right)  \right\vert
+\right. \label{Lagrange}\\
&  \left.  +d\left(  x_{1},x_{2}\right)  \underset{x\in B\left(  x_{1}%
,\frac{1}{2}d\left(  x_{1},y\right)  \right)  }{\sup}\left\vert X_{0}%
Z_{1}\left(  x,y\right)  \right\vert \right) \nonumber
\end{align}
for $d\left(  x_{1},y\right)  \geqslant2d\left(  x_{1},x_{2}\right)  $. Let us
estimate $\sup_{x\in B\left(  x_{1},\frac{1}{2}d\left(  x_{1},y\right)
\right)  }\left\vert X_{i}Z_{1}\left(  x,y\right)  \right\vert $. We know
that:%
\[
Z_{1}\left(  x,y\right)  =\sum_{i=1}^{3}\int_{\mathbb{R}^{m}}\int
_{\mathbb{R}^{m}}Q_{i}\left(  y,k;~x,h\right)  \varphi_{i}\left(  h\right)
\varphi\left(  k\right)  dhdk,
\]
where the $Q_{i}$'s are defined in the proof of Proposition
\ref{Prop bound Z_1}. Let us bound $X_{i}Z_{1}$ for one of the terms $Q_{i},$
for instance%
\[
R_{j}^{\left(  y,k\right)  }R_{j}^{\left(  y,k\right)  }\Gamma\left(
\Theta_{\left(  y,k\right)  }\left(  x,h\right)  \right)
\]
(since the other terms do not behave worse than this). We have, for
$i=1,2,...,n,$%
\begin{align}
&  X_{i}\int_{\mathbb{R}^{m}}\int_{\mathbb{R}^{m}}R_{j}^{\left(  y,k\right)
}R_{j}^{\left(  y,k\right)  }\Gamma\left(  \Theta_{\left(  y,k\right)
}\left(  x,h\right)  \right)  \varphi\left(  h\right)  \varphi\left(
k\right)  dhdk\nonumber\\
&  =\int_{\mathbb{R}^{m}}\int_{\mathbb{R}^{m}}\widetilde{X}_{i}\left[
R_{j}^{\left(  y,k\right)  }R_{j}^{\left(  y,k\right)  }\Gamma\left(
\Theta_{\left(  y,k\right)  }\left(  x,h\right)  \right)  \varphi\left(
h\right)  \right]  \varphi\left(  k\right)  dhdk\nonumber\\
&  =\int_{\mathbb{R}^{m}}\int_{\mathbb{R}^{m}}R_{j}^{\left(  y,k\right)
}R_{j}^{\left(  y,k\right)  }\Gamma\left(  \Theta_{\left(  y,k\right)
}\left(  x,h\right)  \right)  \left(  \widetilde{X}\varphi_{i}\right)  \left(
h\right)  \varphi\left(  k\right)  dhdk\label{holder 1}\\
&  +\int_{\mathbb{R}^{m}}\int_{\mathbb{R}^{m}}\left(  Y_{i}R_{j}^{\left(
y,k\right)  }R_{j}^{\left(  y,k\right)  }\Gamma\right)  \left(  \Theta
_{\left(  y,k\right)  }\left(  x,h\right)  \right)  \varphi\left(  h\right)
\varphi\left(  k\right)  dhdk\nonumber\\
&  +\int_{\mathbb{R}^{m}}\int_{\mathbb{R}^{m}}\left(  R_{i}^{\left(
y,k\right)  }R_{j}^{\left(  y,k\right)  }R_{j}^{\left(  y,k\right)  }%
\Gamma\right)  \left(  \Theta_{\left(  y,k\right)  }\left(  x,h\right)
\right)  \varphi\left(  h\right)  \varphi\left(  k\right)  dhdk.\nonumber
\end{align}
Now, by Proposition \ref{Prop bad theta}, (ii),%
\begin{align*}
&  \left\vert \int_{\mathbb{R}^{m}}\int_{\mathbb{R}^{m}}\left(  Y_{i}%
R_{j}^{\left(  y,k\right)  }R_{j}^{\left(  y,k\right)  }\Gamma\right)  \left(
\Theta_{\left(  y,k\right)  }\left(  x,h\right)  \right)  \varphi\left(
h\right)  \varphi\left(  k\right)  dhdk\right\vert \\
&  \leqslant c\int_{\mathbb{R}^{m}}\int_{\mathbb{R}^{m}}\frac{\varphi\left(
h\right)  \varphi\left(  k\right)  }{\left\Vert \Theta_{\left(  y,k\right)
}\left(  x,h\right)  \right\Vert ^{Q+1-\alpha}}dhdk\leqslant c\int_{d\left(
x,y\right)  }^{R}\frac{r^{\alpha-2}}{\left\vert B\left(  x,r\right)
\right\vert }dr,
\end{align*}
and the other two terms in (\ref{holder 1}) are bounded by the same quantity.
Next, we have to take the supremum of the last quantity for $x\in B\left(
x_{1},\frac{1}{2}d\left(  x_{1},y\right)  \right)  .$ Since $d\left(
x_{1},y\right)  <2d\left(  x,y\right)  $, by (\ref{comparable phi}),
%(\`{e} una propriet\`{a} che abbiamo dimostrato verso
%l'inizio della dimostrazione del teoremone 12),
this sup is bounded by%
\[
c\int_{d\left(  x_{1},y\right)  }^{R}\frac{r^{\alpha-2}}{\left\vert B\left(
x,r\right)  \right\vert }dr,
\]
hence
\[
d\left(  x_{1},x_{2}\right)  \overset{n}{\underset{i=1}{\sum}}\sup_{x\in
B\left(  x_{1},\frac{1}{2}d\left(  x_{1},y\right)  \right)  }\left\vert
X_{i}Z_{1}\left(  x,y\right)  \right\vert \leqslant cd\left(  x_{1}%
,x_{2}\right)  \int_{d\left(  x_{1},y\right)  }^{R}\frac{r^{\alpha-2}%
}{\left\vert B\left(  x,r\right)  \right\vert }dr
\]
since $d\left(  x_{1},x_{2}\right)  \leqslant\frac{1}{2}d\left(
x_{1},y\right)  <r$%
\begin{align*}
&  \leqslant cd\left(  x_{1},x_{2}\right)  ^{\alpha}\int_{d\left(
x_{1},y\right)  }^{R}r^{1-\alpha}\frac{r^{\alpha-2}}{\left\vert B\left(
x,r\right)  \right\vert }dr\\
&  =cd\left(  x_{1},x_{2}\right)  ^{\alpha}\int_{d\left(  x_{1},y\right)
}^{R}\frac{r^{-1}}{\left\vert B\left(  x,r\right)  \right\vert }dr\\
&  =cd\left(  x_{1},x_{2}\right)  ^{\alpha}\phi_{0}\left(  x_{1},y\right)  .
\end{align*}
An analogous computation gives
\begin{align*}
d\left(  x_{1},x_{2}\right)  ^{2}\sup_{x\in B\left(  x_{1},\frac{1}{2}d\left(
x_{1},y\right)  \right)  }\left\vert X_{0}Z_{1}\left(  x,y\right)
\right\vert  &  \leqslant cd\left(  x_{1},x_{2}\right)  ^{2}\int_{d\left(
x_{1},y\right)  }^{R}\frac{r^{\alpha-3}}{\left\vert B\left(  x,r\right)
\right\vert }dr\\
&  =cd\left(  x_{1},x_{2}\right)  ^{\alpha}\int_{d\left(  x_{1},y\right)
}^{R}r^{2-\alpha}\frac{r^{\alpha-3}}{\left\vert B\left(  x,r\right)
\right\vert }dr\\
&  =cd\left(  x_{1},x_{2}\right)  ^{\alpha}\phi_{0}\left(  x_{1},y\right)  .
\end{align*}
Then (\ref{Lagrange}) implies%
\[
\left\vert Z_{1}\left(  x_{1},y\right)  -Z_{1}\left(  x_{2},y\right)
\right\vert \leqslant cd\left(  x_{1},x_{2}\right)  ^{\alpha}\phi_{0}\left(
x_{1},y\right)
\]
and the lemma is proved.
\end{proof}

Next we need the following:

\begin{lemma}
\label{Lemma Holder beta}For any $\beta>0,$ let%
\[
A\left(  x_{1},x_{2},y\right)  =\int_{U}\left\vert Z_{1}^{\prime}\left(
x_{1},z\right)  -Z_{1}^{\prime}\left(  x_{2},z\right)  \right\vert \phi
_{\beta}\left(  z,y\right)  dz.
\]
For any $\varepsilon>0$ there exists $c>0$ such that%
\[
A\left(  x_{1},x_{2},y\right)  \leqslant cd\left(  x_{1},x_{2}\right)
^{\alpha-\varepsilon}\phi_{\beta+\varepsilon}\left(  x_{1},y\right)
\]
for $d\left(  x_{1},y\right)  \geqslant3d\left(  x_{1},x_{2}\right)  $.
\end{lemma}

\begin{proof}
Let us split:%
\[
A\left(  x_{1},x_{2},y\right)  =\int_{d\left(  x_{1},z\right)  \geqslant
2d\left(  x_{1},x_{2}\right)  }\left(  \ldots\right)  dz+\int_{d\left(
x_{1},z\right)  <2d\left(  x_{1},x_{2}\right)  }\left(  \ldots\right)
dz\equiv I+II.
\]
By Lemma \ref{Prop Z_1 Holder},%
\[
I\leqslant c\int_{d\left(  x_{1},z\right)  \geqslant2d\left(  x_{1}%
,x_{2}\right)  }d\left(  x_{1},x_{2}\right)  ^{\alpha}\phi_{0}\left(
x_{1},z\right)  \phi_{\beta}\left(  z,y\right)  dz.
\]
Now, for any $\varepsilon>0,$ and $d\left(  x_{1},z\right)  \geqslant2d\left(
x_{1},x_{2}\right)  ,$ we have%
\begin{align*}
&  d\left(  x_{1},x_{2}\right)  ^{\alpha}\phi_{0}\left(  x_{1},z\right)
\leqslant cd\left(  x_{1},x_{2}\right)  ^{\alpha-\varepsilon}d\left(
x_{1},z\right)  ^{\varepsilon}\int_{d\left(  x_{1},z\right)  }^{R}\frac
{r^{-1}}{\left\vert B\left(  x_{1},r\right)  \right\vert }dr\\
&  \leqslant cd\left(  x_{1},x_{2}\right)  ^{\alpha-\varepsilon}\int_{d\left(
x_{1},z\right)  }^{R}\frac{r^{\varepsilon-1}}{\left\vert B\left(
x_{1},r\right)  \right\vert }dr=cd\left(  x_{1},x_{2}\right)  ^{\alpha
-\varepsilon}\phi_{\varepsilon}\left(  x_{1},z\right)
\end{align*}
hence, by Theorem \ref{Thm phi alfa}%
\begin{align}
I  &  \leqslant cd\left(  x_{1},x_{2}\right)  ^{\alpha-\varepsilon}%
\int_{d\left(  x_{1},z\right)  \geqslant2d\left(  x_{1},x_{2}\right)  }%
\phi_{\varepsilon}\left(  x_{1},z\right)  \phi_{\beta}\left(  z,y\right)
dz\label{Holder (I)}\\
&  \leqslant cd\left(  x_{1},x_{2}\right)  ^{\alpha-\varepsilon}\phi
_{\beta+\varepsilon}\left(  x_{1},y\right)  .\nonumber
\end{align}
Next,%
\begin{align*}
II  &  \leqslant\int_{d\left(  x_{1},z\right)  <2d\left(  x_{1},x_{2}\right)
}\left[  \phi_{\alpha}\left(  x_{1},z\right)  +\phi_{\alpha}\left(
x_{2},z\right)  \right]  \phi_{\beta}\left(  z,y\right)  dz\\
&  \equiv II_{A}+II_{B}.
\end{align*}
From $d\left(  x_{1},y\right)  \geqslant3d\left(  x_{1},x_{2}\right)  $ and
$d\left(  x_{1},z\right)  <2d\left(  x_{1},x_{2}\right)  ,$ we deduce
$d\left(  y,z\right)  \geqslant d\left(  x_{1},x_{2}\right)  ,$ hence
$d\left(  x_{1},z\right)  \leqslant2d\left(  y,z\right)  $ and%
\[
d\left(  x_{1},y\right)  \leqslant d\left(  x_{1},z\right)  +d\left(
z,y\right)  \leqslant3d\left(  z,y\right)
\]
which allows us to write%
\[
II_{A}\leqslant c\phi_{\beta}\left(  x_{1},y\right)  \int_{d\left(
x_{1},z\right)  <2d\left(  x_{1},x_{2}\right)  }\phi_{\alpha}\left(
x_{1},z\right)  dz
\]
by Corollary \ref{coroll Psi beta}%
\[
\leqslant c\phi_{\beta}\left(  x_{1},y\right)  d\left(  x_{1},x_{2}\right)
^{\alpha}.
\]
By the same reason,%
\begin{align*}
II_{B}  &  \leqslant c\phi_{\beta}\left(  x_{1},y\right)  \int_{d\left(
x_{1},z\right)  <2d\left(  x_{1},x_{2}\right)  }\phi_{\alpha}\left(
x_{2},z\right)  dz\\
&  \leqslant c\phi_{\beta}\left(  x_{1},y\right)  \int_{d\left(
x_{2},z\right)  <3d\left(  x_{1},x_{2}\right)  }\phi_{\alpha}\left(
x_{2},z\right)  dz\\
&  \leqslant c\phi_{\beta}\left(  x_{1},y\right)  d\left(  x_{1},x_{2}\right)
^{\alpha}%
\end{align*}
as above. We conclude, for $d\left(  x_{1},y\right)  \geqslant3d\left(
x_{1},x_{2}\right)  ,$%
\begin{align*}
II  &  \leqslant c\phi_{\beta}\left(  x_{1},y\right)  d\left(  x_{1}%
,x_{2}\right)  ^{\alpha}\leqslant\\
&  \leqslant cd\left(  x_{1},x_{2}\right)  ^{\alpha-\varepsilon}d\left(
x_{1},y\right)  ^{\varepsilon}\int_{d\left(  x_{1},y\right)  }^{R}%
\frac{r^{\beta-1}}{\left\vert B\left(  x_{1},r\right)  \right\vert }dr\\
&  \leqslant cd\left(  x_{1},x_{2}\right)  ^{\alpha-\varepsilon}\phi
_{\beta+\varepsilon}\left(  x_{1},y\right)  ,
\end{align*}
which together with (\ref{Holder (I)}) gives the assertion.
\end{proof}

\begin{proof}
[Proof of Proposition \ref{Phi holder}]Let $x_{1},x_{2},y\in U$ with $d\left(
x_{1},y\right)  \geqslant3d\left(  x_{1},x_{2}\right)  $. By the identity
(\ref{int Eq prime}) we can write%
\[
\Phi^{\prime}\left(  x_{1},y\right)  -\Phi^{\prime}\left(  x_{2},y\right)
=Z_{1}^{\prime}\left(  x_{1},y\right)  -Z_{1}^{\prime}\left(  x_{2},y\right)
+\int_{U}\left[  Z_{1}^{\prime}\left(  x_{1},z\right)  -Z_{1}^{\prime}\left(
x_{2},z\right)  \right]  \Phi^{\prime}\left(  z,y\right)  dz
\]
which by (\ref{phi bound}) gives%
\[
\left\vert \Phi^{\prime}\left(  x_{1},y\right)  -\Phi^{\prime}\left(
x_{2},y\right)  \right\vert \leqslant\left\vert Z_{1}^{\prime}\left(
x_{1},y\right)  -Z_{1}^{\prime}\left(  x_{2},y\right)  \right\vert +c\int
_{U}\left\vert Z_{1}^{\prime}\left(  x_{1},z\right)  -Z_{1}^{\prime}\left(
x_{2},z\right)  \right\vert \phi_{\alpha}\left(  z,y\right)  dz.
\]
Exploiting Lemmas \ref{Prop Z_1 Holder} and \ref{Lemma Holder beta}, for any
$\varepsilon>0$ we get%
\begin{align*}
\left\vert \Phi^{\prime}\left(  x_{1},y\right)  -\Phi^{\prime}\left(
x_{2},y\right)  \right\vert  &  \leqslant cd\left(  x_{1},x_{2}\right)
^{\alpha}\phi_{0}\left(  x_{1},y\right)  +cd\left(  x_{1},x_{2}\right)
^{\alpha-\varepsilon}\phi_{\beta+\varepsilon}\left(  x_{1},y\right) \\
&  \leqslant cd\left(  x_{1},x_{2}\right)  ^{\alpha-\varepsilon}%
\phi_{\varepsilon}\left(  x_{1},y\right)
\end{align*}
as desired.
\end{proof}

\subsection{Estimates on the second derivatives of the fundamental solution}

We are now going to prove the existence and a sharp bound of H\"{o}lder type
of the second derivatives of our local fundamental solution.

\begin{theorem}
[Second derivatives of the fundamental solution]\label{Thm XXgamma}Under
Assumptions B, for $i,j=1,2,...,n$ and for $x,y\in U,x\neq y$, the following
assertions hold true.

(i) There exist the second derivatives $X_{j}X_{i}J^{\prime}\left(
x,y\right)  $, $X_{0}J^{\prime}\left(  x,y\right)  $, $X_{i}X_{j}\gamma\left(
x,y\right)  $, $X_{0}\gamma\left(  x,y\right)  $ continuous in the joint
variables for $x\neq y$; in particular,%
\[
\gamma\left(  \cdot,y\right)  \in C_{X}^{2}\left(  U\setminus\left\{
y\right\}  \right)  \text{ for any }y\in U\text{.}%
\]

(ii) For every $\varepsilon\in\left(  0,\alpha\right)  $, every $U^{\prime
}\Subset U$ there exists $c>0$ such that for every $x\in U^{\prime}$ and $y\in
U,$%
\begin{equation}
\left\vert X_{j}X_{i}J^{\prime}\left(  x,y\right)  \right\vert ,\left\vert
X_{0}J^{\prime}\left(  x,y\right)  \right\vert \leqslant cR^{\varepsilon}%
\frac{d\left(  x,y\right)  ^{\alpha-\varepsilon}}{\left\vert B\left(
x,d\left(  x,y\right)  \right)  \right\vert } \label{Stima XXJ}%
\end{equation}
with $R$ as at the beginning of \S 3, and%
\begin{equation}
\left\vert X_{j}X_{i}\gamma\left(  x,y\right)  \right\vert ,\left\vert
X_{0}\gamma\left(  x,y\right)  \right\vert \leqslant c\frac{1}{\left\vert
B\left(  x,d\left(  x,y\right)  \right)  \right\vert }. \label{Stima XXgamma}%
\end{equation}

\end{theorem}

Note the presence, at the right-hand side of (\ref{Stima XXJ}),
(\ref{Stima XXgamma}), of the kernels $d\left(  x,y\right)  ^{\alpha
-\varepsilon}\left\vert B\left(  x,d\left(  x,y\right)  \right)  \right\vert
^{-1}$, $\left\vert B\left(  x,d\left(  x,y\right)  \right)  \right\vert
^{-1}$, instead of $\phi_{\alpha-\varepsilon}\left(  x,y\right)  ,\phi
_{0}\left(  x,y\right)  $, which one could expect.

In order to reduce the length of some computation in the proof of this theorem
and some of the following ones, it is convenient to introduce first the
following abstract definitions, and make a preliminary study of the involved concept.

\begin{definition}
\label{Def abstract remainder}We say that $R_{\ell}\left(  x,y\right)  $ is a
remainder of type $\ell$ ($=0,1,2,3$) if for $x\neq y$%
\[
R_{\ell}\left(  x,y\right)  =\sum_{s=1}^{m}\int_{\mathbb{R}^{m}}%
\int_{\mathbb{R}^{m}}D_{\ell,s}^{\left(  y,k\right)  }\Gamma\left(
\Theta_{\left(  y,k\right)  }\left(  x,h\right)  \right)  a_{s}\left(
h\right)  b_{s}\left(  k\right)  dhdk
\]
where $D_{\ell,s}^{\left(  y,k\right)  }$ are differential operators given by
the composition of at most $\ell$ vector fields of the kind $Y_{i}$ or
$R_{i}^{\left(  y,k\right)  }$, of total weight $\geqslant\alpha-\ell$,
depending on $\left(  y,k\right)  $ in a $C^{\alpha}$ way and $a_{s}$, $b_{s}$
are cutoff functions. Here and in the following, the number $\alpha$ is fixed,
and is the exponent appearing in Assumptions B.
\end{definition}

\begin{definition}
\label{Def abstract k_l}We say that $k_{\ell}\left(  x,y\right)  $ is a
\emph{kernel of type }$\ell$ ($=0,1,2,3$) if for $x\neq y$%
\[
k_{\ell}\left(  x,y\right)  =\int_{\mathbb{R}^{m}}\int_{\mathbb{R}^{m}}%
D_{\ell}\Gamma\left(  \Theta_{\left(  y,k\right)  }\left(  x,h\right)
\right)  a_{0}\left(  h\right)  b_{0}\left(  k\right)  dhdk+R_{\ell}\left(
x,y\right)
\]
where $D_{\ell}$ is a left invariant differential operator homogeneous of
degree $\ell$, $a_{0}$, $b_{0}$ are cutoff functions and $R_{\ell}\left(
x,y\right)  $ is a remainder of type $\ell$. If $R_{\ell}\left(  x,y\right)
\equiv0$, we say that $k_{\ell}\left(  x,y\right)  $ is a \emph{pure kernel of
type }$\ell$.
\end{definition}

\begin{theorem}
\label{Thm Abstract k_l}Under Assumptions B, let $k_{\ell}\left(  x,y\right)
$ be a kernel of type $\ell$. Then for $x\neq y,$ $k_{\ell}\left(  x,y\right)
$ is jointly continuous and satisfies the bound:%
\[
\left\vert k_{\ell}\left(  x,y\right)  \right\vert \leqslant c\phi_{2-\ell
}\left(  x,y\right)  .
\]
Moreover, if $\ell\leqslant2$,$\mathbf{\ }$then $X_{i}k_{\ell}\left(
x,y\right)  $ is a kernel of type $\ell+1$ for $i=1,2,...,n$; if
$\ell\leqslant1$, then $X_{0}k_{\ell}\left(  x,y\right)  $ is a kernel of type
$\ell+2$.

Let $R_{\ell}\left(  x,y\right)  $ be a remainder of type $\ell=0,1,2,3$.
Then, for $x\neq y$, $R_{\ell}\left(  x,y\right)  $ is jointly continuous and
satisfies the bound:%
\[
\left\vert R_{\ell}\left(  x,y\right)  \right\vert \leqslant c\phi
_{2+\alpha-\ell}\left(  x,y\right)  .
\]
Also, if $\ell\leqslant2$,$\mathbf{\ }$then $X_{i}R_{\ell}\left(  x,y\right)
$ is a remainder of type $\ell+1$ for $i=1,2,...,n$; if $\ell\leqslant1$, then
$X_{0}R_{\ell}\left(  x,y\right)  $ is a remainder of type $\ell+2$.
\end{theorem}

\begin{proof}
The continuity properties follow as in the proof of Proposition \ref{Prop P}.
Also, we have%
\begin{align*}
\left\vert D_{\ell}\Gamma\left(  \Theta_{\left(  y,k\right)  }\left(
x,h\right)  \right)  \right\vert  &  \leqslant\frac{c}{\left\Vert
\Theta_{\left(  y,k\right)  }\left(  x,h\right)  \right\Vert ^{Q-2+\ell}},\\
\left\vert D_{\ell,s}^{\left(  y,k\right)  }\Gamma\left(  \Theta_{\left(
y,k\right)  }\left(  x,h\right)  \right)  \right\vert  &  \leqslant\frac
{c}{\left\Vert \Theta_{\left(  y,k\right)  }\left(  x,h\right)  \right\Vert
^{Q-2+\ell-\alpha}}%
\end{align*}
hence by Lemma \ref{Lemma NSW} we have%
\begin{align*}
\left\vert k_{\ell}\left(  x,y\right)  \right\vert  &  \leqslant c\phi
_{2-\ell}\left(  x,y\right)  ,\\
\left\vert R_{\ell}\left(  x,y\right)  \right\vert  &  \leqslant c\phi
_{2-\ell+\alpha}\left(  x,y\right)  .
\end{align*}
Let us compute, for $x\neq y,$%
\begin{align*}
X_{i}k_{\ell}\left(  x,y\right)   &  =\int_{\mathbb{R}^{m}}\int_{\mathbb{R}%
^{m}}\widetilde{X}_{i}\left[  D_{\ell}\Gamma\left(  \Theta_{\left(
y,k\right)  }\left(  x,h\right)  \right)  a_{0}\left(  h\right)  \right]
b_{0}\left(  k\right)  dhdk\\
&  +\sum_{s=1}^{m}\int_{\mathbb{R}^{m}}\int_{\mathbb{R}^{m}}\widetilde{X}%
_{i}\left[  D_{\ell,s}^{\left(  y,k\right)  }\Gamma\left(  \Theta_{\left(
y,k\right)  }\left(  x,h\right)  \right)  a_{s}\left(  h\right)  \right]
b_{s}\left(  k\right)  dhdk
\end{align*}%
\begin{align*}
&  =\int_{\mathbb{R}^{m}}\int_{\mathbb{R}^{m}}Y_{i}D_{\ell}\Gamma\left(
\Theta_{\left(  y,k\right)  }\left(  x,h\right)  \right)  a_{0}\left(
h\right)  b_{0}\left(  k\right)  dhdk\\
&  +\int_{\mathbb{R}^{m}}\int_{\mathbb{R}^{m}}R_{i}^{\left(  y,k\right)
}D_{\ell}\Gamma\left(  \Theta_{\left(  y,k\right)  }\left(  x,h\right)
\right)  a_{0}\left(  h\right)  b_{0}\left(  k\right)  dhdk\\
&  +\int_{\mathbb{R}^{m}}\int_{\mathbb{R}^{m}}D_{\ell}\Gamma\left(
\Theta_{\left(  y,k\right)  }\left(  x,h\right)  \right)  \widetilde{X}%
_{i}a_{0}\left(  h\right)  b_{0}\left(  k\right)  dhdk\\
&  +\sum_{s=1}^{m}\int_{\mathbb{R}^{m}}\int_{\mathbb{R}^{m}}Y_{i}D_{\ell
,s}^{\left(  y,k\right)  }\Gamma\left(  \Theta_{\left(  y,k\right)  }\left(
x,h\right)  \right)  a_{s}\left(  h\right)  b_{s}\left(  k\right)  dhdk\\
&  +\sum_{s=1}^{m}\int_{\mathbb{R}^{m}}\int_{\mathbb{R}^{m}}R_{i}^{\left(
y,k\right)  }D_{\ell,s}^{\left(  y,k\right)  }\Gamma\left(  \Theta_{\left(
y,k\right)  }\left(  x,h\right)  \right)  a_{s}\left(  h\right)  b_{s}\left(
k\right)  dhdk\\
&  +\sum_{s=1}^{m}\int_{\mathbb{R}^{m}}\int_{\mathbb{R}^{m}}D_{\ell
,s}^{\left(  y,k\right)  }\Gamma\left(  \Theta_{\left(  y,k\right)  }\left(
x,h\right)  \right)  \widetilde{X}_{i}a_{s}\left(  h\right)  b_{s}\left(
k\right)  dhdk
\end{align*}
by Proposition \ref{Prop bad theta} and Definition
\ref{Def abstract remainder}%
\begin{align*}
&  =\int_{\mathbb{R}^{m}}\int_{\mathbb{R}^{m}}D_{\ell+1}\Gamma\left(
\Theta_{\left(  y,k\right)  }\left(  x,h\right)  \right)  a_{0}\left(
h\right)  b_{0}\left(  k\right)  dhdk\\
&  +\sum_{s=1}^{m^{\prime}}\int_{\mathbb{R}^{m}}\int_{\mathbb{R}^{m}}%
D_{\ell+1,s}^{\left(  y,k\right)  }\Gamma\left(  \Theta_{\left(  y,k\right)
}\left(  x,h\right)  \right)  a_{s}^{\prime}\left(  h\right)  b_{s}^{\prime
}\left(  k\right)  dhdk
\end{align*}
which gives the desired result for $X_{i}k_{\ell}$; analogously one can handle
$X_{0}k_{\ell}$.
\end{proof}

\begin{definition}
Let $\Phi_{0}:\left\{  \left(  x,y\right)  \in U\times U:x\neq y\right\}
\rightarrow\mathbb{R}$. We say that $\Phi_{0}$ is a function of $\left(
\phi,\alpha\right)  $-type if it is continuous (in the joint variables),
satisfies%
\[
\left\vert \Phi_{0}\left(  x,y\right)  \right\vert \leqslant c\phi_{\alpha
}\left(  x,y\right)
\]
and for every $\varepsilon\in\left(  0,\alpha\right)  $ there exists a
constant $c_{\varepsilon}$ such that for every $x_{1},x_{2},y\in U$ with
$d\left(  x_{1},y\right)  \geqslant3d\left(  x_{1},x_{2}\right)  $%
\[
\left\vert \Phi_{0}\left(  x_{1},y\right)  -\Phi_{0}\left(  x_{2},y\right)
\right\vert \leqslant c_{\varepsilon}d\left(  x_{1},x_{2}\right)
^{\alpha-\varepsilon}\phi_{\varepsilon}\left(  x_{1},y\right)  .
\]

\end{definition}

\begin{lemma}
\label{Lemma Phi holder}Let $\Phi_{0}$ be a $\left(  \phi,\alpha\right)
$-type function. For every $\varepsilon\in\left(  0,\alpha\right)  $ there
exists $c_{\varepsilon}$ such that for every $x_{1},x_{2},y\in U$ we have%
\[
\left\vert \Phi_{0}\left(  x_{1},y\right)  -\Phi_{0}\left(  x_{2},y\right)
\right\vert \leqslant c_{\varepsilon}d\left(  x_{1},x_{2}\right)
^{\alpha-\varepsilon}\left[  \frac{d\left(  x_{1},y\right)  ^{\varepsilon}%
}{\left\vert B\left(  x_{1},d\left(  x_{1},y\right)  \right)  \right\vert
}+\frac{d\left(  x_{2},y\right)  ^{\varepsilon}}{\left\vert B\left(
x_{2},d\left(  x_{2},y\right)  \right)  \right\vert }\right]  .
\]

\end{lemma}

The Lemma follows from the above definition as in the proof of Corollary
\ref{coroll Phi holder}. We will also need the following easy

\begin{lemma}
\label{NSW mod} If $\beta\in\mathbb{R}$ and $\varepsilon>0$, then there exists
$c>0$ such that%
\begin{equation}
\int_{\mathbb{R}^{m}}\int_{\mathbb{R}^{m}}\frac{\psi\left(  h\right)
\varphi\left(  k\right)  }{\left\Vert \Theta_{\left(  y,k\right)  }\left(
x,h\right)  \right\Vert ^{Q-\beta}}\chi_{\left\{  h:\left\Vert \Theta_{\left(
y,k\right)  }\left(  x,h\right)  \right\Vert <\delta\right\}  }dhdk\leqslant
c\delta^{\varepsilon}\phi_{\beta-\varepsilon}\left(  x,y\right)  .
\label{Geo1}%
\end{equation}

\end{lemma}

\begin{proof}
To prove (\ref{Geo1}) it is enough to observe that%
\begin{align*}
&  \int_{\mathbb{R}^{m}}\int_{\mathbb{R}^{m}}\frac{\psi\left(  h\right)
\varphi\left(  k\right)  }{\left\Vert \Theta_{\left(  y,k\right)  }\left(
x,h\right)  \right\Vert ^{Q-\beta}}\chi_{\left\{  h:\left\Vert \Theta_{\left(
y,k\right)  }\left(  x,h\right)  \right\Vert <\delta\right\}  }dhdk\\
&  \leqslant\delta^{\varepsilon}\int_{\mathbb{R}^{m}}\int_{\mathbb{R}^{m}%
}\frac{\psi\left(  h\right)  \varphi\left(  k\right)  }{\left\Vert
\Theta_{\left(  y,k\right)  }\left(  x,h\right)  \right\Vert ^{Q-\beta
+\varepsilon}}\chi_{\left\{  h:\left\Vert \Theta_{\left(  y,k\right)  }\left(
x,h\right)  \right\Vert <\delta\right\}  }dhdk\leqslant c\delta^{\varepsilon
}\phi_{\beta-\varepsilon}\left(  x,y\right)
\end{align*}
by Lemma \ref{Lemma NSW}.
\end{proof}

\begin{theorem}
\label{Thm Big J_0}Let $k$ be a kernel of type $\ell=0$ and let $\Phi_{0}$ be
a function of $\left(  \phi,\alpha\right)  $-type. If
\[
J_{0}\left(  x,y\right)  =\int_{U}k\left(  x,z\right)  \Phi_{0}\left(
z,y\right)  dz,
\]
then for $i=1,2,...,n,$%
\[
X_{i}J_{0}\left(  x,y\right)  =\int_{U}X_{i}k\left(  x,z\right)  \Phi
_{0}\left(  z,y\right)  dz
\]
and%
\[
\left\vert X_{i}J_{0}\left(  x,y\right)  \right\vert \leqslant c\phi
_{1+\alpha}\left(  x,y\right)  .
\]
Let $\omega_{\delta}\in C^{\infty}\left(  \mathbb{G}\right)  $ such that
$\omega_{\delta}\left(  u\right)  =1$ for $\left\Vert u\right\Vert >\delta$
and $\omega_{\delta}\left(  u\right)  =0$ for $\left\Vert u\right\Vert
<\delta/2$, then, for $i,j=1,2,...,n$, $X_{j}X_{i}J_{0}\left(  x,y\right)  $
exists and is continuous in the joint variables for $x\neq y$ and can be
computed as follows%
\begin{align}
&  X_{j}X_{i}J_{0}\left(  x,y\right) \nonumber\\
&  =\lim_{\delta\rightarrow0}\int_{U}\int_{\mathbb{R}^{m}}\int_{\mathbb{R}%
^{m}}Y_{j}\left(  \omega_{\delta}D_{1}\Gamma\right)  \left(  \Theta_{\left(
z,k\right)  }\left(  x,h\right)  \right)  a_{0}\left(  h\right)  b_{0}\left(
k\right)  dhdk\,\Phi_{0}\left(  z,y\right)  dz\nonumber\\
&  +\int_{U}R_{2}^{\prime}\left(  x,z\right)  \Phi_{0}\left(  z,y\right)
dz\nonumber\\
&  =\int_{U}\int_{\mathbb{R}^{m}}\int_{\mathbb{R}^{m}}\left(  Y_{j}D_{1}%
\Gamma\right)  \left(  \Theta_{\left(  z,k\right)  }\left(  x,h\right)
\right)  a_{0}\left(  h\right)  b_{0}\left(  k\right)  dhdk\,\left[  \Phi
_{0}\left(  z,y\right)  -\Phi_{0}\left(  x,y\right)  \right]  dz\nonumber\\
&  +C\left(  x\right)  \Phi_{0}\left(  x,y\right)  +\int_{U}R_{2}%
^{\prime\prime}\left(  x,z\right)  \Phi_{0}\left(  z,y\right)  dz
\label{Big J_0 1}%
\end{align}
where $D_{1}$ is a left invariant homogeneous vector field of degree $1$,
$R_{2}^{\prime}\left(  x,z\right)  $ and $R_{2}^{\prime\prime}\left(
x,z\right)  $ are suitable remainders of type $2$ and $C\in C_{X,loc}^{\alpha
}\left(  U\right)  $. Moreover, for any $U^{\prime}\Subset U$ there exists
$c>0$ such that for every $x\in U^{\prime}$, $y\in U$, $x\neq y$
\begin{equation}
\left\vert X_{j}X_{i}J_{0}\left(  x,y\right)  \right\vert \leqslant
c\frac{d\left(  x,y\right)  ^{\alpha-\varepsilon}}{\left\vert B\left(
x,d\left(  x,y\right)  \right)  \right\vert }\text{.} \label{Big J_0 2}%
\end{equation}

\end{theorem}

\begin{proof}
Since $X_{i}k\equiv k_{1}$ is a kernel of type $1$ we have%
\[
\left\vert X_{i}k\left(  x,y\right)  \right\vert \leqslant c\phi_{1}\left(
x,y\right)  \leqslant c\frac{d\left(  x,y\right)  }{\left\vert B\left(
x,d\left(  x,y\right)  \right)  \right\vert },
\]
so that we can differentiate under the integral sign. Therefore%
\[
X_{i}J_{0}\left(  x,y\right)  =\int_{U}k_{1}\left(  x,z\right)  \Phi
_{0}\left(  z,y\right)  dz
\]
with%
\[
\left\vert X_{i}J_{0}\left(  x,y\right)  \right\vert \leqslant c\phi
_{1+\alpha}\left(  x,y\right)  .
\]
In order to compute $X_{j}X_{i}J_{0}\left(  x,y\right)  $ we rewrite%
\[
k_{1}\left(  x,z\right)  =\int_{\mathbb{R}^{m}}\int_{\mathbb{R}^{m}}%
D_{1}\Gamma\left(  \Theta_{\left(  z,k\right)  }\left(  x,h\right)  \right)
a_{0}\left(  h\right)  b_{0}\left(  k\right)  dhdk+R_{1}\left(  x,z\right)
\]
where $R_{1}\left(  x,z\right)  $ is a remainder of type $1$. Then%
\begin{align*}
X_{i}J_{0}\left(  x,y\right)   &  =\int_{U}\int_{\mathbb{R}^{m}}%
\int_{\mathbb{R}^{m}}D_{1}\Gamma\left(  \Theta_{\left(  z,k\right)  }\left(
x,h\right)  \right)  a_{0}\left(  h\right)  b_{0}\left(  k\right)
dhdk\,\Phi_{0}\left(  z,y\right)  dz\\
&  +\int_{U}R_{1}\left(  x,z\right)  \Phi_{0}\left(  z,y\right)  dz\\
&  \equiv B_{1}\left(  x,y\right)  +B_{2}\left(  x,y\right)  .
\end{align*}
As to $B_{2}$ we can simply write%
\begin{align*}
X_{j}B_{2}\left(  x,y\right)   &  =\int_{U}X_{j}R_{1}\left(  x,z\right)
\Phi_{0}\left(  z,y\right)  dz\\
&  \equiv\int_{U}R_{2}\left(  x,z\right)  \Phi_{0}\left(  z,y\right)  dz
\end{align*}
where $R_{2}\left(  x,z\right)  $ is a remainder of type $2$.

To handle $B_{1}\left(  x,y\right)  $ we consider%
\[
B_{1}^{\delta}\left(  x,y\right)  =\int_{U}\int_{\mathbb{R}^{m}}%
\int_{\mathbb{R}^{m}}\left(  \omega_{\delta}D_{1}\Gamma\right)  \left(
\Theta_{\left(  z,k\right)  }\left(  x,h\right)  \right)  a_{0}\left(
h\right)  b_{0}\left(  k\right)  dhdk\,\Phi_{0}\left(  z,y\right)  dz.
\]
Due to the presence of this cutoff function, we can compute the derivative%
\begin{align*}
X_{j}B_{1}^{\delta}\left(  x,y\right)   &  =\int_{U}\int_{\mathbb{R}^{m}}%
\int_{\mathbb{R}^{m}}\widetilde{X}_{j}\left[  \left(  \omega_{\delta}%
D_{1}\Gamma\right)  \left(  \Theta_{\left(  z,k\right)  }\left(  x,h\right)
\right)  a_{0}\left(  h\right)  \right]  b_{0}\left(  k\right)  dhdk\,\Phi
_{0}\left(  z,y\right)  dz\\
&  =\int_{U}\int_{\mathbb{R}^{m}}\int_{\mathbb{R}^{m}}Y_{j}\left(
\omega_{\delta}D_{1}\Gamma\right)  \left(  \Theta_{\left(  z,k\right)
}\left(  x,h\right)  \right)  a_{0}\left(  h\right)  b_{0}\left(  k\right)
dhdk\,\Phi_{0}\left(  z,y\right)  dz\\
&  +\int_{U}\int_{\mathbb{R}^{m}}\int_{\mathbb{R}^{m}}R_{j}^{\left(
y,k\right)  }\left(  \omega_{\delta}D_{1}\Gamma\right)  \left(  \Theta
_{\left(  z,k\right)  }\left(  x,h\right)  \right)  a_{0}\left(  h\right)
b_{0}\left(  k\right)  dhdk\,\Phi_{0}\left(  z,y\right)  dz\\
&  +\int_{U}\int_{\mathbb{R}^{m}}\int_{\mathbb{R}^{m}}\left(  \omega_{\delta
}D_{1}\Gamma\right)  \left(  \Theta_{\left(  z,k\right)  }\left(  x,h\right)
\right)  \widetilde{X}_{j}a_{0}\left(  h\right)  \,b_{0}\left(  k\right)
dhdk\,\Phi_{0}\left(  z,y\right)  dz\\
&  =B_{1,1}^{\delta}\left(  x,y\right)  +B_{1,2}^{\delta}\left(  x,y\right)
+B_{1,3}^{\delta}\left(  x,y\right)  .
\end{align*}
An argument similar to one already used shows that for any fixed $\delta$ the
function $X_{j}B_{1}^{\delta}\left(  x,y\right)  $ is continuous in the joint
variables for any $x,y\in U$, $x\neq y$.

First of all we observe that%
\[
\lim_{\delta\rightarrow0}B_{1,2}^{\delta}\left(  x,y\right)  =\int_{U}%
\int_{\mathbb{R}^{m}}\int_{\mathbb{R}^{m}}\left(  R_{j}^{\left(  y,k\right)
}D_{1}\Gamma\right)  \left(  \Theta_{\left(  z,k\right)  }\left(  x,h\right)
\right)  a_{0}\left(  h\right)  b_{0}\left(  k\right)  dhdk\,\Phi_{0}\left(
z,y\right)  dz
\]
since%
\[
\left(  R_{j}^{\left(  y,k\right)  }\omega_{\delta}\right)  \left(  u\right)
\neq0\text{ for }\frac{\delta}{2}<\left\Vert u\right\Vert <\delta\text{,}%
\]
hence%
\begin{align*}
&  \left\vert \int_{U}\int_{\mathbb{R}^{m}}\int_{\mathbb{R}^{m}}\left(
\left(  R_{j}^{\left(  y,k\right)  }\omega_{\delta}\right)  D_{1}%
\Gamma\right)  \left(  \Theta_{\left(  z,k\right)  }\left(  x,h\right)
\right)  a_{0}\left(  h\right)  b_{0}\left(  k\right)  dhdk\,\Phi_{0}\left(
z,y\right)  dz\right\vert \\
&  \leqslant\int_{\mathbb{R}^{m}}\int_{\left\Vert \Theta_{\left(  z,k\right)
}\left(  x,h\right)  \right\Vert <\delta}\frac{c}{\left\Vert \Theta_{\left(
z,k\right)  }\left(  x,h\right)  \right\Vert ^{Q-\alpha}}\left\vert
a_{0}\left(  h\right)  b_{0}\left(  k\right)  \Phi_{0}\left(  z,y\right)
\right\vert dk\,dzdh\\
&  \leqslant\int_{\mathbb{R}^{m}}\int_{\left\Vert \Theta_{\left(  z,k\right)
}\left(  x,h\right)  \right\Vert <\delta}\frac{c}{\left\Vert \Theta_{\left(
z,k\right)  }\left(  x,h\right)  \right\Vert ^{Q-\alpha}}dk\,dz\left\vert
a_{0}\left(  h\right)  \right\vert dh\leqslant c\delta^{\alpha}\rightarrow
0\text{ as }\delta\rightarrow0
\end{align*}
for some constant $c$ depending on $d\left(  x,y\right)  $.

Also%
\[
\lim_{\delta\rightarrow0}B_{1,3}^{\delta}\left(  x,y\right)  =\int_{U}%
\int_{\mathbb{R}^{m}}\int_{\mathbb{R}^{m}}\left(  D_{1}\Gamma\right)  \left(
\Theta_{\left(  z,k\right)  }\left(  x,h\right)  \right)  \widetilde{X}%
_{j}a_{0}\left(  h\right)  \,b_{0}\left(  k\right)  dhdk\,\Phi_{0}\left(
z,y\right)  dz
\]
so that%
\[
\lim_{\delta\rightarrow0}X_{j}B_{1}^{\delta}\left(  x,y\right)  =\lim
_{\delta\rightarrow0}B_{1,1}^{\delta}\left(  x,y\right)  +\int_{U}R_{3}\left(
x,z\right)  \Phi_{0}\left(  z,y\right)  dz
\]
where $R_{3}\left(  x,z\right)  $ is still another remainder of type $2$.

Let us now consider
\[
B_{1,1}^{\delta}\left(  x,y\right)  =\int_{U}\int_{\mathbb{R}^{m}}%
\int_{\mathbb{R}^{m}}Y_{j}\left(  \omega_{\delta}D_{1}\Gamma\right)  \left(
\Theta_{\left(  z,k\right)  }\left(  x,h\right)  \right)  a_{0}\left(
h\right)  b_{0}\left(  k\right)  dhdk\,\Phi_{0}\left(  z,y\right)  dz.
\]
We write%
\begin{align*}
B_{1,1}^{\delta}\left(  x,y\right)   &  =\int_{U}\int_{\mathbb{R}^{m}}%
\int_{\mathbb{R}^{m}}Y_{j}\left(  \omega_{\delta}D_{1}\Gamma\right)  \left(
\Theta_{\left(  z,k\right)  }\left(  x,h\right)  \right)  a_{0}\left(
h\right)  b_{0}\left(  k\right)  dhdk\,\left[  \Phi_{0}\left(  z,y\right)
-\Phi_{0}\left(  x,y\right)  \right]  dz\\
&  +\Phi_{0}\left(  x,y\right)  \int_{U}\int_{\mathbb{R}^{m}}\int
_{\mathbb{R}^{m}}Y_{j}\left(  \omega_{\delta}D_{1}\Gamma\right)  \left(
\Theta_{\left(  z,k\right)  }\left(  x,h\right)  \right)  a_{0}\left(
h\right)  b_{0}\left(  k\right)  dhdk\,dz\\
&  \equiv B_{1,1,1}^{\delta}\left(  x,y\right)  +B_{1,1,2}^{\delta}\left(
x,y\right)  .
\end{align*}
We have%
\begin{align*}
B_{1,1,1}^{\delta}\left(  x,y\right)   &  =\int_{U}\int_{\mathbb{R}^{m}}%
\int_{\mathbb{R}^{m}}\left(  Y_{j}\omega_{\delta}\cdot D_{1}\Gamma\right)
\left(  \Theta_{\left(  z,k\right)  }\left(  x,h\right)  \right)  a_{0}\left(
h\right)  b_{0}\left(  k\right)  dhdk\,\left[  \Phi_{0}\left(  z,y\right)
-\Phi_{0}\left(  x,y\right)  \right]  dz\\
&  +\int_{U}\int_{\mathbb{R}^{m}}\int_{\mathbb{R}^{m}}\left(  \omega_{\delta
}\,Y_{j}D_{1}\Gamma\right)  \left(  \Theta_{\left(  z,k\right)  }\left(
x,h\right)  \right)  a_{0}\left(  h\right)  b_{0}\left(  k\right)
dhdk\,\left[  \Phi_{0}\left(  z,y\right)  -\Phi_{0}\left(  x,y\right)
\right]  dz.
\end{align*}
Since $Y_{j}\omega_{\delta}\left(  \Theta_{\left(  z,k\right)  }\left(
x,h\right)  \right)  $ is supported in $\left\{  \frac{\delta}{2}<\left\Vert
\Theta_{\left(  z,k\right)  }\left(  x,h\right)  \right\Vert <\delta\right\}
$ and bounded by $\delta^{-1}$, by Lemma \ref{NSW mod} , Corollary
\ref{coroll Phi holder} and Lemma \ref{Lemma nonintegrale} the first term of
$B_{1,1,1}^{\delta}\left(  x,y\right)  $ is bounded by%
\begin{align*}
&  c\int_{U}\int_{\mathbb{R}^{m}}\int_{\mathbb{R}^{m}}\chi_{\left\{
\delta/2\leqslant\left\Vert \Theta_{\left(  z,k\right)  }\left(  x,h\right)
\right\Vert \leqslant\delta\right\}  }\left\Vert \Theta_{\left(  z,k\right)
}\left(  x,h\right)  \right\Vert ^{-Q}a_{0}\left(  h\right)  b_{0}\left(
k\right)  dhdk\,\left\vert \Phi_{0}\left(  z,y\right)  -\Phi_{0}\left(
x,y\right)  \right\vert dz\\
&  \leqslant c\int_{U}\delta^{\varepsilon}\phi_{-\varepsilon}\left(
x,z\right)  \left\vert \Phi_{0}\left(  z,y\right)  -\Phi_{0}\left(
x,y\right)  \right\vert dz\\
&  \leqslant c\delta^{\varepsilon}\int_{U}\frac{d\left(  x,z\right)
^{\alpha-2\varepsilon}}{\left\vert B\left(  x,d\left(  x,z\right)  \right)
\right\vert }\left(  \frac{d\left(  z,y\right)  ^{\varepsilon}}{\left\vert
B\left(  z,d\left(  z,y\right)  \right)  \right\vert }+\frac{d\left(
x,y\right)  ^{\varepsilon}}{\left\vert B\left(  x,d\left(  x,y\right)
\right)  \right\vert }\right)  dz.
\end{align*}
Since this last integral converges the first term in $B_{1,1,1}^{\delta
}\left(  x,y\right)  $ vanishes uniformly in $x$ (as long as $x$ stays away
from $y$) as $\delta\rightarrow0$. We will show now that the second term
converges uniformly to%
\[
\int_{U}\int_{\mathbb{R}^{m}}\int_{\mathbb{R}^{m}}\left(  Y_{j}D_{1}%
\Gamma\right)  \left(  \Theta_{\left(  z,k\right)  }\left(  x,h\right)
\right)  a_{0}\left(  h\right)  b_{0}\left(  k\right)  dhdk\,\left[  \Phi
_{0}\left(  z,y\right)  -\Phi_{0}\left(  x,y\right)  \right]  dz
\]
as $\delta\rightarrow0$. Again, by Lemma \ref{NSW mod}, Lemma
\ref{Lemma Phi holder} and Lemma \ref{Lemma nonintegrale} we have
\begin{align*}
&  \int_{U}\int_{\mathbb{R}^{m}}\int_{\mathbb{R}^{m}}\left\vert \left(
\left(  1-\omega_{\delta}\right)  \,Y_{j}D_{1}\Gamma\right)  \left(
\Theta_{\left(  z,k\right)  }\left(  x,h\right)  \right)  \right\vert
a_{0}\left(  h\right)  b_{0}\left(  k\right)  dhdk\,\left\vert \Phi_{0}\left(
z,y\right)  -\Phi_{0}\left(  x,y\right)  \right\vert dz\\
&  \leqslant c\int_{U}\int_{\mathbb{R}^{m}}\int_{\mathbb{R}^{m}}\chi_{\left\{
\left\Vert \Theta_{\left(  z,k\right)  }\left(  x,h\right)  \right\Vert
\leqslant\delta\right\}  }\left\Vert \Theta_{\left(  z,k\right)  }\left(
x,h\right)  \right\Vert ^{-Q}\left\vert a_{0}\left(  h\right)  b_{0}\left(
k\right)  \right\vert dhdk\,\left\vert \Phi_{0}\left(  z,y\right)  -\Phi
_{0}\left(  x,y\right)  \right\vert dz\\
&  \leqslant c\delta^{\varepsilon}\int_{U}\phi_{-\varepsilon}\left(
x,z\right)  \left\vert \Phi_{0}\left(  z,y\right)  -\Phi_{0}\left(
x,y\right)  \right\vert dz
\end{align*}
and from this bound we conclude as above that this term converges to $0$ as
$\delta\rightarrow0$, uniformly as soon as $d\left(  x,y\right)  \geqslant c$.

To handle $B_{1,1,2}^{\delta}\left(  x,y\right)  $, let us first fix some
notation. Let $U^{\prime}\Subset U$, $I\subset\mathbb{R}^{m}$ and $r>0$ such
that $I\supset\operatorname{sprt}a_{0}\cup\operatorname{sprt}b_{0}$ and:%
\[
\left(  x,h\right)  \in U^{\prime}\times\operatorname{sprt}a_{0}\text{ and
}\left\Vert \Theta_{\left(  z,k\right)  }\left(  x,h\right)  \right\Vert
<r\Rightarrow\left(  z,k\right)  \in U\times I\equiv\Sigma\text{.}%
\]
Then for any $x\in U^{\prime}$ we have:%
\begin{align*}
&  B_{1,1,2}^{\delta}\left(  x,y\right)  =\Phi_{0}\left(  x,y\right)
\int_{\mathbb{R}^{m}}a_{0}\left(  h\right)  \left(  \int_{\Sigma}Y_{j}\left(
\omega_{\delta}D_{1}\Gamma\right)  \left(  \Theta_{\left(  z,k\right)
}\left(  x,h\right)  \right)  b_{0}\left(  k\right)  dkdz\right)  dh\\
&  =\Phi_{0}\left(  x,y\right)  \int_{\mathbb{R}^{m}}a_{0}\left(  h\right)
\left(  \int_{\left\Vert \Theta_{\left(  z,k\right)  }\left(  x,h\right)
\right\Vert <r}\left(  ...\right)  dkdz+\int_{\Sigma,\left\Vert \Theta
_{\left(  z,k\right)  }\left(  x,h\right)  \right\Vert \geqslant r}\left(
...\right)  dkdz\right)  dh\\
&  =\Phi_{0}\left(  x,y\right)  \left[  I_{1}^{\delta}\left(  x\right)
+I_{2}^{\delta}\left(  x\right)  \right]  .
\end{align*}
Next, making the change of variables $\left(  z,k\right)  \mapsto
u=\Theta_{\left(  z,k\right)  }\left(  x,h\right)  $ and letting
$\widetilde{b}_{0}\left(  \xi,u\right)  =b_{0}\left(  \Theta_{\cdot}\left(
x,h\right)  ^{-1}\left(  u\right)  \right)  ,$
\begin{align*}
I_{1}^{\delta}\left(  x\right)   &  =\int_{\mathbb{R}^{m}}c\left(  \xi\right)
a_{0}\left(  h\right)  \left(  \int_{\left\Vert u\right\Vert <r}\left[
Y_{j}\left(  \omega_{\delta}D_{1}\Gamma\right)  \right]  \left(  u\right)
\left(  1+\chi\left(  \xi,u\right)  \right)  \widetilde{b}_{0}\left(
\xi,u\right)  du\right)  dh\\
&  =\int_{\mathbb{R}^{m}}c\left(  \xi\right)  a_{0}\left(  h\right)  \left(
\int_{\left\Vert u\right\Vert <r}\left[  Y_{j}\left(  \omega_{\delta}%
D_{1}\Gamma\right)  \right]  \left(  u\right)  \chi\left(  \xi,u\right)
\widetilde{b}_{0}\left(  \xi,u\right)  du\right)  dh\\
&  +\int_{\mathbb{R}^{m}}c\left(  \xi\right)  a_{0}\left(  h\right)  \left(
\int_{\left\Vert u\right\Vert <r}\left[  Y_{j}\left(  \omega_{\delta}%
D_{1}\Gamma\right)  \right]  \left(  u\right)  \widetilde{b}_{0}\left(
\xi,u\right)  du\right)  dh\\
&  \equiv\beta_{1}^{\delta}\left(  x\right)  +\beta_{2}^{\delta}\left(
x\right)  .
\end{align*}
By Proposition \ref{Prop bad theta} we know that $\left\vert \chi\left(
\xi,u\right)  \right\vert \leqslant c\left\Vert u\right\Vert ^{\alpha}$, hence
for $\delta\rightarrow0$ (by the same argument used to compute the limit of
$B_{1,2}^{\delta}$)%
\[
\beta_{1}^{\delta}\left(  x\right)  \rightarrow\int_{\mathbb{R}^{m}}c\left(
\xi\right)  a_{0}\left(  h\right)  \left(  \int_{\left\Vert u\right\Vert
<r}\left(  Y_{j}D_{1}\Gamma\right)  \left(  u\right)  \chi\left(
\xi,u\right)  \widetilde{b}_{0}\left(  \xi,u\right)  du\right)  dh\equiv
\beta_{1}\left(  x\right)  .
\]
Note $\beta_{1}\in C^{\alpha}\left(  U^{\prime}\right)  $. Namely, the
functions $c\left(  \cdot\right)  ,\chi\left(  \cdot,u\right)  $ are
H\"{o}lder continuous by \ref{Prop bad theta} (iii); since $\Theta_{\cdot
}\left(  x,h\right)  ^{-1}\left(  u\right)  $ is $C^{1,\alpha}$ also
$\widetilde{b}_{0}\left(  \cdot,u\right)  $ is H\"{o}lder continuous.

To handle $\beta_{2}^{\delta}\left(  x\right)  $ we integrate by parts;
writing $Y_{j}=\sum_{k=1}^{N}a_{jk}\left(  u\right)  \partial_{u_{k}}$ and
denoting by $\nu=\left(  \nu_{1},\nu_{2},...,\nu_{N}\right)  $ the unit outer
normal we get%
\begin{align*}
\beta_{2}^{\delta}\left(  x\right)   &  =-\int_{\mathbb{R}^{m}}c\left(
\xi\right)  a_{0}\left(  h\right)  \left(  \int_{\left\Vert u\right\Vert
<r}\left(  \omega_{\delta}D_{1}\Gamma\right)  \left(  u\right)  \left(
Y_{j}\widetilde{b}_{0}\left(  \xi,\cdot\right)  \right)  \left(  u\right)
du\right)  dh\\
&  +\int_{\mathbb{R}^{m}}c\left(  \xi\right)  a_{0}\left(  h\right)  \left(
\int_{\left\Vert u\right\Vert =r}\left(  \omega_{\delta}D_{1}\Gamma\right)
\left(  u\right)  \widetilde{b}_{0}\left(  \xi,u\right)  \sum_{k}a_{jk}\left(
u\right)  \nu_{k}d\sigma\left(  u\right)  \right)  dh\\
&  \rightarrow-\int_{\mathbb{R}^{m}}c\left(  \xi\right)  a_{0}\left(
h\right)  \left(  \int_{\left\Vert u\right\Vert <r}\left(  D_{1}\Gamma\right)
\left(  u\right)  \left(  Y_{j}\widetilde{b}_{0}\left(  \xi,\cdot\right)
\right)  \left(  u\right)  du\right)  dh\\
&  +\int_{\mathbb{R}^{m}}c\left(  \xi\right)  a_{0}\left(  h\right)  \left(
\int_{\left\Vert u\right\Vert =r}\left(  D_{1}\Gamma\right)  \left(  u\right)
\widetilde{b}_{0}\left(  \xi,u\right)  \sum_{k=1}^{N}a_{jk}\left(  u\right)
\nu_{k}d\sigma\left(  u\right)  \right)  dh\\
&  \equiv\beta_{2}\left(  x\right)
\end{align*}
which is again a $C^{\alpha}\left(  U^{\prime}\right)  $ function by Proposition
\ref{Prop bad theta} (iii). Hence for any $x\in U^{\prime},$%
\[
I_{1}^{\delta}\left(  x\right)  \rightarrow\beta_{1}\left(  x\right)
+\beta_{2}\left(  x\right)  ,
\]
which is a $C^{\alpha}\left(  U^{\prime}\right)  $ function.

As to $I_{2}^{\delta}\left(  x\right)  $, for any $\delta<r$ we have (writing
$\eta=\left(  z,k\right)  $)%
\[
I_{2}^{\delta}\left(  x\right)  =I_{2}\left(  x\right)  \equiv\int
_{\mathbb{R}^{m}}a_{0}\left(  h\right)  \left(  \int_{\Sigma,\left\Vert
\Theta_{\eta}\left(  x,h\right)  \right\Vert >r}\left(  Y_{j}D_{1}%
\Gamma\right)  \left(  \Theta_{\eta}\left(  x,h\right)  \right)  b_{0}\left(
k\right)  d\eta\right)  dh.
\]
Let us show that $I_{2}$ is H\"{o}lder continuous. Actually, we will show
that
\[
\left\vert I_{2}\left(  x_{1}\right)  -I_{2}\left(  x_{2}\right)  \right\vert
\leqslant c\left\vert x_{1}-x_{2}\right\vert
\]
for small $\left\vert x_{1}-x_{2}\right\vert $. Since $I_{2}$ is clearly
bounded, this is enough to conclude H\"{o}lder continuity in $U^{\prime}$.%
\begin{align*}
&  I_{2}\left(  x_{1}\right)  -I_{2}\left(  x_{2}\right) \\
&  =\int_{\mathbb{R}^{m}}a_{0}\left(  h\right)  \left(  \int_{\Sigma
,\left\Vert \Theta_{\eta}\left(  x_{1},h\right)  \right\Vert >r}\left[
Y_{j}D_{1}\Gamma\left(  \Theta_{\eta}\left(  x_{1},h\right)  \right)
-Y_{j}D_{1}\Gamma\left(  \Theta_{\eta}\left(  x_{2},h\right)  \right)
\right]  b_{0}\left(  k\right)  d\eta\right)  dh\\
&  +\int_{\mathbb{R}^{m}}a_{0}\left(  h\right)  \left(  \int_{\Sigma
,\left\Vert \Theta_{\eta}\left(  x_{1},h\right)  \right\Vert >r}Y_{j}%
D_{1}\Gamma\left(  \Theta_{\eta}\left(  x_{2},h\right)  \right)  b_{0}\left(
k\right)  d\eta\right)  dh\\
&  -\int_{\mathbb{R}^{m}}a_{0}\left(  h\right)  \left(  \int_{\Sigma
,\left\Vert \Theta_{\eta}\left(  x_{2},h\right)  \right\Vert >r}Y_{j}%
D_{1}\Gamma\left(  \Theta_{\eta}\left(  x_{2},h\right)  \right)  b_{0}\left(
k\right)  d\eta\right)  dh\\
&  \equiv A+B-C.
\end{align*}
Note that for some small $c_{1}\left(  r\right)  ,c_{2}\left(  r\right)  >0$,
if $\left\vert x_{1}-x_{2}\right\vert <c_{1}$ and $\left\Vert \Theta_{\eta
}\left(  x_{1},h\right)  \right\Vert >r$ then also $\left\Vert \Theta_{\eta
}\left(  x_{2},h\right)  \right\Vert >c_{2}r$. Then%
\[
\left\vert A\right\vert \leqslant c\left(  r\right)  \left\vert x_{1}%
-x_{2}\right\vert \int_{\mathbb{R}^{m}}a_{0}\left(  h\right)  \int
_{\Sigma,\left\Vert \Theta_{\eta}\left(  x_{1},h\right)  \right\Vert >r}%
b_{0}\left(  k\right)  d\eta dh\leqslant c\left(  r\right)  \left\vert
x_{1}-x_{2}\right\vert .
\]
Moreover, letting%
\[
\Lambda=\left\{  \eta:\left\Vert \Theta_{\eta}\left(  x_{2},h\right)
\right\Vert >r,\left\Vert \Theta_{\eta}\left(  x_{1},h\right)  \right\Vert
\leqslant r\ \right\}  \cup\left\{  \eta:\left\Vert \Theta_{\eta}\left(
x_{1},h\right)  \right\Vert >r,\left\Vert \Theta_{\eta}\left(  x_{2},h\right)
\right\Vert \leqslant r\right\}  =\Lambda_{1}\cup\Lambda_{2}%
\]
we have:%
\[
\left\vert B-C\right\vert \leqslant\int_{\mathbb{R}^{m}}a_{0}\left(  h\right)
\left(  \int_{\Sigma\cap\Lambda}\left\vert Y_{j}D_{1}\Gamma\left(
\Theta_{\eta}\left(  x_{2},h\right)  \right)  \right\vert b_{0}\left(
\eta\right)  d\eta\right)  dh
\]
In $\Lambda_{1}$ we have%
\begin{align*}
r  &  <\left\Vert \Theta_{\eta}\left(  x_{2},h\right)  \right\Vert
\leqslant\left\Vert \Theta_{\eta}\left(  x_{1},h\right)  \right\Vert
+\left\Vert \Theta_{\eta}\left(  x_{2},h\right)  -\Theta_{\eta}\left(
x_{1},h\right)  \right\Vert \\
&  \leqslant r+\left\Vert \Theta_{\eta}\left(  x_{2},h\right)  -\Theta_{\eta
}\left(  x_{1},h\right)  \right\Vert \leqslant r+c\left\vert x_{1}%
-x_{2}\right\vert
\end{align*}
hence%
\begin{align*}
&  \int_{\mathbb{R}^{m}}a_{0}\left(  h\right)  \left(  \int_{\Sigma\cap
\Lambda_{1}}\left\vert Y_{j}D_{1}\Gamma\left(  \Theta_{\eta}\left(
x_{2},h\right)  \right)  \right\vert b_{0}\left(  \eta\right)  d\eta\right)
dh\\
&  \leqslant c\left(  r\right)  \int_{\mathbb{R}^{m}}a_{0}\left(  h\right)
\left(  \int_{r<\left\Vert \Theta_{\eta}\left(  x_{2},h\right)  \right\Vert
<r+c\left\vert x_{1}-x_{2}\right\vert }b_{0}\left(  \eta\right)  d\eta\right)
dh\\
&  \leqslant c\int_{\mathbb{R}^{m}}a_{0}\left(  h\right)  \left(
\int_{r<\left\Vert u\right\Vert <r+c\left\vert x_{1}-x_{2}\right\vert
}du\right)  dh\\
&  \leqslant c\left[  \left(  r+c\left\vert x_{1}-x_{2}\right\vert \right)
^{Q}-r^{Q}\right]  \leqslant c\left(  r\right)  \left\vert x_{1}%
-x_{2}\right\vert .
\end{align*}
Since for $\eta\in\Lambda_{2}$ we have $\left\Vert \Theta_{\eta}\left(
x_{2},h\right)  \right\Vert >c_{2}r$ (by the above remark), we can still write%
\begin{align*}
&  \int_{\mathbb{R}^{m}}a_{0}\left(  h\right)  \left(  \int_{\Sigma\cap
\Lambda_{2}}\left\vert Y_{j}D_{1}\Gamma\left(  \Theta_{\eta}\left(
x_{2},h\right)  \right)  \right\vert b_{0}\left(  \eta\right)  d\eta\right)
dh\\
&  \leqslant c\int_{\mathbb{R}^{m}}a_{0}\left(  h\right)  \left(
\int_{r<\left\Vert \Theta_{\eta}\left(  x_{1},h\right)  \right\Vert
<r+c\left\vert x_{1}-x_{2}\right\vert }b_{0}\left(  \eta\right)  d\eta\right)
dh\\
&  \leqslant c\left\vert x_{1}-x_{2}\right\vert .
\end{align*}
We can conclude that%
\[
B_{1,1,2}^{\delta}\left(  x,y\right)  \rightarrow C\left(  x\right)  \Phi
_{0}\left(  x,y\right)  \text{ as }\delta\rightarrow0,
\]
where $C(x)$ is a suitable $C_{X,loc}^{\alpha}\left(  U\right)  $ function.
This completes the proof of (\ref{Big J_0 1}). In particular, for any $x\in
U^{\prime},y\in U,$%
\begin{align*}
\left\vert X_{k}X_{i}J_{0}\left(  x,y\right)  \right\vert  &  \leqslant
c_{1}\int_{U}\phi_{0}\left(  x,z\right)  \left\vert \Phi_{0}\left(
z,y\right)  -\Phi_{0}\left(  x,y\right)  \right\vert dz\\
&  +c_{2}\phi_{\alpha}\left(  x,y\right)  +c_{3}\int_{U}\phi_{\alpha}\left(
x,z\right)  \phi_{\alpha}\left(  z,y\right)  dz.
\end{align*}

Let now%
\begin{align*}
&  \int_{U}\phi_{0}\left(  x,z\right)  \left\vert \Phi_{0}\left(  z,y\right)
-\Phi_{0}\left(  x,y\right)  \right\vert dz\\
&  \leqslant c\int_{\left\{  d\left(  x,y\right)  \geqslant3d\left(
x,z\right)  \right\}  }\phi_{0}\left(  x,z\right)  d^{\alpha-\varepsilon
}\left(  x,z\right)  \phi_{\varepsilon}\left(  x,y\right)  dz\\
&  +c\int_{\left\{  d\left(  x,y\right)  <3d\left(  x,z\right)  \right\}
}\phi_{0}\left(  x,z\right)  \left(  \phi_{\alpha}\left(  z,y\right)
+\phi_{\alpha}\left(  x,y\right)  \right)  dz\\
&  =D+E.
\end{align*}
Then%
\begin{align*}
D  &  \leqslant c\phi_{\varepsilon}\left(  x,y\right)  \int_{\left\{  d\left(
x,y\right)  \geqslant3d\left(  x,z\right)  \right\}  }\phi_{0}\left(
x,z\right)  d^{\alpha-\varepsilon}\left(  x,z\right)  dz\\
&  \leqslant c\phi_{\varepsilon}\left(  x,y\right)  \int_{\left\{  d\left(
x,y\right)  \geqslant3d\left(  x,z\right)  \right\}  }\phi_{\alpha
-\varepsilon}\left(  x,z\right)  dz\\
&  \leqslant c\phi_{\varepsilon}\left(  x,y\right)  d\left(  x,y\right)
^{\alpha-\varepsilon}\leqslant c\frac{d\left(  x,y\right)  ^{\alpha}%
}{\left\vert B\left(  x,d\left(  x,y\right)  \right)  \right\vert }.
\end{align*}
and%
\begin{align*}
E  &  \leqslant\frac{c}{d\left(  x,y\right)  ^{\varepsilon}}\int_{U}\left(
\phi_{\alpha}\left(  z,y\right)  +\phi_{\alpha}\left(  x,y\right)  \right)
\phi_{\varepsilon}\left(  x,z\right)  dz\\
&  \leqslant\frac{c}{d\left(  x,y\right)  ^{\varepsilon}}\left[  \phi
_{\alpha+\varepsilon}\left(  x,y\right)  +\phi_{\alpha}\left(  x,y\right)
\int_{U}\phi_{\varepsilon}\left(  x,z\right)  dz\right] \\
&  \leqslant\frac{c}{d\left(  x,y\right)  ^{\varepsilon}}\left[  \phi
_{\alpha+\varepsilon}\left(  x,y\right)  +\phi_{\alpha}\left(  x,y\right)
R^{\varepsilon}\right] \\
&  \leqslant\frac{c}{d\left(  x,y\right)  ^{\varepsilon}}\left[
\frac{d\left(  x,y\right)  ^{\alpha+\varepsilon}}{\left\vert B\left(
x,d\left(  x,y\right)  \right)  \right\vert }+\frac{d\left(  x,y\right)
^{\alpha}R^{\varepsilon}}{\left\vert B\left(  x,d\left(  x,y\right)  \right)
\right\vert }\right] \\
&  \leqslant cR^{\varepsilon}\frac{d\left(  x,y\right)  ^{\alpha-\varepsilon}%
}{\left\vert B\left(  x,d\left(  x,y\right)  \right)  \right\vert }.
\end{align*}
It follows that%
\begin{align*}
\left\vert X_{k}X_{i}J_{0}\left(  x,y\right)  \right\vert  &  \leqslant
cR^{\varepsilon}\frac{d\left(  x,y\right)  ^{\alpha-\varepsilon}}{\left\vert
B\left(  x,d\left(  x,y\right)  \right)  \right\vert }+c_{2}\phi_{\alpha
}\left(  x,y\right)  +c_{3}\phi_{2\alpha}\left(  x,y\right) \\
&  \leqslant c\frac{d\left(  x,y\right)  ^{\alpha-\varepsilon}}{\left\vert
B\left(  x,d\left(  x,y\right)  \right)  \right\vert },
\end{align*}
which proves (\ref{Big J_0 2}).
\end{proof}

\bigskip

\begin{proof}
[Proof of Theorem \ref{Thm XXgamma}]It is enough to prove (\ref{Stima XXJ})
and the continuity of $X_{i}X_{j}J^{\prime}\left(  x,y\right)  $,
$X_{0}J^{\prime}\left(  x,y\right)  $ in the joint variables, for $x\neq y$,
because these facts together with Proposition \ref{Prop P} imply
(\ref{Stima XXgamma}) and the continuity properties of $X_{i}X_{j}%
\gamma\left(  x,y\right)  $, $X_{0}\gamma\left(  x,y\right)  $. The results
about $X_{i}X_{j}J^{\prime}\left(  x,y\right)  $ immediately follow by Theorem
\ref{Thm Big J_0} choosing $\Phi_{0}=\Phi^{\prime}$. The proof of the analog
result for $X_{0}J^{\prime}\left(  x,y\right)  $ is very similar: we can start
from%
\[
J_{\delta}^{\prime}\left(  x,y\right)  =\int_{U}\int_{\mathbb{R}^{m}}%
\int_{\mathbb{R}^{m}}\left(  \omega_{\delta}\Gamma\right)  \left(
\Theta_{\left(  z,k\right)  }\left(  x,h\right)  \right)  a_{0}\left(
h\right)  b_{0}\left(  k\right)  dhdk\,\Phi^{\prime}\left(  z,y\right)  dz
\]
and compute%
\[
X_{j}J_{\delta}^{\prime}\left(  x,y\right)  =\int_{U}\int_{\mathbb{R}^{m}}%
\int_{\mathbb{R}^{m}}\widetilde{X}_{j}\left[  \left(  \omega_{\delta}%
\Gamma\right)  \left(  \Theta_{\left(  z,k\right)  }\left(  x,h\right)
\right)  a_{0}\left(  h\right)  \right]  b_{0}\left(  k\right)  dhdk\,\Phi
^{\prime}\left(  z,y\right)  dz
\]
From this point the computation of $X_{0}J^{\prime}\left(  x,y\right)  $
proceeds as above.
\end{proof}

We can now refine the previous analysis of the second derivatives of our local
fundamental solution and prove a sharp bound of H\"{o}lder type on $X_{i}%
X_{j}\gamma$. This is both interesting in its own, and will be a basic
ingredient to deduce, via the theory of singular integrals, local H\"{o}lder
estimates for the second derivatives of the local solution to the equation
$Lw=f$ that we will build in the next section.

\begin{theorem}
\label{Thm schauder fundam}For every $\varepsilon\in\left(  0,\alpha\right)  $
and $U^{\prime}\Subset U$ there exists $c>0$ such that for every $x_{1}%
,x_{2}\in U^{\prime}$, $y\in U$ such that $d\left(  x_{1},y\right)
\geqslant2d\left(  x_{1},x_{2}\right)  $, $i,j=1,2,...,n$,%
\begin{align}
\left\vert X_{i}X_{j}P\left(  x_{1},y\right)  -X_{i}X_{j}P\left(
x_{2},y\right)  \right\vert  &  \leqslant c\frac{d\left(  x_{1},x_{2}\right)
}{d\left(  x_{1},y\right)  }\frac{1}{\left\vert B\left(  x_{1},d\left(
x_{1},y\right)  \right)  \right\vert };\label{standard P}\\
\left\vert X_{i}X_{j}J^{\prime}\left(  x_{1},y\right)  -X_{i}X_{j}J^{\prime
}\left(  x_{2},y\right)  \right\vert  &  \leqslant c\left(  \frac{d\left(
x_{1},x_{2}\right)  }{d\left(  x_{1},y\right)  }\right)  ^{\alpha
-2\varepsilon}\frac{d\left(  x_{1},y\right)  ^{\alpha-\varepsilon}}{\left\vert
B\left(  x_{1},d\left(  x_{1},y\right)  \right)  \right\vert }%
;\label{standard J}\\
\left\vert X_{i}X_{j}\gamma\left(  x_{1},y\right)  -X_{i}X_{j}\gamma\left(
x_{2},y\right)  \right\vert  &  \leqslant c\left(  \frac{d\left(  x_{1}%
,x_{2}\right)  }{d\left(  x_{1},y\right)  }\right)  ^{\alpha-\varepsilon}%
\frac{1}{\left\vert B\left(  x_{1},d\left(  x_{1},y\right)  \right)
\right\vert };\label{standard gamma}\\
\left\vert X_{0}\gamma\left(  x_{1},y\right)  -X_{0}\gamma\left(
x_{2},y\right)  \right\vert  &  \leqslant c\left(  \frac{d\left(  x_{1}%
,x_{2}\right)  }{d\left(  x_{1},y\right)  }\right)  ^{\alpha-\varepsilon}%
\frac{1}{\left\vert B\left(  x_{1},d\left(  x_{1},y\right)  \right)
\right\vert }. \label{X_0 gamma}%
\end{align}
In particular, for every $\varepsilon\in\left(  0,\alpha\right)  $ and $y\in
U,$
\[
\gamma\left(  \cdot,y\right)  \in C_{X,loc}^{2,\alpha-\varepsilon}\left(
U\setminus\left\{  y\right\}  \right)  .
\]

\end{theorem}

\begin{proof}
The proof will be achieved in several steps. We know that%
\begin{equation}
X_{i}X_{j}\gamma\left(  x,y\right)  =\frac{1}{c_{0}\left(  y\right)  }\left[
X_{i}X_{j}P\left(  x,y\right)  +X_{i}X_{j}J^{\prime}\left(  x,y\right)
\right]  , \label{XXgamma holder}%
\end{equation}
hence (\ref{standard gamma}) will follow from (\ref{standard P}) and
(\ref{standard J}). Also (\ref{X_0 gamma}) will follow from
(\ref{standard gamma}) since $X_{0}\gamma\left(  x,y\right)  =-\sum_{j=1}%
^{n}X_{j}^{2}\gamma\left(  x,y\right)  $ for $x\neq y$.

Let us first prove (\ref{standard P}). To do this, let us apply
\textquotedblleft Lagrange theorem\textquotedblright\ (Proposition
\ref{Prop Lagrange}) to the function
\[
f\left(  x\right)  =X_{i}X_{j}P\left(  x,y\right)  \text{ for }x\in B\left(
x_{1},\frac{1}{2}d\left(  x_{1},y\right)  \right)  \text{:}%
\]%
\begin{align*}
\left\vert X_{i}X_{j}P\left(  x_{1},y\right)  -X_{i}X_{j}P\left(
x_{2},y\right)  \right\vert  &  \leqslant cd\left(  x_{1},x_{2}\right)
\left(  \overset{n}{\underset{k=1}{\sum}}\underset{B\left(  x_{1},\frac{1}%
{2}d\left(  x_{1},y\right)  \right)  }{\sup}\left\vert X_{k}X_{i}X_{j}P\left(
\cdot,y\right)  \right\vert +\right. \\
&  \left.  +d\left(  x_{1},x_{2}\right)  \underset{B\left(  x_{1},\frac{1}%
{2}d\left(  x_{1},y\right)  \right)  }{\sup}\left\vert X_{0}X_{i}X_{j}P\left(
\cdot,y\right)  \right\vert \right)  .
\end{align*}
Note that since, under our assumptions, the coefficients of the $X_{i}$'s
belong to $C^{r,\alpha}$, with $r\geqslant2$, the compositions $X_{k}%
X_{i}X_{j}$, $X_{0}X_{i}X_{j}$ are actually well defined. Reasoning like in
the proof of Proposition \ref{Prop P} we get, for $x\in B\left(  x_{1}%
,\frac{1}{2}d\left(  x_{1},y\right)  \right)  $%
\begin{align*}
\left\vert X_{k}X_{i}X_{j}P\left(  x,y\right)  \right\vert  &  \leqslant
c\phi_{-1}\left(  x,y\right)  \leqslant c\phi_{-1}\left(  x_{1},y\right) \\
&  \leqslant\frac{c}{d\left(  x_{1},y\right)  \left\vert B\left(
x_{1},d\left(  x_{1},y\right)  \right)  \right\vert }%
\end{align*}
by Lemma \ref{Lemma nonintegrale}. Analogously,%
\[
\left\vert X_{0}X_{i}X_{j}P\left(  x,y\right)  \right\vert \leqslant
c\phi_{-2}\left(  x,y\right)  \leqslant\frac{c}{d\left(  x_{1},y\right)
^{2}\left\vert B\left(  x_{1},d\left(  x_{1},y\right)  \right)  \right\vert }%
\]
so that, for $2d\left(  x_{1},x_{2}\right)  \leqslant d\left(  x_{1},y\right)
,$
\[
\left\vert X_{i}X_{j}P\left(  x_{1},y\right)  -X_{i}X_{j}P\left(
x_{2},y\right)  \right\vert \leqslant c\frac{d\left(  x_{1},x_{2}\right)
}{d\left(  x_{1},y\right)  }\frac{1}{\left\vert B\left(  x_{1},d\left(
x_{1},y\right)  \right)  \right\vert }%
\]
and (\ref{standard P}) is proved. Applying Theorem \ref{Thm Big J_0} with
$\Phi_{0}=\Phi^{\prime}$, we know that for any $x\in U^{\prime},y\in U$%
\begin{align*}
&  X_{j}X_{i}J^{\prime}\left(  x,y\right) \\
&  =\int_{U}\int_{\mathbb{R}^{m}}\int_{\mathbb{R}^{m}}\left(  Y_{j}D_{1}%
\Gamma\right)  \left(  \Theta_{\left(  z,k\right)  }\left(  x,h\right)
\right)  a_{0}\left(  h\right)  b_{0}\left(  k\right)  dhdk\,\left[
\Phi^{\prime}\left(  z,y\right)  -\Phi^{\prime}\left(  x,y\right)  \right]
dz\\
&  +C\left(  x\right)  \Phi^{\prime}\left(  x,y\right)  +\int_{U}R_{2}\left(
x,z\right)  \Phi^{\prime}\left(  z,y\right)  dz\\
&  \equiv A\left(  x,y\right)  +B\left(  x,y\right)  +C\left(  x,y\right)  .
\end{align*}
Let us start from the last two terms, which are easier. By Proposition
\ref{Phi holder} and the local H\"{o}lder continuity of $C\left(  x\right)  $
we have%
\begin{align*}
\left\vert B\left(  x_{2},y\right)  -B\left(  x_{1},y\right)  \right\vert  &
\leqslant\left\vert C\left(  x_{2}\right)  -C\left(  x_{1}\right)  \right\vert
\Phi^{\prime}\left(  x_{2},y\right)  +\left\vert C\left(  x_{1}\right)
\right\vert \left\vert \Phi^{\prime}\left(  x_{2},y\right)  -\Phi^{\prime
}\left(  x_{1},y\right)  \right\vert \\
&  \leqslant cd\left(  x_{1},x_{2}\right)  ^{\alpha}\phi_{\alpha}\left(
x_{2},y\right)  +cd\left(  x_{1},x_{2}\right)  ^{\alpha-\varepsilon}%
\phi_{\varepsilon}\left(  x_{1},y\right) \\
&  \leqslant cd\left(  x_{1},x_{2}\right)  ^{\alpha-\varepsilon}%
\phi_{\varepsilon}\left(  x_{1},y\right)  \leqslant c\left(  \frac{d\left(
x_{1},x_{2}\right)  }{d\left(  x_{1},y\right)  }\right)  ^{\alpha-\varepsilon
}\frac{d\left(  x_{1},y\right)  ^{\alpha}}{\left\vert B\left(  x_{1},d\left(
x_{1},y\right)  \right)  \right\vert }.
\end{align*}
As to $C,$%
\begin{align*}
C\left(  x_{2},y\right)  -C\left(  x_{1},y\right)   &  =\int_{U}\left[
R_{2}\left(  x_{2},z\right)  -R_{2}\left(  x_{1},z\right)  \right]
\Phi^{\prime}\left(  z,y\right)  dz\\
&  =\int_{U,d\left(  z,x_{1}\right)  >2d\left(  x_{1},x_{2}\right)  }\left(
...\right)  dz+\int_{U,d\left(  z,x_{1}\right)  \leqslant2d\left(  x_{1}%
,x_{2}\right)  }\left(  ...\right)  dz\\
&  \equiv C_{1}+C_{2}.
\end{align*}
To bound $C_{1}$ we apply Lagrange theorem:%
\begin{align*}
\left\vert R_{2}\left(  x_{2},z\right)  -R_{2}\left(  x_{1},z\right)
\right\vert  &  \leqslant cd\left(  x_{1},x_{2}\right)  \left(  \overset
{n}{\underset{k=1}{\sum}}\underset{B\left(  x_{1},\frac{1}{2}d\left(
x_{1},z\right)  \right)  }{\sup}\left\vert X_{k}R_{2}\left(  \cdot,z\right)
\right\vert \right. \\
&  \left.  +d\left(  x_{1},x_{2}\right)  \underset{B\left(  x_{1},\frac{1}%
{2}d\left(  x_{1},z\right)  \right)  }{\sup}\left\vert X_{0}R_{2}\left(
\cdot,z\right)  \right\vert \right) \\
&  \leqslant cd\left(  x_{1},x_{2}\right)  \phi_{-1+\alpha}\left(
x_{1},z\right)
\end{align*}
where the bounds on $\left\vert X_{k}R_{2}\left(  \cdot,z\right)  \right\vert
,\left\vert X_{0}R_{2}\left(  \cdot,z\right)  \right\vert $ exploit
Proposition \ref{Prop bad theta} (ii). Hence%
\begin{align*}
\left\vert C_{1}\right\vert  &  \leqslant cd\left(  x_{1},x_{2}\right)
\int_{U,d\left(  z,x_{1}\right)  >2d\left(  x_{1},x_{2}\right)  }%
\phi_{-1+\alpha}\left(  x_{1},z\right)  \phi_{\alpha}\left(  z,y\right)  dz\\
&  \leqslant cd\left(  x_{1},x_{2}\right)  ^{\alpha-\varepsilon}%
\int_{U,d\left(  z,x_{1}\right)  >2d\left(  x_{1},x_{2}\right)  }d\left(
x_{1},z\right)  ^{1-\alpha+\varepsilon}\phi_{-1+\alpha}\left(  x_{1},z\right)
\phi_{\alpha}\left(  z,y\right)  dz\\
&  \leqslant cd\left(  x_{1},x_{2}\right)  ^{\alpha-\varepsilon}\int_{U}%
\phi_{\varepsilon}\left(  x_{1},z\right)  \phi_{\alpha}\left(  z,y\right)
dz\leqslant cd\left(  x_{1},x_{2}\right)  ^{\alpha-\varepsilon}\phi
_{\alpha+\varepsilon}\left(  x_{1},y\right)  ,
\end{align*}
while%
\[
\left\vert C_{2}\right\vert \leqslant\int_{U,d\left(  z,x_{1}\right)
\leqslant2d\left(  x_{1},x_{2}\right)  }\left[  \phi_{\alpha}\left(
x_{2},z\right)  +\phi_{\alpha}\left(  x_{1},z\right)  \right]  \phi_{\alpha
}\left(  z,y\right)  dz.
\]
Since $d\left(  x_{1},z\right)  \leqslant2d\left(  x_{1},x_{2}\right)  $ and
$d\left(  x_{1},y\right)  \geqslant3d\left(  x_{1},x_{2}\right)  $ implies
$d\left(  x_{1},y\right)  \leqslant3d\left(  z,y\right)  $%
\begin{align*}
\left\vert C_{2}\right\vert  &  \leqslant c\int_{U,d\left(  z,x_{1}\right)
\leqslant2d\left(  x_{1},x_{2}\right)  }\left[  \phi_{\alpha}\left(
x_{2},z\right)  +\phi_{\alpha}\left(  x_{1},z\right)  \right]  \phi_{\alpha
}\left(  x_{1},y\right)  dz\\
&  \leqslant c\phi_{\alpha}\left(  x_{1},y\right)  \left(  \int_{U,d\left(
z,x_{2}\right)  \leqslant3d\left(  x_{1},x_{2}\right)  }\phi_{\alpha}\left(
x_{2},z\right)  dz+\int_{U,d\left(  z,x_{1}\right)  \leqslant2d\left(
x_{1},x_{2}\right)  }\phi_{\alpha}\left(  x_{1},z\right)  dz\right) \\
&  \leqslant c\phi_{\alpha}\left(  x_{1},y\right)  d\left(  x_{1}%
,x_{2}\right)  ^{\alpha}%
\end{align*}
by Corollary \ref{coroll Psi beta}. Hence%
\begin{align*}
&  \left\vert C\left(  x_{2},y\right)  -C\left(  x_{1},y\right)  \right\vert
\leqslant cd\left(  x_{1},x_{2}\right)  ^{\alpha-\varepsilon}\phi
_{\alpha+\varepsilon}\left(  x_{1},y\right)  +c\phi_{\alpha}\left(
x_{1},y\right)  d\left(  x_{1},x_{2}\right)  ^{\alpha}\\
&  \leqslant c\left(  \frac{d\left(  x_{1},x_{2}\right)  }{d\left(
x_{1},y\right)  }\right)  ^{\alpha-\varepsilon}\frac{d\left(  x_{1},y\right)
^{2\alpha}}{\left\vert B\left(  x_{1},d\left(  x_{1},y\right)  \right)
\right\vert }+c\left(  \frac{d\left(  x_{1},x_{2}\right)  }{d\left(
x_{1},y\right)  }\right)  ^{\alpha}\frac{d\left(  x_{1},y\right)  ^{2\alpha}%
}{\left\vert B\left(  x_{1},d\left(  x_{1},y\right)  \right)  \right\vert }\\
&  \leqslant c\left(  \frac{d\left(  x_{1},x_{2}\right)  }{d\left(
x_{1},y\right)  }\right)  ^{\alpha-\varepsilon}\frac{d\left(  x_{1},y\right)
^{2\alpha}}{\left\vert B\left(  x_{1},d\left(  x_{1},y\right)  \right)
\right\vert }.
\end{align*}
As to $A$, let%
\[
k\left(  x,z\right)  =\int_{\mathbb{R}^{m}}\int_{\mathbb{R}^{m}}\left(
Y_{j}D_{1}\Gamma\right)  \left(  \Theta_{\left(  z,k\right)  }\left(
x,h\right)  \right)  a_{0}\left(  h\right)  b_{0}\left(  k\right)  dhdk,
\]
then%
\begin{align*}
A\left(  x_{2},y\right)  -A\left(  x_{1},y\right)   &  =\int_{U}\left\{
k\left(  x_{2},z\right)  \left[  \Phi^{\prime}\left(  z,y\right)
-\Phi^{\prime}\left(  x_{2},y\right)  \right]  -k\left(  x_{1},z\right)
\left[  \Phi^{\prime}\left(  z,y\right)  -\Phi^{\prime}\left(  x_{1},y\right)
\right]  \right\}  dz\\
&  =\int_{U,d\left(  x_{1},z\right)  \geqslant2d\left(  x_{1},x_{2}\right)
}\left\{  \cdots\right\}  dz+\int_{U,d\left(  x_{1},z\right)  <2d\left(
x_{1},x_{2}\right)  }\left\{  \cdots\right\}  dz\\
&  \equiv A_{1}\left(  x_{1},x_{2},y\right)  +A_{2}\left(  x_{1}%
,x_{2},y\right)  .
\end{align*}%
\begin{align*}
A_{1}\left(  x_{1},x_{2},y\right)   &  =\int_{U,d\left(  x_{1},z\right)
\geqslant2d\left(  x_{1},x_{2}\right)  }\left[  k\left(  x_{2},z\right)
-k\left(  x_{1},z\right)  \right]  \left[  \Phi^{\prime}\left(  z,y\right)
-\Phi^{\prime}\left(  x_{2},y\right)  \right]  dz\\
&  +\left[  \Phi^{\prime}\left(  x_{1},y\right)  -\Phi^{\prime}\left(
x_{2},y\right)  \right]  \int_{U,d\left(  x_{1},z\right)  \geqslant2d\left(
x_{1},x_{2}\right)  }k\left(  x_{1},z\right)  dz\\
&  \equiv A_{1,1}\left(  x_{1},x_{2},y\right)  +A_{1,2}\left(  x_{1}%
,x_{2},y\right)  .
\end{align*}
Since, for $d\left(  x_{1},z\right)  \geqslant2d\left(  x_{1},x_{2}\right)  $
we have%
\[
\left\vert k\left(  x_{2},z\right)  -k\left(  x_{1},z\right)  \right\vert
\leqslant d\left(  x_{1},x_{2}\right)  \phi_{-1}\left(  x_{1},z\right)
\]
we obtain%
\[
\left\vert A_{1,1}\left(  x_{1},x_{2},y\right)  \right\vert \leqslant
cd\left(  x_{1},x_{2}\right)  \int_{U,d\left(  x_{1},z\right)  \geqslant
2d\left(  x_{1},x_{2}\right)  }\phi_{-1}\left(  x_{1},z\right)  \left\vert
\Phi^{\prime}\left(  z,y\right)  -\Phi^{\prime}\left(  x_{2},y\right)
\right\vert dz.
\]
We now split the domain of integration $\left\{  z\in U:d\left(
x_{1},z\right)  \geqslant2d\left(  x_{1},x_{2}\right)  \right\}  $ into two
pieces%
\begin{align*}
U_{1}  &  =\left\{  z:U:d\left(  x_{1},z\right)  \geqslant2d\left(
x_{1},x_{2}\right)  ,~d\left(  z,y\right)  \geqslant3d\left(  z,x_{2}\right)
\right\}  ,\\
U_{2}  &  =\left\{  z:U:d\left(  x_{1},z\right)  \geqslant2d\left(
x_{1},x_{2}\right)  ,~d\left(  z,y\right)  <3d\left(  z,x_{2}\right)
\right\}  ,
\end{align*}
so that%
\begin{align*}
\left\vert A_{1,1}\left(  x_{1},x_{2},y\right)  \right\vert  &  \leqslant
cd\left(  x_{1},x_{2}\right)  \int_{U_{1}}\left(  \cdots\right)  dz+d\left(
x_{1},x_{2}\right)  \int_{U_{2}}\left(  \cdots\right)  dz\\
&  \equiv A_{1,1,1}\left(  x_{1},x_{2},y\right)  +A_{1,1,2}\left(  x_{1}%
,x_{2},y\right)  .
\end{align*}
Note that $d\left(  x_{1},z\right)  $ and $d\left(  x_{2},z\right)  $ are
equivalent on $U_{1}$ and $U_{2}$. Also on $U_{1}$ (since $d\left(
z,y\right)  \geqslant3d\left(  z,x_{2}\right)  $) we have $\left\vert
\Phi^{\prime}\left(  z,y\right)  -\Phi^{\prime}\left(  x_{2},y\right)
\right\vert \leqslant d\left(  z,x_{2}\right)  ^{\alpha-\varepsilon}%
\phi_{\varepsilon}\left(  z,y\right)  $ and therefore%
\begin{align*}
\left\vert A_{1,1,1}\left(  x_{1},x_{2},y\right)  \right\vert  &  \leqslant
cd\left(  x_{1},x_{2}\right)  \int_{U,d\left(  x_{1},z\right)  \geqslant
2d\left(  x_{1},x_{2}\right)  }\phi_{-1}\left(  x_{1},z\right)  d\left(
z,x_{2}\right)  ^{\alpha-\varepsilon}\phi_{\varepsilon}\left(  z,y\right)
dz\\
&  \leqslant cd\left(  x_{1},x_{2}\right)  \int_{U,d\left(  x_{1},z\right)
\geqslant2d\left(  x_{1},x_{2}\right)  }\phi_{-1}\left(  x_{1},z\right)
d\left(  z,x_{1}\right)  ^{\alpha-\varepsilon}\phi_{\varepsilon}\left(
z,y\right)  dz\\
&  \leqslant cd\left(  x_{1},x_{2}\right)  ^{\alpha-2\varepsilon}%
\int_{U,d\left(  x_{1},z\right)  \geqslant2d\left(  x_{1},x_{2}\right)  }%
\phi_{-1}\left(  x_{1},z\right)  d\left(  z,x_{1}\right)  ^{1+\varepsilon}%
\phi_{\varepsilon}\left(  z,y\right)  dz\\
&  \leqslant cd\left(  x_{1},x_{2}\right)  ^{\alpha-2\varepsilon}\int_{U}%
\phi_{\varepsilon}\left(  x_{1},z\right)  \phi_{\varepsilon}\left(
z,y\right)  dz\\
&  \leqslant cd\left(  x_{1},x_{2}\right)  ^{\alpha-2\varepsilon}%
\phi_{2\varepsilon}\left(  x_{1},y\right)  \leqslant c\left(  \frac{d\left(
x_{1},x_{2}\right)  }{d\left(  x_{1},y\right)  }\right)  ^{\alpha
-2\varepsilon}\frac{d\left(  x_{1},y\right)  ^{\alpha}}{\left\vert B\left(
x_{1},d\left(  x_{1},y\right)  \right)  \right\vert }.
\end{align*}
We now consider the second term. We have%
\begin{align*}
\left\vert A_{1,1,2}\left(  x_{1},x_{2},y\right)  \right\vert  &  \leqslant
cd\left(  x_{1},x_{2}\right)  \int_{U_{2}}\phi_{-1}\left(  x_{1},z\right)
\left(  \phi_{\alpha}\left(  z,y\right)  +\phi_{\alpha}\left(  x_{2},y\right)
\right)  dz\\
&  \equiv A_{1,1,2}^{\prime}+A_{1,1,2}^{\prime\prime}.
\end{align*}
Since $d\left(  y,z\right)  \leqslant\frac{1}{2}d\left(  x_{1},y\right)  $
implies $d\left(  x_{1},y\right)  \leqslant2d\left(  x_{1},z\right)  ,$%
\begin{align*}
A_{1,1,2}^{\prime}  &  \leqslant cd\left(  x_{1},x_{2}\right)  \frac{d\left(
x_{1},y\right)  ^{1+\varepsilon}}{d\left(  x_{1},y\right)  ^{1+\varepsilon}%
}\int_{U_{2}\cap\left\{  d\left(  y,z\right)  \leqslant\frac{1}{2}d\left(
x_{1},y\right)  \right\}  }\phi_{-1}\left(  x_{1},z\right)  \phi_{\alpha
}\left(  z,y\right)  dz\\
&  +cd\left(  x_{1},x_{2}\right)  ^{\alpha-\varepsilon}\int_{U_{2}\cap\left\{
d\left(  y,z\right)  >\frac{1}{2}d\left(  x_{1},y\right)  \right\}  }d\left(
x_{1},x_{2}\right)  ^{1-\alpha+\varepsilon}\phi_{-1}\left(  x_{1},z\right)
\frac{d\left(  z,y\right)  ^{\alpha}}{\left\vert B\left(  z,d\left(
z,y\right)  \right)  \right\vert }dz\\
&  \leqslant c\frac{d\left(  x_{1},x_{2}\right)  }{d\left(  x_{1},y\right)
^{1+\varepsilon}}\int_{U}\phi_{\varepsilon}\left(  x_{1},z\right)
\phi_{\alpha}\left(  z,y\right)  dz\\
&  +\frac{cd\left(  x_{1},x_{2}\right)  ^{\alpha-\varepsilon}}{\left\vert
B\left(  y,d\left(  x_{1},y\right)  \right)  \right\vert }\int_{U_{2}%
\cap\left\{  d\left(  y,z\right)  >\frac{1}{2}d\left(  x_{1},y\right)
\right\}  }d\left(  x_{1},z\right)  ^{1-\alpha+\varepsilon}\phi_{-1}\left(
x_{1},z\right)  d\left(  z,x_{1}\right)  ^{\alpha}dz\\
&  \leqslant c\frac{d\left(  x_{1},x_{2}\right)  }{d\left(  x_{1},y\right)
^{1+\varepsilon}}\phi_{\alpha+\varepsilon}\left(  x_{1},y\right)
+\frac{cd\left(  x_{1},x_{2}\right)  ^{\alpha-\varepsilon}}{\left\vert
B\left(  y,d\left(  x_{1},y\right)  \right)  \right\vert }\int_{U}%
\phi_{\varepsilon}\left(  x_{1},z\right)  dz\\
&  \leqslant cd\left(  x_{1},x_{2}\right)  \frac{d\left(  x_{1},y\right)
^{\alpha-1}}{\left\vert B\left(  y,d\left(  x_{1},y\right)  \right)
\right\vert }+c\frac{d\left(  x_{1},x_{2}\right)  ^{\alpha-\varepsilon}%
}{\left\vert B\left(  x_{1},d\left(  x_{1},y\right)  \right)  \right\vert
}R^{\varepsilon}%
\end{align*}%
\begin{align*}
&  \leqslant c\left(  \frac{d\left(  x_{1},x_{2}\right)  }{d\left(
x_{1},y\right)  }\right)  ^{\alpha-\varepsilon}\left(  \frac{d\left(
x_{1},y\right)  ^{\alpha-1}d\left(  x_{1},x_{2}\right)  }{d\left(  x_{1}%
,x_{2}\right)  ^{\alpha-\varepsilon}}\frac{d\left(  x_{1},y\right)
^{\alpha-\varepsilon}}{\left\vert B\left(  y,d\left(  x_{1},y\right)  \right)
\right\vert }+\frac{d\left(  x_{1},y\right)  ^{\alpha-\varepsilon}}{\left\vert
B\left(  x_{1},d\left(  x_{1},y\right)  \right)  \right\vert }\right) \\
&  \leqslant c\left(  \frac{d\left(  x_{1},x_{2}\right)  }{d\left(
x_{1},y\right)  }\right)  ^{\alpha-\varepsilon}\left(  R^{\varepsilon}%
\frac{d\left(  x_{1},y\right)  ^{\alpha-\varepsilon}}{\left\vert B\left(
y,d\left(  x_{1},y\right)  \right)  \right\vert }+\frac{d\left(
x_{1},y\right)  ^{\alpha-\varepsilon}}{\left\vert B\left(  x_{1},d\left(
x_{1},y\right)  \right)  \right\vert }\right)  .
\end{align*}
Since in $U_{2}$%
\[
d\left(  x_{2},y\right)  \leqslant d\left(  x_{2},z\right)  +d\left(
z,y\right)  \leqslant cd\left(  z,x_{2}\right)  \leqslant cd\left(
z,x_{1}\right)
\]
we have%
\begin{align*}
A_{1,1,2}^{\prime\prime}  &  \leqslant d\left(  x_{1},x_{2}\right)
^{\alpha-\varepsilon}\frac{\phi_{\alpha}\left(  x_{2},y\right)  }{d\left(
x_{2},y\right)  ^{\alpha}}\int_{U_{2}}\phi_{\varepsilon}\left(  x_{1}%
,z\right)  dz\leqslant c\frac{d\left(  x_{1},x_{2}\right)  ^{\alpha
-\varepsilon}}{\left\vert B\left(  x_{2},d\left(  x_{2},y\right)  \right)
\right\vert }R^{\varepsilon}\\
&  \leqslant c\left(  \frac{d\left(  x_{1},x_{2}\right)  }{d\left(
x_{2},y\right)  }\right)  ^{\alpha-\varepsilon}\frac{d\left(  x_{2},y\right)
^{\alpha-\varepsilon}}{\left\vert B\left(  x_{2},d\left(  x_{2},y\right)
\right)  \right\vert }.
\end{align*}
Hence%
\[
\left\vert A_{1,1,2}\left(  x_{1},x_{2},y\right)  \right\vert \leqslant
c\left(  \frac{d\left(  x_{1},x_{2}\right)  }{d\left(  x_{1},y\right)
}\right)  ^{\alpha-\varepsilon}\frac{d\left(  x_{1},y\right)  ^{\alpha
-\varepsilon}}{\left\vert B\left(  x_{1},d\left(  x_{1},y\right)  \right)
\right\vert }%
\]
and%
\[
\left\vert A_{1,1}\left(  x_{1},x_{2},y\right)  \right\vert \leqslant c\left(
\frac{d\left(  x_{1},x_{2}\right)  }{d\left(  x_{1},y\right)  }\right)
^{\alpha-2\varepsilon}\frac{d\left(  x_{1},y\right)  ^{\alpha-\varepsilon}%
}{\left\vert B\left(  x_{1},d\left(  x_{1},y\right)  \right)  \right\vert }.
\]

We now consider $A_{1,2}$. Observe that for $d\left(  x_{1},y\right)
>3d\left(  x_{1},x_{2}\right)  $ we have%
\begin{align*}
\left\vert A_{1,2}\left(  x_{1},x_{2},y\right)  \right\vert  &  \leqslant
\left\vert \Phi^{\prime}\left(  x_{1},y\right)  -\Phi^{\prime}\left(
x_{2},y\right)  \right\vert \int_{U,d\left(  x_{1},z\right)  \geqslant
2d\left(  x_{1},x_{2}\right)  }\phi_{0}\left(  x_{1},z\right)  dz\\
&  \leqslant cd\left(  x_{1},x_{2}\right)  ^{\alpha-\varepsilon}%
\phi_{\varepsilon}\left(  x_{1},y\right)  \int_{U,d\left(  x_{1},z\right)
\geqslant2d\left(  x_{1},x_{2}\right)  }\phi_{0}\left(  x_{1},z\right)  dz\\
&  \leqslant cd\left(  x_{1},x_{2}\right)  ^{\alpha-2\varepsilon}%
\phi_{\varepsilon}\left(  x_{1},y\right)  \int_{U,d\left(  x_{1},z\right)
\geqslant2d\left(  x_{1},x_{2}\right)  }\phi_{\varepsilon}\left(
x_{1},z\right)  dz\\
&  \leqslant cd\left(  x_{1},x_{2}\right)  ^{\alpha-2\varepsilon}%
\phi_{\varepsilon}\left(  x_{1},y\right)  R^{\varepsilon}\leqslant c\left(
\frac{d\left(  x_{1},x_{2}\right)  }{d\left(  x_{1},y\right)  }\right)
^{\alpha-2\varepsilon}\frac{d\left(  x_{1},y\right)  ^{\alpha-\varepsilon}%
}{\left\vert B\left(  x_{1},d\left(  x_{1},y\right)  \right)  \right\vert }.
\end{align*}
Finally we have to bound $A_{2}\left(  x_{1},x_{2},y\right)  $. We have%
\begin{align*}
\left\vert A_{2}\left(  x_{1},x_{2},y\right)  \right\vert  &  \leqslant
\int_{U,d\left(  x_{2},z\right)  <3d\left(  x_{1},x_{2}\right)  }\phi
_{0}\left(  x_{2},z\right)  \left\vert \Phi^{\prime}\left(  z,y\right)
-\Phi^{\prime}\left(  x_{2},y\right)  \right\vert dz\\
&  +\int_{U,d\left(  x_{1},z\right)  <2d\left(  x_{1},x_{2}\right)  }\phi
_{0}\left(  x_{1},z\right)  \left\vert \Phi^{\prime}\left(  z,y\right)
-\Phi^{\prime}\left(  x_{1},y\right)  \right\vert dz.
\end{align*}
Since the two terms are similar it is enough to bound the second. We have%
\begin{align*}
&  \int_{U,d\left(  x_{1},z\right)  <2d\left(  x_{1},x_{2}\right)  }\phi
_{0}\left(  x_{1},z\right)  \left\vert \Phi^{\prime}\left(  z,y\right)
-\Phi^{\prime}\left(  x_{1},y\right)  \right\vert dz\\
&  =\int_{U,d\left(  x_{1},z\right)  <2d\left(  x_{1},x_{2}\right)  ,d\left(
y,z\right)  \leqslant\frac{1}{2}d\left(  x_{1},y\right)  }\left\{
\cdots\right\}  dz+\int_{U,d\left(  x_{1},z\right)  <2d\left(  x_{1}%
,x_{2}\right)  ,d\left(  y,z\right)  >\frac{1}{2}d\left(  x_{1},y\right)
}\left\{  \cdots\right\}  dz\\
&  \equiv A_{2,1}\left(  x_{1},x_{2},y\right)  +A_{2,2}\left(  x_{1}%
,x_{2},y\right)  .
\end{align*}
As to $A_{2,1}\left(  x_{1},x_{2},y\right)  ,$ we note that, under the
assumption $d\left(  x_{1},y\right)  \geqslant3d\left(  x_{1},x_{2}\right)  $,
in the domain of integration the following equivalences hold:%
\[
d\left(  x_{1},y\right)  \simeq d\left(  z,y\right)  \simeq d\left(
x_{1},z\right)  .
\]
Therefore%
\[
\left\vert \Phi^{\prime}\left(  z,y\right)  -\Phi^{\prime}\left(
x_{1},y\right)  \right\vert \leqslant\phi_{\alpha}\left(  z,y\right)
+\phi_{\alpha}\left(  x_{1},y\right)  \leqslant c\phi_{\alpha}\left(
z,y\right)
\]
and%
\begin{align*}
&  A_{2,1}\left(  x_{1},x_{2},y\right)  \leqslant c\int_{\left\{  d\left(
x_{1},z\right)  <2d\left(  x_{1},x_{2}\right)  ,d\left(  y,z\right)
\leqslant\frac{1}{2}d\left(  x_{1},y\right)  \right\}  }\phi_{0}\left(
x_{1},z\right)  \phi_{\alpha}\left(  z,y\right)  dz\\
&  \leqslant\frac{c}{d\left(  x_{1},y\right)  ^{\alpha}}\int_{\left\{
d\left(  x_{1},z\right)  <2d\left(  x_{1},x_{2}\right)  ,d\left(  y,z\right)
\leqslant\frac{1}{2}d\left(  x_{1},y\right)  \right\}  }d\left(
x_{1},z\right)  ^{\alpha}\phi_{0}\left(  x_{1},z\right)  \phi_{\alpha}\left(
z,y\right)  dz\\
&  \leqslant\frac{c}{d\left(  x_{1},y\right)  ^{\alpha}}\phi_{\alpha}\left(
x_{1},y\right)  \int_{\left\{  d\left(  x_{1},z\right)  <2d\left(  x_{1}%
,x_{2}\right)  \right\}  }\phi_{\alpha}\left(  x_{1},z\right)  dz\\
&  \leqslant c\left(  \frac{d\left(  x_{1},x_{2}\right)  }{d\left(
x_{1},y\right)  }\right)  ^{\alpha}\phi_{\alpha}\left(  x_{1},y\right)
\leqslant c\left(  \frac{d\left(  x_{1},x_{2}\right)  }{d\left(
x_{1},y\right)  }\right)  ^{\alpha}\frac{d\left(  x_{1},y\right)  ^{\alpha}%
}{\left\vert B\left(  x_{1},d\left(  x_{1},y\right)  \right)  \right\vert }.
\end{align*}

On the other hand, since $d\left(  x_{1},z\right)  \leqslant2d\left(
x_{1},x_{2}\right)  \leqslant\frac{2}{3}d\left(  x_{1},y\right)  <3d\left(
x_{1},y\right)  $, by Proposition \ref{Phi holder}
\begin{align*}
&  A_{2,2}\left(  x_{1},x_{2},y\right)  \leqslant\\
&  \leqslant c\int_{U,d\left(  x_{1},z\right)  <2d\left(  x_{1},x_{2}\right)
,d\left(  y,z\right)  >\frac{1}{2}d\left(  x_{1},y\right)  }\phi_{0}\left(
x_{1},z\right)  d\left(  x_{1},z\right)  ^{\alpha-\varepsilon}\phi
_{\varepsilon}\left(  z,y\right)  dz\\
&  \leqslant cd\left(  x_{1},x_{2}\right)  ^{\alpha-2\varepsilon}%
\int_{U,d\left(  x_{1},z\right)  <2d\left(  x_{1},x_{2}\right)  ,d\left(
y,z\right)  >\frac{1}{2}d\left(  x_{1},y\right)  }d\left(  x_{1},z\right)
^{\varepsilon}\phi_{0}\left(  x_{1},z\right)  \phi_{\varepsilon}\left(
z,y\right)  dz\\
&  \leqslant cd\left(  x_{1},x_{2}\right)  ^{\alpha-2\varepsilon}\int_{U}%
\phi_{\varepsilon}\left(  x_{1},z\right)  \phi_{\varepsilon}\left(
z,y\right)  dz\\
&  \leqslant cd\left(  x_{1},x_{2}\right)  ^{\alpha-2\varepsilon}%
\phi_{2\varepsilon}\left(  x_{1},y\right)  \leqslant c\left(  \frac{d\left(
x_{1},x_{2}\right)  }{d\left(  x_{1},y\right)  }\right)  ^{\alpha
-2\varepsilon}\frac{d\left(  x_{1},y\right)  ^{\alpha}}{\left\vert B\left(
x_{1},d\left(  x_{1},y\right)  \right)  \right\vert }.
\end{align*}
We can conclude that%
\[
\left\vert A\left(  x_{2},y\right)  -A\left(  x_{1},y\right)  \right\vert
\leqslant c\left(  \frac{d\left(  x_{1},x_{2}\right)  }{d\left(
x_{1},y\right)  }\right)  ^{\alpha-2\varepsilon}\frac{d\left(  x_{1},y\right)
^{\alpha-\varepsilon}}{\left\vert B\left(  x_{1},d\left(  x_{1},y\right)
\right)  \right\vert }.
\]
This completes the proof of (\ref{standard J}).
\end{proof}

\subsection{Local solvability and H\"{o}lder estimates on the highest
derivatives of the solution\label{sec highest}}

Throughout this section we keep Assumptions B, stated at the beginning of
\S \ref{sec:further regularity}. We can now prove one of the main results in
this paper:

\begin{theorem}
[Local solvability of $L$]\label{ThLocalSolv}Under Assumptions B, the function
$\gamma$ is a solution to the equation%
\[
L\gamma\left(  \cdot,y\right)  =0\text{ in }U\setminus\left\{  y\right\}
\text{, for any }y\in U.
\]
Moreover, for any $\beta>0,$ $f\in C_{X}^{\beta}\left(  U\right)  $, the
function%
\begin{equation}
w\left(  x\right)  =-\int_{U}\gamma\left(  x,y\right)  f\left(  y\right)  dy
\label{u}%
\end{equation}
is a $C_{X}^{2}\left(  U\right)  $ solution to the equation $Lw=f$ in $U$ (in
the sense of Definition \ref{def solution})$.$ Hence the operator $L$ is
locally solvable\emph{ }in $\Omega$.

Moreover, if $X_{0}\equiv0$, choosing $U$ small enough, we have the following
positivity property: if $f\in C_{X}^{\beta}\left(  U\right)  ,f\leqslant0$ in
$U$, then the equation $Lw=f$ has at least a $C_{X}^{2}\left(  U\right)  $
solution $w\geqslant0$ in $U$.
\end{theorem}

\begin{proof}
By Theorem \ref{Thm gamma} and Theorem \ref{Thm XXgamma} we already know that
$\gamma\left(  \cdot,y\right)  \in C_{X}^{2}\left(  U\setminus\left\{
y\right\}  \right)  .$ For fixed $y\in U$ and $r>0,$ let $\omega\in
C_{0}^{\infty}\left(  U\right)  ,$ $\omega$ vanishing in the ball $B\left(
y,r\right)  $. Then, by Theorem \ref{Thm gamma} we have%
\[
0=\int\gamma\left(  x,y\right)  L^{\ast}\omega\left(  x\right)  dx=\int
L\gamma\left(  x,y\right)  \omega\left(  x\right)  dx
\]
with $L\gamma\left(  \cdot,y\right)  $ continuous in the support of $\omega$.
Since $r$ and $\omega$ are arbitrary, we get $L\gamma\left(  x,y\right)  =0$
for every $x\in U\setminus\left\{  y\right\}  $, any $y\in U.$

Let now $w$ be as in (\ref{u}) for some $f\in C^{\beta}\left(  U\right)  ,$
$\beta>0$; for any $\psi\in C_{0}^{\infty}\left(  U\right)  $ we can write, by
Theorem \ref{Thm gamma},%
\begin{align}
\int_{U}w\left(  x\right)  L^{\ast}\psi\left(  x\right)  dx  &  =\int
_{U}\left(  -\int_{U}\gamma\left(  x,y\right)  f\left(  y\right)  dy\right)
L^{\ast}\psi\left(  x\right)  dx\nonumber\\
&  =-\int_{U}\left(  \int_{U}\gamma\left(  x,y\right)  L^{\ast}\psi\left(
x\right)  dx\right)  f\left(  y\right)  dy\nonumber\\
&  =\int_{U}\psi\left(  y\right)  f\left(  y\right)  dy. \label{wL*}%
\end{align}
Hence if we show that $Lw$ actually exists and is continuous in $U$, we can
write
\[
\int_{U}w\left(  x\right)  L^{\ast}\psi\left(  x\right)  dx=\int_{U}Lw\left(
x\right)  \psi\left(  x\right)  dx\text{ }\forall\psi\in C_{0}^{\infty}\left(
U\right)  ,
\]
which coupled with (\ref{wL*}) gives $Lw=f$. Actually, we will prove that
$w\in C_{X}^{2}\left(  U\right)  $.

By the results in \S \ref{sec parametrix} it is easy to see that $w\in
C_{X}^{1}\left(  U\right)  $. Namely, by Proposition
\ref{Lemma joint continuity} (ii), $w\in C\left(  U\right)  $ by the estimate
(\ref{gam 1}) while
\[
X_{i}w\left(  x\right)  =-\int_{U}X_{i}\gamma\left(  x,y\right)  f\left(
y\right)  dy
\]
is continuous in $U$ by the estimate (\ref{gam 2}).

Let us write:%
\begin{align*}
X_{j}X_{i}w\left(  x\right)   &  =-X_{j}X_{i}\int_{U}\gamma\left(  x,y\right)
f\left(  y\right)  dy=\\
&  =-X_{j}X_{i}\int_{U}\frac{1}{c_{0}\left(  y\right)  }\left[  P\left(
x,y\right)  +J^{\prime}\left(  x,y\right)  \right]  f\left(  y\right)
dy\equiv A\left(  x\right)  +B\left(  x\right)  .
\end{align*}
By Theorem \ref{Thm XXgamma} we can write%
\begin{equation}
B\left(  x\right)  =-\int_{U}X_{j}X_{i}J^{\prime}\left(  x,y\right)
\widetilde{f}\left(  y\right)  dy, \label{B(x)}%
\end{equation}
having set%
\begin{equation}
\widetilde{f}\left(  y\right)  =\frac{f\left(  y\right)  }{c_{0}\left(
y\right)  } \label{def ftilde}%
\end{equation}
and again by Proposition \ref{Lemma joint continuity} (ii), and the bound
(\ref{Stima XXgamma}), $B$ is continuous in $U$.

Let us now consider%
\begin{equation}
A\left(  x\right)  =-X_{j}X_{i}\int_{U}P\left(  x,y\right)  \widetilde
{f}\left(  y\right)  dy. \label{A(x)}%
\end{equation}
From the computation in the proof of Theorem \ref{Thm Big J_0} we read that%
\[
-X_{i}P\left(  x,y\right)  =k_{1}\left(  x,y\right)
\]
with $k_{1}\left(  x,y\right)  $ kernel of type $1$ in the sense of Definition
\ref{Def abstract k_l}, hence%
\begin{equation}
A\left(  x\right)  =X_{j}\int_{U}k_{1}\left(  x,y\right)  \widetilde{f}\left(
y\right)  dy \label{A(x) int}%
\end{equation}
where the function $\widetilde{f}$ is H\"{o}lder continuous in $U$. To show
that $A\left(  x\right)  $ exists and is continuous we can now proceed as we
did in the proof of Theorem \ref{Thm Big J_0} for the term $X_{j}B\left(
x,y\right)  $, getting, analogously to (\ref{Big J_0 1}) and with the same
notation,
\begin{align*}
A\left(  x\right)   &  =\int_{\mathbb{R}^{m}}a_{0}\left(  h\right)
\int_{\Sigma}Y_{j}D_{1}\Gamma\left(  \Theta_{\eta}\left(  \xi\right)  \right)
b_{0}\left(  k\right)  \left[  \widetilde{f}\left(  z\right)  -\widetilde
{f}\left(  x\right)  \right]  d\eta dh\\
&  +c_{1}\left(  x\right)  \widetilde{f}\left(  x\right)  +\int_{U}%
R_{2}\left(  x,z\right)  \widetilde{f}\left(  z\right)  dz
\end{align*}
where $\xi=\left(  x,h\right)  $, $\eta=\left(  z,k\right)  $, $\Sigma=U\times
I$, $I\subset\mathbb{R}^{m}$ such that $I\supset\operatorname{sprt}a_{0}%
\cup\operatorname{sprt}b_{0}$. Note that here $\widetilde{f}$ plays the role
of the function $\Phi_{0}\left(  \cdot,y\right)  $ in the proof of Theorem
\ref{Thm Big J_0}; since $\widetilde{f}\in C_{X}^{\beta}\left(  U\right)  $
for some $\beta>0,$ it obviously satisfies the properties required in the
definition of $\Phi_{0}\left(  \cdot,y\right)  $. Hence%
\begin{align*}
X_{j}X_{i}w\left(  x\right)   &  =\int_{\mathbb{R}^{m}}a_{0}\left(  h\right)
\int_{\Sigma}Y_{j}D_{1}\Gamma\left(  \Theta_{\eta}\left(  \xi\right)  \right)
b_{0}\left(  k\right)  \left[  \widetilde{f}\left(  z\right)  -\widetilde
{f}\left(  x\right)  \right]  d\eta dh\\
&  +c_{1}\left(  x\right)  \widetilde{f}\left(  x\right)  +\int_{U}%
R_{2}\left(  x,z\right)  \widetilde{f}\left(  z\right)  dz-\int_{U}X_{j}%
X_{i}J^{\prime}\left(  x,z\right)  \widetilde{f}\left(  z\right)  dz,
\end{align*}
and this function is continuous in $U$.

To complete the proof we should prove the existence and continuity of%
\[
X_{0}\int_{U}P\left(  x,z\right)  \widetilde{f}\left(  z\right)  dz.
\]
However, this is very similar to what we have just done.

Finally, the positivity property of $L$ when $X_{0}\equiv0$ and $U$ is small
enough immediately follows from (\ref{u}) and (\ref{gam 3}). So we have finished.
\end{proof}

\bigskip

From the proof of the above theorem we read in particular a representation
formula for the second derivatives $X_{i}X_{j}w$ of our solution. In view of
the proof of local H\"{o}lder continuity of $X_{i}X_{j}w$, we have to localize
our representation formula.

For $\overline{x}\in U\ $and $B\left(  \overline{x},R\right)  \subset U$, pick
a cutoff function%
\begin{equation}
b\in C_{0}^{\infty}\left(  B\left(  \overline{x},R\right)  \right)  \text{
such that }b=1\text{ in }B\left(  \overline{x},\frac{3}{4}R\right)  \text{.}
\label{cutoff 2}%
\end{equation}

For any $\beta>0$, $f\in C_{X}^{\beta}\left(  U\right)  $, let $w$ be the
solution to $Lw=f$ in $U$ assigned by (\ref{u}). Then, for any $x\in B\left(
\overline{x},R\right)  $ we can write:%
\begin{equation}
w\left(  x\right)  =-\int_{B\left(  \overline{x},R\right)  }\gamma\left(
x,y\right)  b\left(  y\right)  f\left(  y\right)  dy+\int_{U}\gamma\left(
x,y\right)  \left[  b\left(  y\right)  -1\right]  f\left(  y\right)  dy.
\label{defw}%
\end{equation}
We also have:

\begin{corollary}
\label{Corollary rep formula}With the notation and assumptions just recalled,
for every $x\in B\left(  \overline{x},\frac{R}{2}\right)  $ and
$i,j=1,2,...,n,$ we have:%
\begin{align*}
X_{j}X_{i}w\left(  x\right)   &  =\int_{U}X_{i}X_{j}\gamma\left(  x,y\right)
\left[  b\left(  y\right)  -1\right]  f\left(  y\right)  dy+c_{1}\left(
x\right)  \widetilde{f}\left(  x\right) \\
&  +\int_{B\left(  \overline{x},R\right)  }k_{2}\left(  x,z\right)  \left[
\widetilde{f}\left(  z\right)  -\widetilde{f}\left(  x\right)  \right]
b\left(  z\right)  dz\\
&  +\int_{B\left(  \overline{x},R\right)  }R_{2}\left(  x,z\right)  b\left(
z\right)  \widetilde{f}\left(  z\right)  dz-\int_{B\left(  \overline
{x},R\right)  }X_{j}X_{i}J^{\prime}\left(  x,z\right)  b\left(  z\right)
\widetilde{f}\left(  z\right)  dz\\
&  \equiv\sum_{k=1}^{5}T_{k}f\left(  x\right)  ,
\end{align*}
where $c_{1}\in C_{X}^{\alpha}\left(  B\left(  \overline{x},\frac{R}%
{2}\right)  \right)  $, $k_{2}$ and $R_{2}$ are a pure kernel and a remainder
of type $2$, respectively, in the sense of Definition \ref{Def abstract k_l}
and $\widetilde{f}$ is defined in (\ref{def ftilde}).
%\begin{align}
%\widetilde{f}\left(  y\right)   &  =\frac{f\left(  y\right)  }{c_{0}\left(
%y\right)  };\nonumber\\
%\widetilde{k}_{1}\left(  x,z\right)   &  =\int_{\mathbb{R}^{m}}a_{0}\left(
%h\right)  \int_{I}Y_{j}D_{1}\Gamma\left(  \Theta_{\eta}\left(  \xi\right)
%\right)  \left[  b_{0}\left(  k\right)  -b_{0}\left(  h\right)  \right]
%dhdk.\label{k_1_tilde}%
%\end{align}

\end{corollary}

\begin{proof}
Let us write
\begin{align*}
w\left(  x\right)   &  =-\int_{B\left(  \overline{x},R\right)  }\gamma\left(
x,y\right)  b\left(  y\right)  f\left(  y\right)  dy+\int_{U}\gamma\left(
x,y\right)  \left[  b\left(  y\right)  -1\right]  f\left(  y\right)  dy\\
&  \equiv K_{1}f\left(  x\right)  +K_{2}f\left(  x\right)  .
\end{align*}
Note that for $x\in B\left(  \overline{x},R/2\right)  $ the integral defining
$K_{2}f\left(  x\right)  $ can be freely differentiated since $\left[
b\left(  y\right)  -1\right]  \neq0$ only if $d\left(  x,y\right)  \geqslant
R/4$, so%
\[
X_{i}X_{j}K_{2}f(x)=\int_{U}X_{i}X_{j}\gamma\left(  x,y\right)  \left[
b\left(  y\right)  -1\right]  f\left(  y\right)  dy.
\]
Arguing as in the proof of Theorems \ref{ThLocalSolv} and \ref{Thm Big J_0} we
have therefore (with $\eta=\left(  z,k\right)  $, $\xi=\left(  x,h\right)  $,
$\Sigma=U\times I$ for $I\supset\operatorname{sprt}a_{0}\cup
\operatorname{sprt}b_{0}$)%
\begin{align*}
X_{j}X_{i}w\left(  x\right)   &  =\int_{U}X_{i}X_{j}\gamma\left(  x,y\right)
\left[  b\left(  y\right)  -1\right]  f\left(  y\right)  dy+c_{1}\left(
x\right)  \widetilde{f}\left(  x\right) \\
&  +\int_{\mathbb{R}^{m}}a_{0}\left(  h\right)  \int_{\Sigma}Y_{j}D_{1}%
\Gamma\left(  \Theta_{\eta}\left(  \xi\right)  \right)  \left[  \widetilde
{f}\left(  z\right)  b\left(  z\right)  -\widetilde{f}\left(  x\right)
b\left(  x\right)  \right]  b_{0}\left(  k\right)  d\eta dh\\
&  +\int_{B\left(  \overline{x},R\right)  }R_{2}\left(  x,z\right)  b\left(
z\right)  \widetilde{f}\left(  z\right)  dz-\int_{B\left(  \overline
{x},R\right)  }X_{j}X_{i}J^{\prime}\left(  x,z\right)  b\left(  z\right)
\widetilde{f}\left(  z\right)  dz.
\end{align*}
Let us rewrite the third term as%
\begin{align*}
&  \int_{\mathbb{R}^{m}}a_{0}\left(  h\right)  \int_{\Sigma}Y_{j}D_{1}%
\Gamma\left(  \Theta_{\eta}\left(  \xi\right)  \right)  \left[  \widetilde
{f}\left(  z\right)  -\widetilde{f}\left(  x\right)  \right]  b_{0}\left(
k\right)  b\left(  z\right)  d\eta dh\\
&  +\widetilde{f}\left(  x\right)  \int_{\mathbb{R}^{m}}a_{0}\left(  h\right)
\int_{\Sigma}Y_{j}D_{1}\Gamma\left(  \Theta_{\eta}\left(  \xi\right)  \right)
\left[  b\left(  z\right)  -b\left(  x\right)  \right]  b_{0}\left(  k\right)
d\eta dh\\
&  =\int_{B\left(  \overline{x},R\right)  }k_{2}\left(  x,z\right)  \left[
\widetilde{f}\left(  z\right)  -\widetilde{f}\left(  x\right)  \right]
b\left(  z\right)  dz\\
&  +\widetilde{f}\left(  x\right)  c_{2}\left(  x\right)
\end{align*}
where $k_{2}$ is a kernel of type 2, while%
\begin{align*}
c_{2}\left(  x\right)   &  =\int_{U}k_{2}\left(  x,z\right)  \left[  b\left(
z\right)  -b\left(  x\right)  \right]  dz\\
&  =\int_{U}k_{2}\left(  x,z\right)  \left[  b\left(  z\right)  -1\right]  dz
\end{align*}
is another $C_{X}^{\alpha}\left(  B\left(  \overline{x},\frac{R}{2}\right)
\right)  $ function. Namely, recalling that $b=1$ in $B\left(  \overline
{x},\frac{3}{4}R\right)  ,$ for any $x_{1},x_{2}\in B\left(  \overline
{x},R/2\right)  $, we have%
\begin{equation}
\left\vert c_{2}\left(  x_{2}\right)  -c_{2}\left(  x_{1}\right)  \right\vert
\leqslant\int_{U}\left\vert k_{2}\left(  x_{2},z\right)  -k_{2}\left(
x_{1},z\right)  \right\vert \left[  1-b\left(  z\right)  \right]  dz.
\label{c_2 holder}%
\end{equation}
Note that, from%
\[
k_{2}\left(  x,y\right)  =\int_{\mathbb{R}^{m}}\int_{\mathbb{R}^{m}}%
D_{2}\Gamma\left(  \Theta_{\left(  y,k\right)  }\left(  x,h\right)  \right)
a_{0}\left(  h\right)  b_{0}\left(  k\right)  dhdk,
\]
by Proposition \ref{Prop bad theta} (ii) we read that%
\begin{align}
\left\vert k_{2}\left(  x,y\right)  \right\vert  &  \leqslant c\phi_{0}\left(
x,y\right)  ;\label{standard 1 k_2}\\
\left\vert X_{i}k_{2}\left(  x,y\right)  \right\vert  &  \leqslant c\phi
_{-1}\left(  x,y\right)  \text{ for }i=1,2,...,n;\nonumber\\
\left\vert X_{0}k_{2}\left(  x,y\right)  \right\vert  &  \leqslant c\phi
_{-2}\left(  x,y\right)  ,\nonumber
\end{align}
hence by Lagrange theorem (Proposition \ref{Prop Lagrange}),
\begin{equation}
\left\vert k_{2}\left(  x_{2},z\right)  -k_{2}\left(  x_{1},z\right)
\right\vert \leqslant c\frac{d\left(  x_{1},x_{2}\right)  }{d\left(
x_{1},z\right)  }\frac{1}{\left\vert B\left(  x_{1},d\left(  x_{1},z\right)
\right)  \right\vert }\text{ for }d\left(  x_{1},z\right)  \geqslant2d\left(
x_{1},x_{2}\right)  . \label{standard k_2}%
\end{equation}
Now, note that the integrand function in (\ref{c_2 holder}) does not vanish
only for $d\left(  x_{1},z\right)  \geqslant R/4,$ $d\left(  x_{2},z\right)
\geqslant R/4.$ Hence if $d\left(  x_{1},x_{2}\right)  \leqslant R/8$ by
(\ref{standard k_2}) we get%
\[
\left\vert c_{2}\left(  x_{2}\right)  -c_{2}\left(  x_{1}\right)  \right\vert
\leqslant c\left(  R\right)  d\left(  x_{1},x_{2}\right)  .
\]
On the other hand, if $d\left(  x_{1},x_{2}\right)  >R/8,$%
\[
\left\vert c_{2}\left(  x_{2}\right)  -c_{2}\left(  x_{1}\right)  \right\vert
\leqslant\left\vert c_{2}\left(  x_{2}\right)  \right\vert +\left\vert
c_{2}\left(  x_{1}\right)  \right\vert \leqslant c\left(  R\right)  \leqslant
c\left(  R\right)  d\left(  x_{1},x_{2}\right)  ,
\]
and $c_{2}\in C_{X}^{\alpha}\left(  B\left(  \overline{x},\frac{R}{2}\right)
\right)  $. This completes the proof.
\end{proof}

The rest of this section will be devoted to the proof of the following:

\begin{theorem}
\label{Thm holder w}For any $\beta\in\left(  0,\alpha\right)  \ $and $f\in
C_{X}^{\beta}\left(  U\right)  $, let $w\in C_{X}^{2}\left(  U\right)  $ be
the solution to $Lw=f$ in $U$ assigned by (\ref{u}).\ Then $w\in
C_{X,loc}^{2,\beta}\left(  U\right)  $. More precisely, for any $U^{\prime
}\Subset U$ there exists $c>0$ (depending on $U$, $U^{\prime}$, $\beta$ and on
the vector fields as specified at the beginning of section
\ref{sec:further regularity}) such that%
\begin{equation}
\left\Vert w\right\Vert _{C_{X}^{2,\beta}\left(  U^{\prime}\right)  }\leqslant
c\left\Vert f\right\Vert _{C_{X}^{\beta}\left(  U\right)  }.
\label{schauder compact}%
\end{equation}

\end{theorem}

\begin{corollary}
[$C_{X}^{2,\beta}$ local solvability]Under assumptions B, for every $\beta
\in\left(  0,\alpha\right)  $ the operator $L$ is locally $C_{X}^{2,\beta}$
solvable in $\Omega$ in the following senses:

(i) for every $\overline{x}\in\Omega$ there exists a neighborhood $U$ of
$\overline{x}$ such that for every $f\in C_{X}^{\beta}\left(  U\right)  $
there exists a solution $u\in C_{X,loc}^{2,\beta}\left(  U\right)  $ to $Lu=f$
in $U.$

(ii) for every $\overline{x}\in\Omega$ there exists a neighborhood $U$ of
$\overline{x}$ such that for every $f\in C_{X,0}^{\beta}\left(  U\right)  $
there exists a solution $u\in C_{X}^{2,\beta}\left(  U\right)  $ to $Lu=f$ in
$U.$
\end{corollary}

\begin{proof}
Point (i) immediately follows by the above theorem and Theorem
\ref{ThLocalSolv}. As to point (ii), let $U$ be the neighborhood of
$\overline{x}$ given by point (i), and let $U^{\prime}$ be another
neighborhood of $\overline{x}$ such that $U^{\prime}\Subset U$. For any $f\in
C_{X,0}^{\beta}\left(  U^{\prime}\right)  $ we can regard $f$ also as a
function in $C_{X,0}^{\beta}\left(  U\right)  ,$ and solve $Lu=f$ in $U$
getting a $u\in C_{X,loc}^{2,\beta}\left(  U\right)  $ by point (i); hence in
particular $u\in C_{X}^{2,\beta}\left(  U^{\prime}\right)  $. Then $U^{\prime
}$ is the required neighborhood.
\end{proof}

Since, in order to prove the above theorem, we will apply several abstract
results about singular and fractional integrals, it is time to explain what is
the suitable abstract context for the present situation. Recall that in our
neighborhood $U$ we have the distance $d,$ such that the Lebesgue measure is
\emph{locally }doubling (see Theorem \ref{Thm doubling nonsmooth}). However,
we cannot assure the validity of a \emph{global} doubling condition in $U$,
which should mean:%
\begin{equation}
\left\vert B\left(  x,2r\right)  \cap U\right\vert \leqslant c\left\vert
B\left(  x,r\right)  \cap U\right\vert \text{ for any }x\in U,r>0\text{.}
\label{global doubling}%
\end{equation}
Actually, even for the Carnot-Carath\'{e}odory distance induced by smooth
H\"{o}rmander's vector fields, condition (\ref{global doubling}) is known when
$U$ is for instance a metric ball and the drift term $X_{0}$ is lacking; in
presence of a drift, however, the distance $d$ does not satisfy the segment
property, and the validity of a condition (\ref{global doubling}) on some
reasonable $U$ seems to be an open problem (fur further details on this issue
we refer to the introduction of \cite{BZ2}). This means that in our situation
$\left(  U,d,dx\right)  $ is not a space of homogeneous type in the sense of
Coifman-Weiss. However, $\left(  U,d,dx\right)  $ fits the assumptions of
\emph{locally homogeneous spaces }as defined in \cite{BZ2}. We will apply some
results proved in \cite{BZ2} which assure the local $C^{\alpha}$ continuity of
singular and fractional integrals defined by a kernel of the kind%
\[
a\left(  x\right)  k\left(  x,y\right)  b\left(  y\right)
\]
(with $a,b$ smooth cutoff functions) provided that the kernel $k$ satisfies
natural assumptions which never involve integration over domains of the kind
$B\left(  x,r\right)  \cap U$, but only over balls $B\left(  x,r\right)
\Subset U,$ which makes our local doubling condition usable. Before starting
the proof of the above theorem we need the following

\begin{definition}
We say that the a kernel $k\left(  x,y\right)  $ satisfies the standard
estimates of fractional integrals with (positive) exponents $\nu,\beta$ in
$B\left(  \overline{x},R\right)  $ if%
\[
\left\vert k\left(  x,y\right)  \right\vert \leqslant c\frac{d\left(
x,y\right)  ^{\nu}}{\left\vert B\left(  x,d\left(  x,y\right)  \right)
\right\vert }%
\]
for every $x,y\in B\left(  \overline{x},R\right)  $, and%
\[
\left\vert k\left(  x,y\right)  -k\left(  x_{0},y\right)  \right\vert
\leqslant c\frac{d\left(  x_{0},y\right)  ^{\nu}}{\left\vert B\left(
x_{0},d\left(  x_{0},y\right)  \right)  \right\vert }\left(  \frac{d\left(
x_{0},x\right)  }{d\left(  x_{0},y\right)  }\right)  ^{\beta}%
\]
for every $x_{0},x,y\in B\left(  \overline{x},R\right)  $ such that $d\left(
x_{0},y\right)  \geqslant Md\left(  x_{0},x\right)  $ for suitable $M>1.$

We say that $k\left(  x,y\right)  $ satisfies the standard estimates of
singular integrals if the previous estimates hold with $\nu=0$ and some
positive $\beta$.
\end{definition}

\begin{proof}
[Proof of Theorem \ref{Thm holder w}, first part]Fix $U^{\prime}\Subset
U\ $and choose $R_{0}>0$ such that for any $\overline{x}\in U^{\prime}\ $one
has $B\left(  \overline{x},KR_{0}\right)  \subset U$, for some large number
$K>1$ which is not important to specify (it comes out from some proofs in
\cite{BZ2}). For any $R\leqslant R_{0}$, pick a cutoff function $b\in
C_{0}^{\infty}\left(  B\left(  \overline{x},R\right)  \right)  $ such that
$b\equiv1$ in $B\left(  \overline{x},\frac{3}{4}R\right)  $. Then for any
$x\in B\left(  \overline{x},R/2\right)  $ the representation formula proved in
Corollary \ref{Corollary rep formula} holds:%
\[
X_{i}X_{j}w\left(  x\right)  =\sum_{k=1}^{5}T_{k}f\left(  x\right)  \text{ for
}i,j=1,2,....,n\text{.}%
\]
Our proof will mainly consist in showing that for any $\beta\in\left(
0,\alpha\right)  $ and $f\in C_{X}^{\beta}\left(  U\right)  ,$%
\[
\left\vert X_{i}X_{j}w\left(  x_{1}\right)  -X_{i}X_{j}w\left(  x_{2}\right)
\right\vert \leqslant cd\left(  x_{1},x_{2}\right)  ^{\beta}\left\Vert
f\right\Vert _{C^{\beta}\left(  U\right)  }%
\]
for any $x_{1},x_{2}\in B\left(  \overline{x},\frac{R}{2}\right)  $. We are
going to show how to bound the $C^{\beta}\left(  B\left(  \overline{x}%
,\frac{R}{2}\right)  \right)  $ seminorm of each term in this formula,
starting with the easier ones.

Consider the operator%
\[
T_{1}f\left(  x\right)  =\int_{U}X_{i}X_{j}\gamma\left(  x,y\right)  \left[
b\left(  y\right)  -1\right]  f\left(  y\right)  dy.
\]
Then by our choice of the cutoff function $b,$ we have, for $x_{1},x_{2}\in
B\left(  \overline{x},R/2\right)  ,$%
\begin{align*}
&  \left\vert T_{1}f\left(  x_{1}\right)  -T_{1}f\left(  x_{2}\right)
\right\vert \\
&  \leqslant\left\Vert f\right\Vert _{C^{0}\left(  U\right)  }\int_{U,d\left(
\overline{x},y\right)  >\frac{3}{4}R,d\left(  x_{1},y\right)  >\frac{R}%
{4},d\left(  x_{2},y\right)  >\frac{R}{4}}\left\vert X_{i}X_{j}\gamma\left(
x_{1},y\right)  -X_{i}X_{j}\gamma\left(  x_{2},y\right)  \right\vert dy\\
&  =\left\Vert f\right\Vert _{C^{0}\left(  U\right)  }\left(  \int_{2d\left(
x_{1},x_{2}\right)  <d\left(  x_{1},y\right)  ,d\left(  x_{1},y\right)
>\frac{R}{4}}\left(  ...\right)  dy+\int_{2d\left(  x_{1},x_{2}\right)
\geqslant d\left(  x_{1},y\right)  ,d\left(  x_{1},y\right)  >\frac{R}%
{4},d\left(  x_{2},y\right)  >\frac{R}{4}}\left(  ...\right)  dy\right)
\end{align*}
by (\ref{standard gamma}) and (\ref{Stima XXgamma})%
\begin{align*}
&  \leqslant c\left\Vert f\right\Vert _{C^{0}\left(  U\right)  }\left\{
d\left(  x_{1},x_{2}\right)  ^{\alpha-\varepsilon}\int_{d\left(
x_{1},y\right)  >\frac{R}{4}}\frac{dy}{d\left(  x_{1},y\right)  ^{\alpha
-\varepsilon}\left\vert B\left(  x_{1},d\left(  x_{1},y\right)  \right)
\right\vert }\right. \\
&  +\left.  \frac{1}{R}\int_{2d\left(  x_{1},x_{2}\right)  \geqslant d\left(
x_{1},y\right)  }\frac{d\left(  x_{1},y\right)  }{\left\vert B\left(
x_{1},d\left(  x_{1},y\right)  \right)  \right\vert }dy+\frac{1}{R}%
\int_{3d\left(  x_{1},x_{2}\right)  \geqslant d\left(  x_{2},y\right)  }%
\frac{d\left(  x_{2},y\right)  }{\left\vert B\left(  x_{2},d\left(
x_{2},y\right)  \right)  \right\vert }dy\right\} \\
&  \leqslant\left\Vert f\right\Vert _{C^{0}\left(  U\right)  }\left\{
c\left(  R\right)  d\left(  x_{1},x_{2}\right)  ^{\alpha-\varepsilon}%
+\frac{d\left(  x_{1},x_{2}\right)  }{R}\right\}  =cd\left(  x_{1}%
,x_{2}\right)  ^{\alpha-\varepsilon}\left\Vert f\right\Vert _{C^{0}\left(
U\right)  },
\end{align*}
so that%
\[
\left\Vert T_{1}f\right\Vert _{C_{X}^{\beta}\left(  B\left(  \overline
{x},R/2\right)  \right)  }\leqslant c\left(  \beta,R\right)  \left\Vert
f\right\Vert _{C^{0}\left(  U\right)  }\text{ }\forall\beta<\alpha.
\]
Next we introduce a second cutoff function $a\in C_{0}^{\infty}\left(
B\left(  \overline{x},\frac{3}{4}R\right)  \right)  $ such that $a\equiv1$ in
$B\left(  \overline{x},\frac{R}{2}\right)  $. For $x\in B\left(  \overline
{x},\frac{R}{2}\right)  $ we have $T_{k}f\left(  x\right)  =\widetilde{T}%
_{k}f\left(  x\right)  $, $k=4,5$ with%
\begin{align*}
\widetilde{T}_{4}f\left(  x\right)   &  =a\left(  x\right)  \int_{B\left(
\overline{x},R\right)  }R_{2}\left(  x,z\right)  b\left(  z\right)
\widetilde{f}\left(  z\right)  dz\\
\widetilde{T}_{5}f\left(  x\right)   &  =-a\left(  x\right)  \int_{B\left(
\overline{x},R\right)  }X_{j}X_{i}J^{\prime}\left(  x,z\right)  b\left(
z\right)  \widetilde{f}\left(  z\right)  dz.
\end{align*}
These new operators have the form
\[
\widetilde{T}_{j}f(x)=\int_{B\left(  \overline{x},R\right)  }a\left(
x\right)  k_{j}\left(  x,y\right)  \frac{f\left(  y\right)  }{c_{0}\left(
y\right)  }b\left(  y\right)  dy\text{ for }j=4,5,
\]
where the kernels $k_{j}\left(  x,y\right)  $ satisfy the standard estimates
of fractional integrals. Indeed, by Definition \ref{Def abstract remainder}
and Proposition \ref{Prop bad theta} (ii),\ the kernel $k_{4}$ satisfies
\begin{align*}
\left\vert k_{4}\left(  x,z\right)  \right\vert  &  \leqslant c\phi_{\alpha
}\left(  x,z\right)  \leqslant c\frac{d\left(  x,z\right)  ^{\alpha}%
}{\left\vert B\left(  x,d\left(  x,z\right)  \right)  \right\vert };\\
\left\vert X_{k}k_{4}\left(  x,z\right)  \right\vert  &  \leqslant
c\phi_{\alpha-1}\left(  x,z\right)  ;\\
\left\vert X_{0}k_{4}\left(  x,z\right)  \right\vert  &  \leqslant
c\phi_{\alpha-2}\left(  x,z\right)  .
\end{align*}
If $d\left(  x_{1},z\right)  \geqslant2d\left(  x_{1},x_{2}\right)  $, then by
Lagrange theorem we can bound
\begin{align}
\left\vert k_{4}\left(  x_{1},z\right)  -k_{4}\left(  x_{2},z\right)
\right\vert  &  \leqslant c\left\{  d\left(  x_{1},x_{2}\right)  \phi
_{\alpha-1}\left(  x_{1},z\right)  +d\left(  x_{1},x_{2}\right)  ^{2}%
\phi_{\alpha-2}\left(  x_{1},z\right)  \right\} \nonumber\\
&  \leqslant cd\left(  x_{1},x_{2}\right)  ^{\alpha-\varepsilon}%
\phi_{\varepsilon}\left(  x_{1},z\right)  . \label{k_2-alfa}%
\end{align}
Then $k_{4}$ satisfies the standard estimates of fractional integrals with
exponents $\nu,\alpha$, for any $\nu<\alpha$;

The kernel $k_{5}$ satisfies, by (\ref{Stima XXJ}) and (\ref{standard J})
(note that the cutoff function $a\left(  x\right)  $ compensates the local
charachter of those bounds), the standard estimates of fractional integrals
with exponents $\nu,\beta,$ for any $\nu$ and $\beta$ both $<\alpha$, hence by
\cite[Thm. 5.8]{BZ2}, for any $\beta<\alpha$%
\[
\left\Vert T_{j}f\right\Vert _{C_{X}^{\beta}\left(  B\left(  \overline
{x},R/2\right)  \right)  }=\left\Vert \widetilde{T}_{j}f\right\Vert
_{C_{X}^{\beta}\left(  B\left(  \overline{x},R/2\right)  \right)  }\leqslant
c\left\Vert f\right\Vert _{C_{X}^{\beta}\left(  B\left(  \overline
{x},R\right)  \right)  }\text{ for }j=4,5,
\]
with $c$ depending on $R$ and $\beta$.

Next, $T_{2}f(x)=\frac{c_{1}\left(  x\right)  }{c_{0}\left(  x\right)  }f(x)$,
with $c_{1},c_{0}$ H\"{o}lder continuous functions of exponent $\alpha$ and
$c_{0}$ bounded away from zero.

%On the other hand, for $x\in B\left(  \overline{x},\frac{R}{2}\right)  $ we
%have $T_{3}f\left(  x\right)  =\widetilde{T}_{3}f\left(  x\right)  $ with
%\[
%\widetilde{T}_{3}f\left(  x\right)  =\frac{f\left(  x\right)  }{c_{0}\left(
%x\right)  }\cdot a\left(  x\right)  \int_{B\left(  \overline{x},R\right)
%}\widetilde{k}_{1}\left(  x,z\right)  b\left(  z\right)  dz,
%\]
%where $\widetilde{k}_{1}\left(  x,z\right)  $ is readily checked to satisfy
%the standard estimates of fractional integrals with exponent $1,1$, hence the
%operator
%\[
%g\mapsto a\left(  \cdot\right)  \int_{B\left(  \overline{x},R\right)
%}\widetilde{k}_{1}\left(  \cdot,z\right)  g\left(  z\right)  b\left(
%z\right)  dz
%\]
%maps $C_{X}^{\beta^{\prime}}\left(  B\left(  \overline{x},R\right)  \right)  $
%into $C_{X}^{\beta^{\prime}}\left(  B\left(  \overline{x},R/2\right)  \right)
%$ continuously, therefore the function
%\[
%a_{3}\left(  x\right)  =a\left(  x\right)  \int_{B\left(  \overline
%{x},R\right)  }\widetilde{k}_{1}\left(  x,z\right)  b\left(  z\right)  dz
%\]
%belongs to $C_{X}^{\beta^{\prime}}\left(  B\left(  \overline{x},R/2\right)
%\right)  $ SISTEMA. Chiarire qui che \`{e} $\beta^{\prime}$. We conclude
%\[
%\left\Vert T_{k}f\right\Vert _{C_{X}^{\beta^{\prime}}\left(  B\left(
%\overline{x},R/2\right)  \right)  }\leq c\left(  \beta^{\prime},R\right)
%\left\Vert f\right\Vert _{C_{X}^{\beta}\left(  B\left(  \overline
%{x},R/2\right)  \right)  }\text{ for }k=2,3.
%\]

We are left to handle the term%
\[
T_{3}f(x)=\int_{B\left(  \overline{x},R\right)  }k_{2}\left(  x,z\right)
\left[  \widetilde{f}\left(  z\right)  -\widetilde{f}\left(  x\right)
\right]  b\left(  z\right)  dz
\]
with $k_{2}$ pure kernel of order $2$, satisfying the standard estimates of
singular integrals (see (\ref{standard 1 k_2}), (\ref{standard k_2})).
Moreover, the same is true for the kernel
\[
\widetilde{k}_{2}\left(  x,y\right)  =a\left(  x\right)  k_{2}\left(
x,y\right)  b\left(  y\right)  .
\]
In order to deduce an H\"{o}lder estimate for $T_{3}f$ we still need to
establish a suitable cancellation property for $\widetilde{k}_{2}$. So, let us
pause for a moment this proof and pass to this auxiliary result.
\end{proof}

\begin{proposition}
[Cancellation property]There exists $C>0$ such that for a.e. $x\in B\left(
\overline{x},R\right)  $ and $0<\varepsilon_{1}<\varepsilon_{2}<\infty$
\begin{equation}
\left\vert \int_{\varepsilon_{1}<d(x,y)<\varepsilon_{2}}a(x)k_{2}\left(
x,y\right)  b(y)\,dy\right\vert \leqslant C. \label{canc}%
\end{equation}

\end{proposition}

\begin{proof}
By Proposition 5.1 in \cite{BZ2}, it is enough to prove the following
cancellation property for $k_{2}$: there exists $C>0$ such that for a.e. $x\in
B\left(  \overline{x},R_{0}\right)  $ and every $\varepsilon_{1}%
,\varepsilon_{2}$ such that $0<\varepsilon_{1}<\varepsilon_{2}$ and $B\left(
x,\varepsilon_{2}\right)  \subset U,$%
\begin{equation}
\left\vert \int_{\varepsilon_{1}<d(x,y)<\varepsilon_{2}}k_{2}\left(
x,y\right)  \,dy\right\vert \leqslant C. \label{standard 3}%
\end{equation}
According to Definition \ref{Def abstract k_l} of kernel of type $2$ we write
\begin{align*}
&  \int_{\varepsilon_{1}<d(x,y)<\varepsilon_{2}}k_{2}\left(  x,y\right)
\,dy=\\
=  &  \int_{\varepsilon_{1}<d(x,y)<\varepsilon_{2}}\int_{\mathbb{R}^{m}}%
\int_{\mathbb{R}^{m}}D_{2}\Gamma\left(  \Theta_{\left(  y,k\right)  }\left(
x,h\right)  \right)  a_{0}\left(  h\right)  b_{0}\left(  k\right)
dhdk\,dy+\int_{\varepsilon_{1}<d(x,y)<\varepsilon_{2}}R_{2}\left(  x,y\right)
\,dy,
\end{align*}
where the last integral is uniformly bounded in $\varepsilon_{1}%
,\varepsilon_{2}$ since the remainder $R_{2}$ is locally integrable.

We can assume $\varepsilon_{2}<1$. Let us recall that%
\[
c\left\Vert \Theta_{\left(  y,k\right)  }\left(  x,h\right)  \right\Vert
\geqslant d_{\widetilde{X}}\left(  \left(  x,h\right)  ,\left(  y,k\right)
\right)  \geqslant d\left(  x,y\right)  ,
\]
then
\begin{align*}
\int_{\varepsilon_{1}<d\left(  x,y\right)  <\varepsilon_{2}}  &  \left(
\int_{\mathbb{R}^{m}}\int_{\mathbb{R}^{m}}D_{2}\Gamma\left(  \Theta_{\left(
y,k\right)  }\left(  x,h\right)  \right)  a_{0}\left(  h\right)  b_{0}\left(
k\right)  dhdk\right)  dy\\
&  =\int_{\mathbb{R}^{m}}a_{0}\left(  h\right)  \left(  \int_{\varepsilon
_{1}<c\left\Vert \Theta_{\left(  y,k\right)  }\left(  x,h\right)  \right\Vert
<\varepsilon_{2}}D_{2}\Gamma\left(  \Theta_{\left(  y,k\right)  }\left(
x,h\right)  \right)  b_{0}\left(  k\right)  dkdy\right)  dh\\
&  +\int_{\mathbb{R}^{m}}a_{0}\left(  h\right)  \left(  \int_{c\left\Vert
\Theta_{\left(  y,k\right)  }\left(  x,h\right)  \right\Vert >\varepsilon
_{2},d\left(  x,y\right)  <\varepsilon_{2}}D_{2}\Gamma\left(  \Theta_{\left(
y,k\right)  }\left(  x,h\right)  \right)  b_{0}\left(  k\right)  dkdy\right)
dh\\
&  -\int_{\mathbb{R}^{m}}a_{0}\left(  h\right)  \left(  \int_{c\left\Vert
\Theta_{\left(  y,k\right)  }\left(  x,h\right)  \right\Vert >\varepsilon
_{1},d\left(  x,y\right)  <\varepsilon_{1}}D_{2}\Gamma\left(  \Theta_{\left(
y,k\right)  }\left(  x,h\right)  \right)  b_{0}\left(  k\right)  dkdy\right)
dh\\
&  \equiv C^{\varepsilon_{1},\varepsilon_{2}}\left(  x\right)  +D^{\varepsilon
_{2}}\left(  x\right)  -E^{\varepsilon_{1}}\left(  x\right)  .
\end{align*}
To handle $C^{\varepsilon_{1},\varepsilon_{2}}\left(  x\right)  $ we start
rewriting%
\begin{align*}
C^{\varepsilon_{1},\varepsilon_{2}}\left(  x\right)   &  =\int_{\mathbb{R}%
^{m}}a_{0}\left(  h\right)  \left(  \int_{\varepsilon_{1}<c\left\Vert
\Theta_{\left(  y,k\right)  }\left(  x,h\right)  \right\Vert <\varepsilon_{2}%
}D_{2}\Gamma\left(  \Theta_{\left(  y,k\right)  }\left(  x,h\right)  \right)
\left[  b_{0}\left(  k\right)  -b_{0}\left(  h\right)  \right]  dkdy\right)
dh\\
&  +\int_{\mathbb{R}^{m}}a_{0}\left(  h\right)  b_{0}\left(  h\right)  \left(
\int_{\varepsilon_{1}<c\left\Vert \Theta_{\left(  y,k\right)  }\left(
x,h\right)  \right\Vert <\varepsilon_{2}}D_{2}\Gamma\left(  \Theta_{\left(
y,k\right)  }\left(  x,h\right)  \right)  dkdy\right)  dh\\
&  \equiv C_{1}^{\varepsilon_{1},\varepsilon_{2}}\left(  x\right)
+C_{2}^{\varepsilon_{1},\varepsilon_{2}}\left(  x\right)  .
\end{align*}
As to $C_{1}^{\varepsilon_{1},\varepsilon_{2}}\left(  x\right)  $, since
\[
\left\vert b_{0}\left(  k\right)  -b_{0}\left(  h\right)  \right\vert
\leqslant c\left\vert k-h\right\vert \leqslant c\left\Vert \Theta_{\left(
y,k\right)  }\left(  x,h\right)  \right\Vert ,
\]
we have%
\begin{align*}
\left\vert C_{1}^{\varepsilon_{1},\varepsilon_{2}}\left(  x\right)
\right\vert  &  \leqslant\int_{\mathbb{R}^{m}}\left\vert a_{0}\left(
h\right)  \right\vert \left(  \int_{\left\Vert \Theta_{\left(  y,k\right)
}\left(  x,h\right)  \right\Vert <\varepsilon_{2}}\frac{c}{\left\Vert
\Theta_{\left(  y,k\right)  }\left(  x,h\right)  \right\Vert ^{Q-1}%
}dkdy\right)  dh\\
&  \leqslant c\varepsilon_{2}\int_{\mathbb{R}^{m}}\left\vert a_{0}\left(
h\right)  \right\vert dh\leqslant c.
\end{align*}
As to $C_{2}^{\varepsilon_{1},\varepsilon_{2}}\left(  x\right)  $, by the
change of variables $\left(  y,k\right)  \mapsto u=\Theta_{\left(  y,k\right)
}\left(  x,h\right)  $ and Proposition \ref{Prop bad theta} we have, letting
$\xi=\left(  x,h\right)  ,$%
\[
C_{2}^{\varepsilon_{1},\varepsilon_{2}}\left(  x\right)  =\int_{\mathbb{R}%
^{m}}a_{0}\left(  h\right)  b_{0}\left(  h\right)  c\left(  \xi\right)
\left(  \int_{\varepsilon_{1}<c\left\Vert u\right\Vert <\varepsilon_{2}}%
D_{2}\Gamma\left(  u\right)  \left(  1+\chi\left(  \xi,u\right)  \right)
du\right)  dh.
\]
Keeping in mind the vanishing property of $D_{2}\Gamma,$ that is%
\[
\int_{\varepsilon_{1}<c\left\Vert u\right\Vert <\varepsilon_{2}}D_{2}%
\Gamma\left(  u\right)  du=0,
\]
we have%
\[
C_{2}^{\varepsilon_{1},\varepsilon_{2}}\left(  x\right)  =\int_{\mathbb{R}%
^{m}}a_{0}\left(  h\right)  b_{0}\left(  h\right)  c\left(  \xi\right)
\left(  \int_{\varepsilon_{1}<c\left\Vert u\right\Vert <\varepsilon_{2}}%
D_{2}\Gamma\left(  u\right)  \chi\left(  \xi,u\right)  du\right)  dh
\]
which is uniformly bounded in $\varepsilon_{1},\varepsilon_{2}$ since%
\[
\int_{\varepsilon_{1}<c\left\Vert u\right\Vert <\varepsilon_{2}}\left\vert
D_{2}\Gamma\left(  u\right)  \chi\left(  \xi,u\right)  \right\vert
du\leqslant\int_{c\left\Vert u\right\Vert <\varepsilon_{2}}\frac{c}{\left\Vert
u\right\Vert ^{Q-\alpha}}du\leqslant c\varepsilon_{2}^{\alpha}\leqslant c.
\]
Let us come to the terms $D^{\varepsilon_{2}}\left(  x\right)  $ and
$E^{\varepsilon_{1}}\left(  x\right)  $. Choosing some small $\delta>0$ we can
write, by Corollary \ref{coroll Psi beta},%
\begin{align*}
\left\vert D^{\varepsilon_{2}}\left(  x\right)  \right\vert  &  \leqslant
\int_{\mathbb{R}^{m}}a_{0}\left(  h\right)  \left(  \int_{c\left\Vert
\Theta_{\left(  y,k\right)  }\left(  x,h\right)  \right\Vert >\varepsilon
_{2},d\left(  x,y\right)  <\varepsilon_{2}}\left\Vert \Theta_{\left(
y,k\right)  }\left(  x,h\right)  \right\Vert ^{-Q}b_{0}\left(  k\right)
dkdy\right)  dh\\
&  \leqslant\frac{1}{\varepsilon_{2}^{\delta}}\int_{d\left(  x,y\right)
\leqslant\varepsilon_{2}}\left(  \int_{\mathbb{R}^{m}}\int_{\mathbb{R}^{m}%
}\left\Vert \Theta_{\left(  y,k\right)  }\left(  x,h\right)  \right\Vert
^{-Q+\delta}a_{0}\left(  h\right)  b_{0}\left(  k\right)  dkdh\right)  dy\\
&  \leqslant\frac{c}{\varepsilon_{2}^{\delta}}\int_{d\left(  x,y\right)
\leqslant\varepsilon_{2}}\phi_{\delta}\left(  x,y\right)  dy\leqslant\frac
{c}{\varepsilon_{2}^{\delta}}\cdot\varepsilon_{2}^{\delta}=c
\end{align*}
and the term $E^{\varepsilon_{1}}\left(  x\right)  $ can be bounded at the
same way.
\end{proof}

\bigskip

\begin{proof}
[Conclusion of the proof of Theorem \ref{Thm holder w}]We are left to prove
the $C_{X}^{\beta}$ continuity of the operator $T_{3}$. Let us consider first%
\[
\widetilde{T}_{3}f(x)=\int_{B\left(  \overline{x},R\right)  }\widetilde{k}%
_{2}\left(  x,y\right)  \left[  \widetilde{f}\left(  y\right)  -\widetilde
{f}\left(  x\right)  \right]  dy.
\]
We know that the kernel $\widetilde{k}_{2}\left(  x,y\right)  $ satisfies the
standard estimates of singular integrals with exponent $\beta=1$ (see the end
of the first part of this proof) and the cancellation property
(\ref{standard 3}). This is enough to repeat \textit{verbatim} the proof of
Theorem 2.7 in \cite{BB}: the quantity%
\[
\widetilde{T}_{3}f\left(  x\right)  -\widetilde{T}_{3}f\left(  x_{0}\right)
\]
is exactly the quantity which is called $A$ in that proof, see \cite[p.183]%
{BB}, and the proof of the bound%
\begin{equation}
\left\vert T_{3}f\left(  x\right)  -T_{3}f\left(  x_{0}\right)  \right\vert
=\left\vert \widetilde{T}_{3}f\left(  x\right)  -\widetilde{T}_{3}f\left(
x_{0}\right)  \right\vert \leqslant cd\left(  x,x_{0}\right)  ^{\beta
}\left\Vert f\right\Vert _{C_{X}^{\beta}\left(  B\left(  \overline
{x},R\right)  \right)  }\text{ \ \ }\forall\beta<1 \label{T_3 holder}%
\end{equation}
for any $x,x_{0}\in B\left(  \overline{x},R/2\right)  $ only relies on the
properties of the kernel that we have already pointed out. In particular,
since the integral defining $\widetilde{T}_{3}f$ is over $B\left(
\overline{x},R\right)  $ and $B\left(  \overline{x},3R\right)  \subset U$, we
can safely apply the local doubling condition on the small balls which are
involved in that proof. Combining (\ref{T_3 holder}) with the first part of
the proof of this theorem, we can write%
\[
\left\vert X_{i}X_{j}w\left(  x_{1}\right)  -X_{i}X_{j}w\left(  x_{2}\right)
\right\vert \leqslant cd\left(  x_{1},x_{2}\right)  ^{\beta}\left\Vert
f\right\Vert _{C_{X}^{\beta}\left(  U\right)  }\text{\ \ }\forall\beta<\alpha
\]
for any $x_{1},x_{2}\in B\left(  \overline{x},\frac{R}{2}\right)  $, with some
constant $c$ also depending on $R$.

An analogous, but easier, inspection of each term $T_{j}f$ also shows that
\begin{equation}
\sup_{x\in B\left(  \overline{x},\frac{R}{2}\right)  }\left\vert X_{i}%
X_{j}w\left(  x\right)  \right\vert \leqslant c\left\Vert f\right\Vert
_{C_{X}^{\beta}\left(  U\right)  }. \label{upper bound ball}%
\end{equation}
By a covering argument this implies
\begin{equation}
\sup_{x\in U^{\prime}}\left\vert X_{i}X_{j}w\left(  x\right)  \right\vert
\leqslant c\left\Vert f\right\Vert _{C_{X}^{\beta}\left(  U\right)  }
\label{upper compact}%
\end{equation}
so that for each couple of points $x_{1},x_{2}\in U^{\prime}$ we can write
\[
\left\vert X_{i}X_{j}w\left(  x_{1}\right)  -X_{i}X_{j}w\left(  x_{2}\right)
\right\vert \leqslant cd\left(  x_{1},x_{2}\right)  ^{\beta}\left\Vert
f\right\Vert _{C_{X}^{\beta}\left(  U\right)  }%
\]
if $d\left(  x_{1},x_{2}\right)  <R_{0}/2,$ and, by (\ref{upper compact}),
\[
\left\vert X_{i}X_{j}w\left(  x_{1}\right)  -X_{i}X_{j}w\left(  x_{2}\right)
\right\vert \leqslant2\sup_{x\in U^{\prime}}\left\vert X_{i}X_{j}w\left(
x\right)  \right\vert \leqslant c\left(  \frac{d\left(  x_{1},x_{2}\right)
}{R_{0}}\right)  ^{\beta}\left\Vert f\right\Vert _{C_{X}^{\beta}\left(
U\right)  }%
\]
if $d\left(  x_{1},x_{2}\right)  \geqslant R_{0}/2$. Hence%
\[
\left\Vert X_{i}X_{j}w\right\Vert _{C_{X}^{\beta}\left(  U^{\prime}\right)
}\leqslant c\left\Vert f\right\Vert _{C_{X}^{\beta}\left(  U\right)  }.
\]
The norms $\left\Vert X_{i}w\right\Vert _{C_{X}^{\beta}\left(  U^{\prime
}\right)  }$, $i=1,\ldots n$, and $\left\Vert w\right\Vert _{C_{X}^{\beta
}\left(  U^{\prime}\right)  }$ can be more easily handled and
(\ref{schauder compact}) follows.
\end{proof}

\section{Appendix. Examples of nonsmooth H\"{o}rmander's operators satisfying
assumptions A or B}

\begin{example}
[Nonsmooth sublaplacian of Heisenberg type]In $\mathbb{R}^{3}\ni\left(
x,y,t\right)  $, let%
\begin{align*}
X_{1}  &  =\frac{\partial}{\partial x}+y\left(  1+\left\vert y\right\vert
\right)  \frac{\partial}{\partial t};\,X_{2}=\frac{\partial}{\partial
y}-x\left(  1+\left\vert x\right\vert \right)  \frac{\partial}{\partial
t};\,\,\\
\lbrack X_{1},X_{2}]  &  =-2\left(  1+\left\vert x\right\vert +\left\vert
y\right\vert \right)  \frac{\partial}{\partial t};\\
L  &  =X_{1}^{2}+X_{2}^{2}.
\end{align*}
The vector fields $X_{1},X_{2}$ are $C^{1,1}$ and satisfy H\"{o}rmander's
condition with $r=2$, hence Assumptions A hold. Replacing $\left\vert
x\right\vert ,\left\vert y\right\vert $ with $x\left\vert x\right\vert
,y\left\vert y\right\vert $ we find $C^{2,1}$ vector fields, satisfying
Assumptions B.
\end{example}

\begin{example}
[Nonsmooth operator of Kolmogorov type]In $\mathbb{R}^{3}\ni\left(
x,y,t\right)  ,$ with $\alpha\in(0,1]$, let:%
\begin{align*}
X_{1}  &  =\frac{\partial}{\partial x};\,X_{0}=x\left(  1+\left\vert
x\right\vert ^{\alpha}\right)  \frac{\partial}{\partial y}+\frac{\partial
}{\partial t};\,\,[X_{1},X_{0}]=\left(  1+\left(  \alpha+1\right)  \left\vert
x\right\vert ^{\alpha}\right)  \frac{\partial}{\partial y};\\
L  &  =X_{1}^{2}+X_{0}.
\end{align*}
$X_{1},X_{0}$ satisfy H\"{o}rmander's condition at weighted step $r=3$;
$X_{1}\in C^{2,\alpha}$, $X_{0}\in C^{1,\alpha}$, hence Assumptions A hold.
Replacing $\left\vert x\right\vert ^{\alpha}$ with $x\left\vert x\right\vert
^{\alpha}$, Assumptions B hold.
\end{example}

\begin{example}
[Nonsmooth operators of Grushin type with high step $r$]In $\mathbb{R}^{2}%
\ni\left(  x,y\right)  $, with $\alpha\in(0,1],$ $r\geqslant2$ positive
integer, let%
\begin{align*}
X_{1}  &  =\frac{\partial}{\partial x};\,X_{2}=x^{r-1}\left(  1+\left\vert
x\right\vert ^{\alpha}\right)  \frac{\partial}{\partial y};\\
L  &  =X_{1}^{2}+X_{2}^{2}.
\end{align*}
$X_{1},X_{2}$ satisfy H\"{o}rmander's condition at step $r$; $X_{2}\in
C^{r-1,\alpha}$, hence Assumptions A hold (if $r=2$ we need to take $\alpha
=1$). Replacing $\left\vert x\right\vert ^{\alpha}$ with $x\left\vert
x\right\vert ^{\alpha}$, Assumptions B hold.
\end{example}

\pagebreak

\bigskip

\noindent\textsc{Dipartimento di Matematica}

\noindent\textsc{Politecnico di Milano}

\noindent\textsc{Via Bonardi 9, 20133 Milano, ITALY}

\noindent\texttt{marco.bramanti@polimi.it}

\bigskip

\noindent\textsc{Dipartimento di Ingegneria}

\noindent\textsc{Universit\`{a} di Bergamo}

\noindent\textsc{Viale Marconi 5, 24044 Dalmine BG, ITALY}

\noindent\texttt{luca.brandolini@unibg.it}

\bigskip

\noindent\textsc{Dipartimento di Matematica}

\noindent\textsc{Universit\`{a} di Bologna}

\noindent\textsc{Piazza Porta S. Donato 5, 40126 Bologna BO, ITALY}

\noindent\texttt{manfredi@dm.unibo.it}

\bigskip

\noindent\textsc{Dipartimento di Ingegneria}

\noindent\textsc{Universit\`{a} di Bergamo}

\noindent\textsc{Viale Marconi 5, 24044 Dalmine BG, ITALY}

\noindent\texttt{marco.pedroni@unibg.it}

\end{document}